\documentclass{article}
\usepackage{graphicx} 
\usepackage{amssymb,amsmath,amsthm,physics}
\usepackage{CJK}
\usepackage[colorlinks,linkcolor=blue]{hyperref} 
\usepackage{fancyhdr}
\usepackage{tikz-cd}
\usepackage{multicol}
\usepackage{stmaryrd}
\usepackage{wasysym}
\usepackage{geometry}
\usepackage{mathrsfs}
\usepackage{natbib} 
\usepackage{graphicx} 
\usepackage{float} 
\usepackage{bm}
\usepackage{tikz-cd}
\usepackage{tabularx}
\allowdisplaybreaks[4]
\usepackage{todonotes}

\usepackage[font=footnotesize, labelfont=bf]{caption}
\usepackage[english]{babel}

\definecolor{Turk}{rgb}{0,0.7,0.4}

\theoremstyle{definition}
\newtheorem{theorem}{Theorem}[section]
\newtheorem{corollary}[theorem]{Corollary}
\newtheorem{lem}[theorem]{Lemma}
\newtheorem{prop}[theorem]{Proposition}

\newtheorem{defi}[theorem]{Definition}
\newtheorem{assump}[theorem]{Assumption}

\newtheorem{remark}[theorem]{Remark}

\usepackage{hyperref}

\numberwithin{equation}{section}

\newcommand{\sym}[1]{(#1)_{\mbox{\tiny{sym}}}}
\newcommand{\skw}[1]{(#1)_{\mbox{\tiny{skw}}}}
\newcommand{\dom}[1]{\mbox{dom}(#1)}
\newcommand{\divge}[1]{\mbox{div}(#1)}

\newcommand{\tU}{\tilde{\pmb U}}
\makeatletter 
\newcommand\superimpose[2]{%
	\ooalign{$\m@th#1\@firstoftwo#2$\cr
		\hidewidth$\m@th#1\@secondoftwo#2$\hidewidth}%
}
\makeatother 
\newcommand{\threedotsord}[0]{\mathpalette\superimpose{{\mathop:}{\cdot}}} 
\newcommand{\threedotsbin}[0]{\mathbin{\threedotsord}}     

\title{Energy-variational solutions for geodynamical two-phase flows{:}\\ 
{From} logarithmic to double-obstacle potentials by variational convergence}
\author{Fan Cheng\thanks{Freie Universit\"at Berlin, Institute of Mathematics, Arnimallee 9, 14195 Berlin, Germany, e-mail: fan.cheng@fu-berlin.de}, Robert Lasarzik\thanks{Weierstrass Institute for Applied Analysis and Stochastics, Anton-Wilhelm-Amo-Str. 39, 10117 Berlin, Germany, e-mail: robert.lasarzik@wias-berlin.de}, and Marita Thomas\thanks{Freie Universit\"at Berlin, Institute of Mathematics, Arnimallee 9, 14195 Berlin, Germany, e-mail: marita.thomas@fu-berlin.de}}
\date{\today}

\begin{document}

\maketitle
\paragraph{Abstract.} 
In [Nonlinear Anal. Real World Appl. 93(2026), art. id. 104652] {we} studied a Cahn--Hilliard two-phase model describing the flow of two viscoelastoplastic fluids in the framework of dissipative solutions using a logarithmic potential for the phase-field variable. This choice of potential has the effect that the fluid mixture cannot fully separate into two pure phases. The notion of dissipative solutions is based on a relative energy-dissipation inequality featuring a suitable regularity weight. In this way, this is a very weak solution concept. In the present work, we study the well-posedness of the geodynamical two-phase flow in the notion of energy-variational solutions. They feature an additional scalar energy variable that majorizes the system energy along solutions and they are further characterized by a variational inequality that combines an energy-dissipation estimate with the weak formulation of the system adding an error term that accounts for the mismatch between the energy variable and the system energy multiplied by a suitable regularity weight. We give a comparison of these two concepts. We further study different phase-field potentials for the geodynamical two-phase flow model. In particular, we address the variational limit from a potential with a logarithmic contribution to a double-obstacle potential, then also allowing for the emergence of pure phases. This study underlines that, thanks to its structure, the energy-variational solution is better suited for variational convergence methods than the dissipative solution. 

\paragraph{{\bf MSC2020:}}\leavevmode 
    {primary}
    35A15, 
    35A35, 
    35M86, 
    35Q86, 
    {76T99;} 
    {secondary 35D99,} 
    76M30. 
    
\clearpage
\section{Introduction}
{We study} a system of partial differential equations modeling the two-phase flow of an incompressible mixture of two viscoelastoplastic fluids with a diffuse interface, which arises in geodynamics. In a time interval $(0,T)$, where $T\in(0,\infty]$, and a bounded $C^{2}$-domain $\Omega\subseteq\mathbb{R}^3$, the system is {given as}
\begin{subequations}
	\label{sys:two phase}
	\begin{align}
		\partial_{t} (\rho v)+\divge{v\otimes(\rho v+J)}-\divge{\mathbb{T}}&=f-\varepsilon\divge{\nabla\varphi\otimes\nabla\varphi}&\text{ in }\Omega\times(0,T), 
        \label{sys: momentum}\\
		\divge{v}&=0&\text{ in }\Omega\times(0,T),
        \label{sys: incompressible}\\
        \mathbb{T}&=\eta(\varphi)S+2\nu(\varphi)\sym{\nabla v}-p\mathbb{I}&\text{ in }\Omega\times(0,T),
        \label{sys: cauchy stress tensor}\\
		\overset{\triangledown}{S}+\partial {P}(\varphi;S)-\gamma\Delta S&\ni\eta(\varphi)\sym{\nabla v}&\text{ in }\Omega\times(0,T), 
        \label{sys: stress}\\
		\partial_{t}\varphi+v\cdot\nabla\varphi&=\divge{m(\varphi)\nabla\mu}&\text{ in }\Omega\times(0,T), 
        \label{sys: first CHE}\\
		\mu&=\frac{1}{\varepsilon}W^{\prime}(\varphi)-\varepsilon\Delta\varphi&\text{ in }\Omega\times(0,T). 
        \label{sys: second CHE}		
	\end{align}
The system is complemented by {the boundary and initial conditions} 
	\label{BC and ID for TP}
	\begin{alignat}{3}
		v|_{\partial\Omega}&=0&&\mbox{ on }\partial\Omega\times(0,T),
        \\
        \gamma \vec{n}\cdot\nabla S|_{\partial\Omega}&=0&&\mbox{ on }\partial\Omega\times(0,T),
        \\\vec{n}\cdot\nabla\varphi|_{\partial\Omega}=\vec{n}\cdot\nabla\mu|_{\partial\Omega}&=0&&\mbox{ on }\partial\Omega\times(0,T),
        \\
		(v,S,\varphi)|_{t=0}&=(v_{0},S_{0},\varphi_{0})\quad&&\mbox{ in }\Omega,
	\end{alignat}
where $\vec{n}$ denotes the outward unit normal vector to the boundary $\partial\Omega$ of the domain $\Omega$.      
\end{subequations}

Equation~\eqref{sys: momentum} describes the momentum balance of the fluids, formulated in terms of the volume-averaged Eulerian velocity field $v:[0,T]\times\Omega\to\mathbb{R}^{3}$,
the mass density $\rho:[0,T]\times\Omega\to\mathbb{R}$, the volume-averaged mass flux $J:[0,T]\times\Omega\to\mathbb{R}^{3}$, the Cauchy stress tensor $\mathbb{T}:[0,T]\times\Omega\to\mathbb{R}^{3\times 3}$, and an external force $f:[0,T]\times\Omega\to\mathbb{R}^{3}$. The mass flux $J$, which is given as $J:=-\frac{\rho_{2}-\rho_{1}}{2}m(\varphi)\nabla\mu$, arises from the mismatch of the  mass densities $\rho_{1}$ and $\rho_{2}$ of the pure phases. Capillarity is modeled through a Korteweg-type stress contribution $\varepsilon \nabla\varphi\otimes\nabla\varphi$ with an interfacial thickness parameter $\varepsilon>0$. The velocity field is modeled as incompressible, so that the divergence-free constraint~\eqref{sys: incompressible} is imposed. The Cauchy stress tensor~\eqref{sys: cauchy stress tensor} for the viscoelastoplastic fluid consists of a spherical part and a deviatoric part. {Here} $p:[0,T]\times\Omega\to\mathbb{R}$ denotes {the pressure and} $\sym{\nabla v} := \frac{1}{2}(\nabla v + \nabla v^{\top})$
is the strain-rate tensor which is given by the symmetric part of the velocity gradient 
and measures the relative motion of fluid particles. Moreover, 
$S:[0,T]\times\Omega\to\mathbb{R}^{3\times 3}_{\mathrm{sym,Tr}}$ is an internal stress. The functions $\nu(\varphi)$ and $\eta(\varphi)$ in \eqref{sys: cauchy stress tensor} describe the phase-dependent viscosity and elastic modulus, respectively. 

The evolution of the internal stress $S$ is governed by a Maxwell-type constitutive law~\eqref{sys: stress}. The rate is expressed in terms of the Zaremba--Jaumann objective rate
\begin{equation}\label{Zaremba Jaumann rate}
  \overset{\triangledown}{S}
  := \partial_{t}S + v\cdot\nabla S + S\,\skw{\nabla v} - \skw{\nabla v}\,S,
\end{equation}
which is widely used in geophysical modeling, see, e.g., \cite{moresi_mantle_2002,gerya_robust_2007,herrendorfer_invariant_2018}. Here, the spin tensor
$\skw{\nabla v} := \frac{1}{2}(\nabla v - \nabla v^{\top})$ is the skew-symmetric part of the velocity gradient. Plastic effects are incorporated through an additional term $\partial P(\varphi;S)$ in the evolution law for $S$, where $P$ is a convex plastic potential density. A typical choice, motivated by geodynamical applications to the plastic deformation of lithospheric plates, is
\begin{equation}\label{def:P}
  P(\varphi;S):= \begin{cases}
    \dfrac{a(\varphi)}{2}\lvert S\rvert^{2}  & \text{if } \lvert S\rvert \leq \sigma_{\mathrm{yield}},\\
    \infty & \text{if } \lvert S\rvert > \sigma_{\mathrm{yield}}.
  \end{cases}
\end{equation}
Here, $a(\varphi)>0$ is {a} material parameter, and $\sigma_{\mathrm{yield}}$ denotes the corresponding yield stress, determining the onset of plastic flow (see, e.g., \cite{moresi_mantle_2002,gerya_robust_2007}). For fixed $\varphi$, the subdifferential of the convex function \(P(\varphi;\cdot)\) is given by
\begin{equation}
  \partial P(\varphi;S) := \left\{ \xi \in \mathbb{R}^{3\times 3}_{\mathrm{sym,Tr}} :
  \langle \xi,\tilde{S}-S\rangle_{\mathbb{R}^{3\times 3}} + P(\varphi;S) \leq P(\varphi;\tilde{S})
  \ \text{for all } \tilde{S}\in\mathbb{R}^{3\times3}_{\mathrm{sym,Tr}} \right\}\,,
\end{equation}
where $\mathbb{R}^{3\times3}_{\mathrm{sym,Tr}}$ denotes the set of all symmetric, trace-free, real-valued $3\times3$-matrices. 
On the level of functionals, the associated plastic potential is
\begin{equation}
\label{def-disspot}
  \mathcal{P}(\varphi;S) := \int_{\Omega} P(\varphi;S)\,\mathrm{d}x.
\end{equation}
In order to handle the {nonsmooth} evolution 
law of the stress tensor $S$ analytically, \eqref{sys: stress} also contains the term $-\gamma\Delta S$ with some parameter $\gamma>0$ as a regularization in the sense of stress diffusion, which ensures the equation to be parabolic{; see, e.g.,}~\cite{zbMATH07488956}. In certain applications such a stress-diffusion is also argued to be reasonable from the modeling perspective. For example,  in~\cite{stressdiffusion}, the presence of stress diffusion admits two interpretations: It can be viewed either as a consequence of a nonlocal energy storage mechanism, or as a consequence of a nonlocal entropy production mechanism.

The phase separation of the two incompressible fluids is described by a Cahn--Hilliard type system~\eqref{sys: first CHE}-\eqref{sys: second CHE} for the phase-field variable $\varphi:[0,T]\times\Omega\to\mathbb{R}$, which indicates the two phases. Physically, one expects, for $(t,x)\in[0,T]\times\Omega$,
\begin{equation}
  \varphi(t,x) 
  \begin{cases}
    =-1, & \text{pure fluid 1},\\
    \in(-1,1), & \text{mixture of fluid 1 and fluid 2},\\
    =1, & \text{pure fluid 2}.
  \end{cases}
\end{equation}
The {Cahn--Hilliard-type} equations~\eqref{sys: first CHE}-\eqref{sys: second CHE} couple $\varphi$ to the chemical potential $\mu:[0,T]\times\Omega\to\mathbb{R}$, and involve a phase-field potential $W:\mathbb{R}\to[0,\infty]$ together with a fixed small parameter $\varepsilon>0$ controlling the thickness of diffuse interfaces.

The  density of the mixture is assumed to depend affinely on $\varphi$ via
\begin{equation*}
  \rho(\varphi) := \frac{1-\varphi}{2}\rho_{1} + \frac{1+\varphi}{2}\rho_{2},
\end{equation*}
where $\rho_{i}>0$ is the constant density of the pure fluid $i$, $i=1,2$. With this choice, equation~\eqref{sys: first CHE} yields a modified continuity equation
\begin{equation}
  \partial_{t}\rho + \operatorname{div}(\rho v + J) = 0,
\end{equation}
so that the total mass balance is consistent with the diffusive interface description and the additional mass flux $J$.

In~\cite{arxiv:2509.25508}, we analyzed system \eqref{sys:two phase} under the assumptions of constant mobility $m(\varphi)=\mathrm{const}>0$ and a logarithmic potential. We proved global-in-time existence in the notion of dissipative solutions, which are characterized by a relative energy-dissipation estimate involving a suitable regularity weight. We established this existence result both for the system  regularized by stress-diffusion $\gamma>0$  and for the {nonregularized} system $\gamma=0$. 
However, the analysis in \cite{arxiv:2509.25508} relied crucially on the properties of the logarithmic potential.  In particular, 
by~\cite{zbMATH05191618}, the phase-field energy 
\begin{equation*}
    \int_{\Omega}\frac{\varepsilon}{2}|\nabla\varphi|^{2}+\varepsilon^{-1}W_{\mathrm{sg}}(\varphi)\dd{x}
\end{equation*}
is a proper, lower semicontinuous, convex functional, if $W_{\mathrm{sg}}$ is a logarithmic potential. Moreover, $W_{\mathrm{sg}}$ is smooth on its domain of definition $D$. Consequently, its subdifferential is a maximal monotone operator, single-valued on $D,$ and thus coincides with the term $-\varepsilon\Delta\varphi+\varepsilon^{-1}W_{\mathrm{sg}}^{\prime}(\varphi)$.
In addition, with the choice of logarithmic potential, the phase-field variable is constrained to take values in $(-1,1)$. These properties were exploited for the analysis in \cite{arxiv:2509.25508}.  

The total energy of system \eqref{sys:two phase} is of the form  
\begin{equation}
\begin{aligned}
	&\mathcal{E}_{\alpha}(v,S,\varphi):=\int_{\Omega}
    \frac{\rho(\varphi)}{2}\abs{v}^{2}
    +\frac{1}{2}\abs{S}^{2}
    +\frac{\varepsilon}{2}\abs{\nabla\varphi}^{2}
    +\varepsilon^{-1}W_{\mathrm{dw}}(\varphi)
    +\varepsilon^{-1}W_{\mathrm{sg},\alpha}(\varphi)\dd{x}\,,
\end{aligned}
\end{equation}
with $W_{\mathrm{dw}}$ being a smooth, regular phase-field potential, {e.g.,}\ a  double-well potential, and $W_{\mathrm{sg},\alpha}$ being a singular phase-field potential. In~\cite{arxiv:2509.25508}, we restricted the analysis on  singular potentials of logarithmic type, {i.e.,}\ $W_{\mathrm{sg},\alpha}$ was of the form $W_{\mathrm{sg},\alpha}(\varphi):=(1+\varphi)\ln(1+\varphi)+(1-\varphi)\ln(1-\varphi).$ This potential is smooth and singular in $\pm1.$ This has the effect that the pure phases, located in $\pm1,$ can neither be exceeded nor even reached apart from a set of zero {Lebesgue measure}. To overcome this drawback, we will start the analysis in this paper with a logarithmic potential of the form 
\begin{equation*}
    W_{\mathrm{sg},\alpha}(\varphi):=\alpha\left((1+\alpha+\varphi)\ln(1+\alpha+\varphi)+(1+\alpha-\varphi)\ln(1+\alpha-\varphi)\right) 
\end{equation*}
and we will study the variational limit $\alpha \searrow 0$. 
In the static setting,  one can show by means of $\Gamma$-convergence that 
\begin{equation}
    \mathcal{E}_{\alpha}\overset{\Gamma}{\longrightarrow}\mathcal{E}\text{ as }\alpha \searrow 0\,,
\end{equation}
where the limit energy $\mathcal{E}(v,S,\varphi)$ is given as 
\begin{equation}
    \mathcal{E}(v,S,\varphi):=
    \int_{\Omega}
    \frac{\rho(\varphi)}{2}\abs{v}^{2}
    +\frac{1}{2}\abs{S}^{2}
    +\frac{\varepsilon}{2}\abs{\nabla\varphi}^{2}
    +\varepsilon^{-1}W_{\mathrm{dw}}(\varphi)\dd{x}
    +\mathcal{I}_{[-1,1]}(\varphi)\,.
\end{equation}
Herein, the energy term $\mathcal{I}_{[-1,1]}$ is the double-obstacle potential, {i.e.,}\ 
\begin{equation}
    \mathcal{I}_{[-1,1]}(\varphi):=
    \begin{cases}
        0&\text{ }\varphi\in[-1,1]\text{ a.e.\ in }\Omega,\\
        \infty&\text{ otherwise.}
    \end{cases}
\end{equation}
{This singular, nonsmooth phase-field energy term now allows the pure phases, located in $\pm1$, to be reached, but not exceeded.} 

The single-phase system was introduced in~\cite{moresi_mantle_2002} and first analyzed in~\cite{zbMATH07488956}, where the  stress evolution equation~\eqref{sys: stress}, featuring the nonlinear Zaremba--Jaumann derivative~\eqref{Zaremba Jaumann rate} of $S$ and a {multivalued} term induced by the {nonsmooth} dissipation potential $\mathcal{P}$ as in \eqref{def-disspot}, was regularized by a stress diffusion term $-\gamma\Delta S$ with $\gamma>0$ for analytical reasons, {see}~\cite{zbMATH07488956,zbMATH07600529}. 
In order to remain closer to geoscientific models, {see, e.g.,}~\cite{moresi_mantle_2002}, we will avoid this regularization or rather consider the singular limit $\gamma\to0$ of the stress regularization. 
In~\cite{zbMATH07488956}, the authors proved the existence of generalized solutions for the single-phase model with stress diffusion, whereas, in~\cite{zbMATH07600529}, the weak-strong uniqueness of these solutions was considered as well as the limit of vanishing stress diffusion for the single-phase model. 
In our previous work~\cite{arxiv:2509.25508}, we introduced the two-phase model and proved the existence of dissipative solutions in the sense of Lions~{\cite[Section~4.4]{zbMATH00928933}}. 

In fact, the multiphase structure leads to additional analytical difficulties. We refer to~\cite{zbMATH05575962,zbMATH05960942,zbMATH06210388,zbMATH06286084} for related results. In particular, in~\cite{zbMATH06210388}, a coupled {Cahn--Hilliard--Navier--Stokes} system with logarithmic potentials was studied, while in~\cite{zbMATH06286084}, the double-obstacle potential was approximated by a logarithmic potential, and a limit passage yields the existence of weak solutions for the double-obstacle potential, under the assumption that the mobility $m(\varphi)$ is zero in each pure phase.

In this work, we introduce the concept of energy-variational solutions to the geodynamical two-phase system~\eqref{sys:two phase} under consideration. 
Note that the solution concept studied in~\cite{zbMATH07600529} was already coined energy-variational solution, but it was rather an extended version of dissipative solutions in the sense of Lions~\cite{zbMATH00928933}. 
The main improvement of energy-variational solutions studied in the present work in comparison to dissipative solutions is that we can pass to the limit in a suitable approximating sequence without applying some Gr\"onwall-type argument, which is needed in the case of dissipative solutions. This has the advantage that the arguments for any limit passage in this notion of solution mainly rely on lower semicontinuity, $\Gamma$-convergence, and graph-convergence. This will be exemplarily demonstrated for the variational limit $\alpha \searrow 0$ when passing from the logarithmic to a double-obstacle potential in the notion of energy-variational solution for the geodynamical two-phase flow system \eqref{sys:two phase}.  
The general idea about the notion of energy-variational solution is to introduce an auxiliary variable $E:[0,T]\to\mathbb{R}$, which is an upper bound of the system energy and to introduce an error term $ \mathcal{K}[ \mathcal{E}(v,S,\varphi) - E] $ to the formulation in order to make it stable under weak convergence. Herein $\mathcal{K}$ denotes a suitable regularity weight.  Let us demonstrate this idea for the relevant terms of the considered system. The main problem in passing to the limit with a suitable approximating sequence are the last two terms in the Zaremba--Jaumann derivative~\eqref{Zaremba Jaumann rate}. With only weak convergence of the sequences $(\nabla v_{n})_n$ and $(S_{n})_n$ in $L^2(\Omega)$ at our disposal, such a product term must not converge. The main idea is to convexify this term by adding {``enough''} energy to the system and using the occurring dissipation. With the choice of the regularity weight $\mathcal{K}(\tilde S) := C\| \tilde S \|^2_{L^\infty}$ for any suitable test function $\tilde S$, we will show that the associated term 
$$
S \mapsto \int_\Omega \Big(2 {\nu}_1 |(\nabla v)_{\mathrm{sym}} |^2+ (S\skw{\nabla v} - \skw{\nabla v}S):\tilde{S}\Big)\, \dd x  + \mathcal{K}(\tilde S)  \frac{1}{2}\| S\|_{L^2}^2 
$$
is {lower semicontinuous} with respect to the available weak convergences. For a suitable approximating sequence of solutions $(v_n,S_n,\varphi_n)_n$ with $(v_n,S_n,\varphi_n)\rightharpoonup(v,S,\varphi)$ in a suitable topology, for each $n\in\mathbb{N}$, the associated energy 
$\mathcal{E}(v_n(t),S_n(t),\varphi_n(t))$ is a 
function of bounded variation in time. This can be observed from the fact that solutions satisfy for all $t\in[0,T]$ an energy-dissipation estimate of the form 
\begin{equation}\label{smooth energy balance}
\begin{aligned}
   &\mathcal{E}(v_n(t),S_n(t),\varphi_n(t))
   +\int_{0}^{t}\int_{\Omega}
   2\nu(\varphi_n)\abs{\sym{\nabla v_n}}^{2}
   +\gamma\abs{\nabla S_n}^{2}
	+2 P(\varphi_n;S_n)
   +m(\varphi_n)\abs{\nabla \mu_n}^{2} \,\dd{x}\dd{\tau}
   \\
   &\leq\mathcal{E}(v_n(0),S_n(0),\varphi_n(0))
   +\int_{0}^{t}\langle f, v_{n}\rangle_{H^{1}}\dd{\tau}\,,
 \end{aligned}
 \end{equation}
which also can be shown to provide a uniform bound on the sequence of energy functions. Hence, for the uniformly bounded sequence of monotone real functions, by Helly’s selection principle, we can extract a subsequence that converges pointwise in time to an auxiliary energy function $E:[0,T]\to\mathbb{R}$. The defect $ E(t)-\mathcal{E}(v(t),S(t),\varphi(t))\geq 0$ quantifies the difference between the system  energy, which is only lower semicontinuous, and the actual limit pointwise in time of the energy function $E$. This idea allows us to pass to the limit in the notion of energy-variational solutions.
This setting is further explained in Section~\ref{intro:general pde setting}.  
Energy-variational solutions have been introduced for incompressible fluid dynamics in~\cite{lasarzik}, in~\cite{zbMATH07834723} for general hyperbolic conservation laws, and in~\cite{ALREVS} for general damped Hamiltonian systems with applications to viscoelastic fluid dynamics.
The methodology of energy-variational solutions developed here is sufficiently robust to admit extensions in several directions, as already partially exploited in~\cite{ALREVS}. 
The consideration of other 
material derivatives, like lower or upper convected derivative, $\partial_{t}S + v\cdot\nabla S - S(\nabla v) + (\nabla v)S$ {or} $\partial_{t}S + v\cdot\nabla S + S(\nabla v) - (\nabla v)S$, as used in an  
{Oldroyd--B-type} model, would lead to other coupling terms that could probably be treated similarly. Nevertheless, in the case of the Oldroyd-B model, the energy of the system changes to $F(S)-\ln(\det(S))$, which would change our analysis dramatically, since it relies heavily on the underlying energetic structure. Cases involving {nonquadratic}, and especially {nonsmooth}, energies for the viscoelastic model have been considered in~\cite{ALREVS} and~\cite{eiter2025solutionconceptsmodelviscoelastoplasticity}, 
and these techniques could probably also be extended to the model considered here.  

For several systems the equivalence of energy-variational solutions to measure-valued solutions could be shown, for example, for the isentropic Euler system~\cite{zbMATH07834723}, two-phase Navier--Stokes equations, polyconvex elastodynamics~\cite{envarMeas} the Ericksen--Leslie system for liquid crystals~\cite{zbMATH07680741,envarMeas}. In the case of this system, the limit of a suitable structure-inheriting scheme could be shown to converge to an energy-variational solution with a different regularity weight than the one for which the equivalence to measure-valued solutions could be shown~\cite{zbMATH07680741}. This shows the flexibility of this concept for the identification of limits. 

In the current work, we use the novel concept of energy-variational solutions for the considered system in order to handle a {nonconstant} mobility as well as the variational limit from a logarithmic to a double-obstacle potential in the Cahn--Hilliard equation. Here we show the advantage of the concept of energy-variational solutions, as it is a finer notion than the dissipative solutions. Indeed, we show that an energy-variational solution is always a dissipative solution. Moreover, energy-variational solutions fulfill the so-called {semiflow} property, which is a natural property for a solution and it also allows to use some improved selection criteria as in~\cite{lasarzik}. 
In this way, we develop a more robust and  finer solution concept for 
the considered geophysical system in order to overcome the restriction to constant mobilities and logarithmic potentials imposed in~\cite{arxiv:2509.25508}. 
This results in a system for geodynamical two-phase flows admitting  the 
double-obstacle potential that is {nonsmooth}, preserves the uniform bound on phase-field variable $\varphi,$ but allows it to attain the pure phases, such that $\varphi\in[-1,1]$ a.e.\ in $\Omega\times[0,T)$. 

The present paper is structured as follows: After introducing general {notation} and preliminaries in Section \ref{section preliminary}, we give the generalized formulation of energy-variational solutions in Section~\ref{intro:general pde setting}, provide general  results for this notion of solution and compare it with dissipative solutions. Subsequently, in Section \ref{Sec:ENVarSol-Geo}, we introduce energy-variational solutions specifically for system \eqref{sys:two phase} and discuss, in Section \ref{Sec:DissipSol-Geo}, their connection to the dissipative solutions studied in \cite{arxiv:2509.25508}. 
In Section \ref{Sec:singular}, we prove the existence of energy-variational solutions for system \eqref{sys:two phase} with the logarithmic phase-field potential. Our proof is based on a time-discrete scheme applied to system \eqref{sys:two phase}, {see} Section \ref{time discretization}, and, by compactness arguments, the existence of energy-variational solutions is first deduced in presence of stress diffusion $\gamma>0$ in {Theorem} \ref{exitence evs regulairzed system}. Invoking the arguments of  \cite[Section 5]{arxiv:2509.25508} we pass to the limit $\gamma\to0$ in {Theorem} \ref{exitence evs non-regulairzed system} and obtain the existence of energy-variational solutions for the system without stress diffusion. Finally, in Section \ref{sec:5}, we perform the variational limit from the logarithmic phase-field potential to a double-obstacle potential in the notion of energy-variational solution for system \eqref{sys:two phase} in the {nonregularized} setting $\gamma=0$.   

\section{General {notation}, preliminaries and assumptions}\label{section preliminary}
In this section, we fix the notation that will be used throughout this work and recall some useful results that will be applied for our analysis. 
\subsection{General {notation}}
Let $a=(a_j)_{j=1}^{3},b=(b_j)_{j=1}^{3}\in\mathbb{R}^3$ be two {vectors. Their} inner product is written as $a\cdot b:=a_{j}b_{j}$, and their tensor product is written as $a\otimes b=(a_{j}b_{k})_{j,k=1}^{3}$.  Similarly, let $A=(A_{jk})_{j,k=1}^{3},B=(B_{jk})_{j,k=1}^{3}\in\mathbb{R}^{3\times3}$ be two second-order tensors, the tensor inner product is written as $A:B=A_{jk}B_{jk}$. Besides, for two third-order tensors $C=(C_{jkl})_{j,k,l=1}^{3},D=(D_{jkl})_{j,k,l=1}^{3}\in\mathbb{R}^{3\times3\times3}$, we denote the inner product by $C\threedotsbin D=C_{jkl}D_{jkl}$. The tensor product between a second-order tensor $A\in\mathbb{R}^{3\times3}$ and a vector $a\in\mathbb{R}^{3}$ gives a third-order tensor and it is defined as $(A\otimes a)_{jkl}=(A_{jk}a_{l})_{j,k,l=1}^{3}$. We write the transpose and trace of a matrix $A\in\mathbb{R}^{3\times3}$ in the usual way that $A^{\top}$ and $\mbox{Tr}A$. Moreover, we set the space of the symmetric and trace-free second order tensors as
\begin{align}
    \mathbb{R}_{\mathrm{sym,Tr}}^{3\times3}
    :=\left\{A\in\mathbb{R}^{3\times3}:A=A^{\top},\mbox{Tr}A=0\right\}.
\end{align} 

The point $(x,t)\in\Omega\times(0,T)$ is defined by the spatial variable $x\in\Omega$ and the time variable $t\in(0,T)$. Thus, we write the partial time derivative and partial spatial derivative of a (sufficiently regular) function $u:\Omega\times{(0,T)}\to\mathbb{R}$ as $\partial_{t} u$ and $\partial_{x_{i}} u$, where $i=1,2,3$. Moreover, $\nabla u$ and $\Delta u$ denote the gradient and Laplace operator of $u$ with respect to the spatial variable. The symmetrized and skew-symmetrized part of $\nabla v$ of a vector field $v:\Omega\to\mathbb{R}^{3}$ is given by 
\begin{align}
	\sym{\nabla v}:=\frac{1}{2}\left(\nabla v+\nabla v^{\top}\right)
	\mbox{ and }
	\skw{\nabla v}:=\frac{1}{2}\left(\nabla v-\nabla v^{\top}\right).
\end{align}
We also write $\divge{v}=\sum_{i=1}^3\partial_{x_{i}} v_{i}=\partial_{x_{i}} v_{i}$ using the Einstein summation convention as the divergence of vector field $v$. Similarly, for a second-order tensor $S$, the divergence is defined by $\divge{S}=\partial_{x_{k}} S_{jk}$.

~\\
\textbf{Function spaces.} Let $X$ be a Banach space with norm $\lVert\cdot\rVert_{X}$ and dual space $X^{\prime}$. The same notation is used also for $X^{3}$ and $X^{3\times3}$. When the dimension is clear, we simply write $X$ instead of $X^{3}$ or $X^{3\times3}$. We use $\langle x^{\prime},x\rangle_{X}$ to denote the duality pairing of $x^{\prime}\in X^{\prime}$ and $x\in X$. 

The space $C^{\infty}(\Omega)$ denotes the class of smooth functions in $\Omega$ and the space $C_{0}^{\infty}(\Omega)$ denotes the class of smooth functions with compact support in $\Omega$. The Lebesgue spaces and Sobolev spaces are denoted as $L^p(\Omega)$ and $W^{k,p}(\Omega)$ for $p\in[1,\infty]$ and $k\in\mathbb{N}$, in particular, for $p=2$, we write $W^{k,2}(\Omega)=H^{k}(\Omega)$. Moreover, we write $H_{0}^{1}(\Omega)$ as the space of functions in $H^1(\Omega)$ whose boundary value is zero in the trace sense and $H^{-1}(\Omega):=(H_{0}^{1}(\Omega))^{\prime}$ is the dual space.

Now let $(0,T)\subseteq\mathbb{R}$ be an interval. The space $C^{0}(0,T;X)$ consists of the class of continuous functions in time with values in the Banach space $X$. For $p\in[1,\infty]$, the corresponding Lebesgue-Bochner spaces are denoted by $L^{p}(0,T;X)$. Moreover, we write $W^{1,p}(I;X):=\left\{u\in L^{p}(0,T;X):\partial_{t}u\in L^{p}(0,T;X)\right\}$ and $H^{1}(0,T;X)=W^{1,2}(0,T;X)$. The local Lebesgue-Bochner spaces $L_{\mathrm{loc}}^{p}(0,T;{X})$ and $H_{\mathrm{loc}}^{1}(0,T;{X})$ consist of the class of functions in $L^{p}(I;X)$ and $H^{1}(I;X)$ for every compact subinterval $I\subseteq (0,T)${,} respectively. 
Moreover, for a function $u$ defined on $(0,T)$ and for $s\in(0,T)$, we denote by $u(s_{\pm})$ the left- and right-hand side limits of $u$ at $s$, respectively, {i.e.,} 
\begin{equation}
\label{lr-limit}
    u(s_{+}):=\lim_{t\searrow s}u(t)\,,\quad 
    u(s_{-}):=\lim_{t\nearrow s}u(t).
\end{equation}

For any $u\in L^{1}(\Omega)$
\begin{align}\label{mean value in Omega}
	u_{\Omega}:=\frac{1}{|\Omega|}\int_{\Omega}u\dd{x}
\end{align} 
is the mean value of $u$ in $\Omega$.

~\\
\textbf{Solenoidal vector fields and symmetric deviatoric fields.} We introduce function spaces for solenoidal (divergence-free) vector fields and symmetric, deviatoric (trace-free) fields. The corresponding classes of smooth functions on $\Omega$ are given by
\begin{subequations}
\begin{align}
	C_{0,\mathrm{div}}^{\infty}(\Omega)&:=\left\{\varphi\in C_{0}^{\infty}(\Omega)^3:\divge{\varphi}=0\text{ in }\Omega\right\},
    \\
    C_{\mathrm{sym,Tr}}^{\infty}(\bar{\Omega})&:=\left\{\psi\in C^{\infty}(\bar{\Omega})^{3\times 3}:\psi=\psi^{\top},\Tr(\psi)=0\text{ in }\Omega\right\}.
\end{align}
We further denote the spaces of smooth, time-dependent solenoidal vector fields and symmetric, deviatoric fields as
\begin{align}
	C_{0,\mathrm{div}}^{\infty}(\Omega\times I)&:=\left\{\Phi\in C_{0}^{\infty}(\Omega\times I)^3:\divge{\Phi}=0\text{ in }\Omega\times I\right\},
    \\
	C_{0,\mathrm{sym,Tr}}^{\infty}(\Omega\times I)&:=\left\{\Psi\in C_{0}^{\infty}(\Omega\times I)^{3\times3}:\Psi=\Psi^{\top},\Tr(\Psi)=0\text{ in }\Omega\times I\right\},
\end{align}
where $I\subseteq[0,\infty)$ is an interval. The corresponding Lebesgue spaces of integrable functions on $\Omega$ are defined by
\begin{align}
	L_{\mathrm{div}}^{2}(\Omega)&:=\left\{v\in L^{2}(\Omega)^{3}:\divge{v} =0\text{ in }\Omega\right\},\\
	L_{\mathrm{sym,Tr}}^{2}(\Omega)&:=\left\{S\in L^{2}(\Omega)^{3\times 3}:S=S^{\top},\Tr(S)=0\text{ in }\Omega\right\}.
\end{align}
The Sobolev spaces obtained as the closure of $C_{0,\mathrm{div}}^{\infty}(\Omega)$ and ${C_{\mathrm{sym,Tr}}^{\infty}(\bar{\Omega})}$ with respect to the $H^{1}(\Omega)$-norm are denoted by
\begin{align}
	H_{0,\mathrm{div}}^{1}(\Omega)&:=\left\{v\in H_{0}^{1}(\Omega)^{3}: \divge{v}=0\text{ in }\Omega\right\},
    \\
    H_{\mathrm{sym,Tr}}^{1}(\Omega)&:=\left\{S\in H^{1}(\Omega)^{3\times 3}:S=S^{\top},\Tr(S)=0\text{ in }\Omega\right\}.
\end{align}
\end{subequations}
Notice that all the boundary conditions are identified in the trace sense. \\
\subsection{General assumptions and further {notation}}
In the following, we collect and discuss the mathematical assumptions on the domain, the given data, and the material parameters.
    \begin{assump}[on the domain]\label{ASM:domain}
    We assume that $\Omega\subseteq\mathbb{R}^3$ is a bounded domain with $C^{2}$-boundary $\partial\Omega$ and write $\vec{n}$ for the outward unit normal vector.
\end{assump}

Also in view of~\eqref{smooth energy balance}, we make the following hypothesis for the {nonsmooth} dissipation potential $\mathcal{P}$ in~\eqref{sys: stress}:
\begin{assump}[on the plastic potential]\label{ASM:Dissipation potential}
	For the plastic potential 
    \begin{equation}\label{integral form of dissipation potential}
    \begin{aligned}
        \mathcal{P}:L^{2}(\Omega)\times L_{\mathrm{sym,Tr}}^{2}(\Omega)&\to[0,\infty]\,\\
        (\varphi;S)&\mapsto\int_{\Omega}P(x,\varphi(x),S(x))\dd{x},
    \end{aligned}
    \end{equation}
    we make the following assumptions: The density $P:\Omega\times\mathbb{R}\times \mathbb{R}_{\mathrm{sym,Tr}}^{3\times 3}\to[0,\infty]$ is proper and measurable with ${P}(x,\varphi,0)=0$ for all $x \in \Omega$ and $\varphi\in\mathbb{R}$. Moreover, 
 \begin{itemize}
    \item for all $x\in\Omega$, the mapping $(y,z)\mapsto P(x,y,z)$ is lower semicontinuous,
     \item for all $(x,y) \in \Omega\times\mathbb{R}$, the mapping $z \mapsto {P}(x,y,z) $ is convex,
     \item for all $(x,z)\in \Omega\times\mathbb{R}_{\mathrm{sym,Tr}}^{3\times 3}$, the mapping $y \mapsto {P}(x,y,z) $ is continuous.
 \end{itemize}
 Besides, for fixed $\varphi\in L^{2}(\Omega),$ the convex partial subdifferential of $\mathcal{P}(\varphi;\cdot)$ in the point $S\in L_{\mathrm{sym,Tr}}^{2}(\Omega)$ is given by
\begin{equation}\label{definition of subdifferential}
	\partial\mathcal{P}(\varphi;S):=\left\{ \xi\in (L_{\mathrm{sym,Tr}}^{2}(\Omega))^{\prime}:\langle \xi,\tilde{S}-S\rangle_{L^{2}(\Omega)}+\mathcal{P}(\varphi;S)\leq\mathcal{P}(\varphi;\tilde{S})\mbox{ for all }\tilde{S}\in L_{\mathrm{sym,Tr}}^{2}(\Omega) \right\}.
\end{equation}
Notice that, by definition, $\partial\mathcal{P}(\varphi;S)=\emptyset$ if $\mathcal{P}(\varphi;S)=\infty$.

\end{assump}

Furthermore, we make the following hypotheses for the initial data and the external loading:
\begin{assump}[on the given data]\label{ASM:initial data}
	Assume that $v_{0}\in L_{\mathrm{div}}^{2}(\Omega)$, $S_{0}\in L_{\mathrm{sym,Tr}}^{2}(\Omega)$ and $f\in L_{\mathrm{loc}}^{2}([0,T);H^{-1}(\Omega)^{3})$. Moreover, assume that $\varphi_{0}\in H^{1}(\Omega)$ with $\abs{\varphi_{0}}\leq1$ almost everywhere in $\Omega$ and
	\begin{align*}
		\frac{1}{\abs{\Omega}}\int_{\Omega}\varphi_{0}\dd{x}\in(-1,1).
	\end{align*}
\end{assump}

In order to guarantee the existence of weak solutions, we make the following assumptions.

\begin{assump}[on the material parameters]\label{ASM:material paramaters}
    The dependence of the material parameters on the composition of the mixture, i.e., on $\varphi$ is assumed to be as follows:\\
	(1) The dependence of the mass density $\rho$ on the phase-field variable $\varphi$ is given by 
	\begin{align}\label{assumption on density}
		\rho(\varphi)=\frac{\rho_{1}+\rho_{2}}{2}+\frac{\rho_{2}-\rho_{1}}{2}\varphi{,}
	\end{align}
	where the constant $\rho_{i}>0$ is the mass density of fluid $i$, $i=1,2$.
    \\
	(2) For the viscosity parameter $\nu$ and the elastic modulus $\eta$, {we assume}
    $\nu\in C^{0}(\mathbb{R})$, $\eta\in C^{1}(\mathbb{R})$ and 
	\begin{align}\label{assumption on coefficients}
		0<\nu_{1}\leq\nu(\varphi)\leq \nu_{2}\,, \quad 0\leq\eta_{1}\leq\eta(\varphi)\leq\eta_{2}\,,\;\text{ and }\abs{\eta^{\prime}(\varphi)}\leq C\text{ for all }\varphi\in\mathbb{R}{,}
	\end{align}
	for some positive constants $C,\eta_{1},\eta_{2},\nu_{1},\nu_{2}$. Herein, the constants $\eta_{i},\nu_{i}$ can be viewed as the constants associated with the pure fluid $i$, $i=1,2$.
\end{assump}
\begin{assump}[on the mobility]\label{ASM:mobility}
    For the mobility, we assume that $m\in C^{1}(\mathbb{R})$ and 
    \begin{equation}
        0<m_{1}\leq m(\varphi)\leq m_{2}\text{ for all }\varphi\in\mathbb{R}{,}
    \end{equation}
    for some positive constants $m_{1}$, $m_{2}$. Herein, again, the constant $m_{i}$ can be viewed as the constant associated with the pure fluid $i$, $i=1,2$.
\end{assump}
\begin{assump}[on the singular potentials] 
\label{ASM:singular potential} 
We consider a singular potential 
\begin{equation}
\label{def-Wsing}
W_\mathrm{sing}\in\{W_\mathrm{sg},I_{[-1,1]}\}{,}
\end{equation}
{which} is either a smooth, singular potential $W_\mathrm{sg}$ of logarithmic type, or a {nonsmooth} potential $I_{[-1,1]}$ of obstacle-type. Hence:  
\begin{enumerate}
\item     For the logarithmic potential $W_{\mathrm{sg}}$, we assume that $W_{\mathrm{sg}}\in C([-1,1])\cap C^{2}(-1,1)$ such that
	\begin{align}\label{assumption on singular potential}
		\lim\limits_{\varphi\to-1}W_{\mathrm{sg}}^{\prime}(\varphi)=-\infty\,,\;\lim\limits_{\varphi\to1}W_{\mathrm{sg}}^{\prime}(\varphi)=\infty\,,\;W_{\mathrm{sg}}^{\prime\prime}(\varphi)\geq-\kappa\;\mbox{ for some }\kappa\geq0.
	\end{align} 
	Moreover, we extend $W_{\mathrm{sg}}(\varphi)$ to $+\infty$ for $\varphi\in\mathbb{R}\backslash[-1,1]$. Without loss of generality, we also assume that $W_{\mathrm{sg}}(\varphi)\geq0$ for all $\varphi\in[-1,1]$.
\item The {nonsmooth} potential $I_{[-1,1]}(\varphi)$ is given by the indicator function of  the interval $[-1,1]$, i.e.,
    \begin{equation}
    \label{def-indicator}
        I_{[-1,1]}(\varphi):=
        \begin{cases}
            0,&\varphi\in[-1,1],\\
            \infty,&\varphi\notin[-1,1].
        \end{cases}
    \end{equation}
\end{enumerate}
\end{assump}
Observe that, for all $\varphi${,} the subdifferential of $W_\mathrm{sing}$ is thus given by 
\begin{equation}
\label{def-subdiff-Wsing}
\begin{split}
\partial W_\mathrm{sing}(\varphi)&\in\big\{\partial W_\mathrm{sg}(\varphi),\partial I_{[-1,1]}(\varphi)\big\}\quad\text{ with }
\\[1ex]
 \partial W_\mathrm{sg}(\varphi)&=\left\{
\begin{array}{cl}
\emptyset&\text{if }\varphi\in\mathbb{R}\backslash(-1,1)\,,\\
\{W_\mathrm{sg}'(\varphi)\}&\text{if }\varphi\in(-1,1)\,,
\end{array}
 \right.\\[1ex]
 \partial I_{[-1,1]}(\varphi)&=
    \left\{\begin{array}{cl}
        [0,+\infty)&\text{ if }\varphi=1,\\
        \{0\}&\text{ if }\varphi\in(-1,1),\\
        (-\infty,0]&\text{ if }\varphi=-1,\\
        \emptyset&\text{ otherwise}.
    \end{array}\right.
    \end{split}
    \end{equation}
\begin{assump}[on the double-well and double-obstacle potentials]\label{ASM:Double well potential}
    For a double-well potential $W_{\mathrm{dw}}:\mathbb{R}\to[0,\infty)$, we assume that $W_{\mathrm{dw}}\in C^{2}(\mathbb{R})$, $W_{\mathrm{dw}}(\pm1)=0$ and $W_{\mathrm{dw}}(\varphi)>0$ for all $\varphi\neq\pm1$.
\end{assump}
Finally, we define the phase-field potential $W$ as the sum of the double-well potential $W_{\mathrm{dw}}$ and one of the singular potentials $W_\mathrm{sing}\in\{W_{\mathrm{sg}},I_{[-1,1]}\}$ from \eqref{def-Wsing}, i.e., 
\begin{equation}
\label{def-W}
    W(\varphi):=W_{\mathrm{dw}}(\varphi)+W_{\mathrm{sing}}(\varphi)\,.
\end{equation}

\begin{subequations}\label{energies functionals}
\noindent
\textbf{Energy functionals.} 
We define the kinetic energy, the elastic energy and the phase-field energy {as} 
\begin{align}
    &\mathcal{E}_{\mathrm{kin}}:L^{\infty}(\Omega)\times L_{\mathrm{div}}^{2}(\Omega)\to[0,\infty],\,(\varphi,v)\mapsto\int_{\Omega}\rho(\varphi)\frac{\abs{v}^2}{2}\dd{x},
    \\
    &\mathcal{E}_{\mathrm{el}}:L_{\mathrm{sym,Tr}}^{2}(\Omega)\to[0,\infty],\,
    S\mapsto\int_{\Omega}\frac{\abs{S}^2}{2}\dd{x},
    \\
    &\mathcal{E}_{\mathrm{pf}}:H^{1}(\Omega)\to[0,\infty],\,
	\varphi\mapsto\int_{\Omega}\varepsilon\frac{\abs{\nabla\varphi}^2}{2}+\varepsilon^{-1}W(\varphi)\dd{x}\,,
\end{align}
with the phase-field potential $W$ as in \eqref{def-W}. 
Hence, the total energy of this system is given by
\begin{equation}
	\mathcal{E}(v,S,\varphi)
	:=\mathcal{E}_{\mathrm{kin}}(\varphi,v)+\mathcal{E}_{\mathrm{el}}(S)+\mathcal{E}_{\mathrm{pf}}(\varphi).
	\label{total energy two phase system}
\end{equation}
In addition, we introduce the space of test functions as
\begin{equation}\label{test function space}
    \mathfrak{T}:= C_{0,\mathrm{div}}^{\infty}(\Omega\times[0,T))\times C_{0,\mathrm{sym,Tr}}^{\infty}(\Omega\times[0,T))\times C_{0}^{\infty}(\Omega\times[0,T)) \times C_{0}^{\infty}(\Omega\times[0,T)).
\end{equation}
\end{subequations}

\begin{lem}\cite[][Lemma 2.5]{zbMATH07834723}\label{Lem: equivalence weak form}
    Let $f\in L_{\mathrm{loc}}^{1}(0,T)$, $g\in L_{\mathrm{loc}}^{\infty}(0,T)$ and $g_{0}\in\mathbb{R}$. Then the following two statements are equivalent:\\
    1. The inequality 
    \begin{equation}
        -\int_{0}^{T}\phi^{\prime}(\tau)g(\tau)\dd{\tau}
        +\int_{0}^{T}\phi(\tau)f(\tau)\dd{\tau}\leq \phi(0)g_{0}
    \end{equation}
    holds for all $\phi\in C_{0}^{\infty}([0,T))$ with $\phi\geq0$.\\
    2. The inequality
    \begin{equation}
        g(t)-g(s)+\int_{s}^{t}f(\tau)\dd{\tau}\leq0
    \end{equation}
    holds for a.e. $s<t\in[0,T)$, including $s=0$ if we replace $g(0)$ with $g_{0}$.

    If one of these {conditions} is satisfied, then $g$ can be identified with a function in $BV([0,T])$ such that 
    \begin{equation}
        g(t_{+})-g(s_{-})+\int_{s}^{t}f(\tau)\dd{\tau}\leq0
    \end{equation}
    for all $s\leq t\in[0,T)$, where we set $g(0_{-}):=g_{0}$. {In particular,} $g(0_{+})\leq g_{0}${,} and $g(t_{+})\leq g(t_{-})$ for all $t\in(0,T)$.
\end{lem}

\section{Energy-variational solutions}
\label{sec:3}
In the following we give a general introduction to the concept of energy-variational solutions and compare them with the notion of dissipative solutions in Section \ref{intro:general pde setting}. Subsequently, we apply the concept to the geodynamical two-phase system \eqref{sys:two phase} in Section \ref{Sec:ENVarSol-Geo} and discuss the connection to dissipative solutions in Section \ref{Sec:DissipSol-Geo}. 
\subsection{General concept of energy-variational solution}\label{intro:general pde setting}
\begin{subequations}
To better explain the concept of energy-variational solutions, we follow the general approach proposed in~\cite{ALREVS}. To this end, we consider two reflexive Banach spaces $\mathbb{V}$ and $\mathbb{Y}$ with dual spaces $\mathbb{V}^{\prime}$ and $\mathbb{Y}^{\prime}$ such that $\mathbb{Y}\subseteq\mathbb{V}\subseteq\mathbb{V}^{\prime}\subseteq\mathbb{Y}^{\prime}$ and a general evolutionary PDE on the time interval $(0,T)$ of the form
\begin{equation}\label{intro:Eq}
\partial_ t \pmb U(t) + A(\pmb U(t)) = \pmb 0 \quad \text{ in }\mathbb{Y}^\prime \; \text{ with }\quad\pmb U(0)=\pmb U_0 \quad \text{in }\mathbb V.\,
\end{equation}
Here, $A:\mathbb{V}\to\mathbb{Y}^{\prime}$ denotes a differential operator and $U_{0}\in\mathbb{V}$ the initial datum. Let $\mathcal{E}:V\to[0,\infty]$ be the energy functional and $\Psi:\mathbb{V}\to[0,\infty]$ 
the dissipation functional associated with~\eqref{intro:Eq}. For any initial value $U_{0}\in\mathbb{V}$, a sufficiently regular solution $\pmb U\in\mathbb{V}$ for all $t\in[0,T]$ of~\eqref{intro:Eq} formally fulfills the energy-dissipation mechanism
\begin{equation}\label{intro:ede}
    \mathcal{E}(\pmb U)\Big|_{s}^{t}+\int_{s}^{t}\Psi(\pmb U)\dd{\tau}\leq0
\end{equation}
for all $s<t\in[0,T]$, where the dissipation is formally defined as
\begin{equation}
    \Psi(\pmb U):=\langle A(\pmb U), \mathrm{D}\mathcal{E}(\pmb U)\rangle_{\mathbb{Y}}.
\end{equation}
Moreover, $\pmb{U}$ also satisfies {the weak formulation}
\begin{equation}\label{intro:wf}
    \langle \pmb U,\tU\rangle_{\mathbb{Y}}\Big|_{s}^{t}
    +\int_{s}^{t}
    -\langle \pmb U,\partial_{t}\tU\rangle_{\mathbb{Y}}
    +\langle A(\pmb U),\tU \rangle_{\mathbb{Y}}\dd{\tau}=0
\end{equation}
{for all} $\tU\in C_{0}^{\infty}([0,T);\mathbb{Y})$. Therefore, by summing~\eqref{intro:ede} and~\eqref{intro:wf}, we have
\begin{equation}\label{intro:ede+wf}
    \left(\mathcal{E}(\pmb U)-\langle \pmb U,\tU\rangle_{\mathbb{Y}}\right)\Big|_{s}^{t}
        +\int_{s}^{t}
        \Psi(\pmb U)
        +\langle \pmb U,\partial_{t}\tU\rangle_{\mathbb{Y}}
        -\langle A(\pmb U),\tU \rangle_{\mathbb{Y}}\dd{\tau}
        \leq
        0
\end{equation}
for a.e. $s<t\in[0,T]$. In addition, we introduce an auxiliary energy $E$ such that
\begin{equation}
    E(t)\geq\mathcal{E}(\pmb U(t))
\end{equation}
for a.e.\ $t\in[0,T]$, as well as a so-called regularity weight
\begin{equation}\label{regularity weight}
    \mathcal{K}:\mathbb{Y}\to[0,\infty)\text{ with }\mathcal{K}(0)=0,
\end{equation}
which is to be chosen such that both sides of the variational inequality below in \eqref{intro:evs} remain finite. Therefore, by replacing $\mathcal{E}(\pmb U)$ in~\eqref{intro:ede+wf} and adding on its right-hand side the term $\mathcal{K}(\tU)\left(E-\mathcal{E}(\pmb U)\right)$, we arrive at
\begin{equation}\label{intro:evs}
        \left(E-\langle \pmb U,\tU\rangle_{\mathbb{Y}}\right)\Big|_{s}^{t}
        +\int_{s}^{t}
        \Psi(\pmb U)
        +\langle \pmb U,\partial_{t}\tU\rangle_{\mathbb{Y}}
        -\langle A(\pmb U),\tU \rangle_{\mathbb{Y}}\dd{\tau}
        \leq
        \int_{s}^{t}\mathcal{K}(\tU)\left(E-\mathcal{E}(\pmb U)\right)\dd{\tau}.
    \end{equation}
\begin{defi}[Energy-variational solution for the general system~\eqref{intro:Eq}]\label{defi: general evs}
    We call a pair $(\pmb U,E)$ an energy-variational solution to~\eqref{intro:Eq} if $\mathcal{E}(\pmb U)\leq E$ on $(0,T)$ and if~\eqref{intro:evs} {holds} for all $\tU\in C_{0}^{\infty}([0,T);\mathbb{Y})$ and a.e. $s<t\in(0,T)$ including $s=0$ with $\pmb U(0)=\pmb U_{0}$.
\end{defi}
Now, following~\cite{ALREVS}, we formally define the relative energy as the residual term of the first-order Taylor expansion of the energy $\mathcal{E}$ around $\tU\in C_0^\infty([0,T);\mathbb{Y})$, which is given by
\begin{equation}\label{intro:re}
    {\mathcal{R}}(\pmb{U}|\tU):= \mathcal{E}(\pmb{U})-\mathcal{E}(\tU)-\langle \mathrm{D}\mathcal{E}(\tU), \pmb{U}-\tU \rangle_{{\mathbb{V}}}.
\end{equation}
Moreover, we formally define the relative energy with respect to the auxiliary energy as 
\begin{equation}\label{intro:re up}
    \begin{aligned}
        \widetilde{\mathcal{R}}(\pmb{U},E|\tU):= E-\mathcal{E}(\pmb{U})+\mathcal{R}(\pmb{U}|\tU).
    \end{aligned}
\end{equation}
In addition, we formally define the relative dissipation as the residual term of the first-order Taylor expansion around $\tU\in C_0^\infty([0,T);\mathbb{Y})$ of the term
\begin{equation*}
    \Psi(\pmb U) - \langle A(\pmb U), \mathrm{D} \mathcal E (\tU) \rangle_{{\mathbb{Y}}} + \mathcal{K}(\tU) \mathcal E(\pmb U),
\end{equation*}
which is then given as
\begin{equation}\label{intro:rd}
\begin{aligned}
    \mathcal{W}^{(\mathcal{K})}(\pmb U| \tU) &: =
   \Psi(\pmb U) - \langle A(\pmb U) , \mathrm{D}\mathcal E (\tU)\rangle_{\mathbb{Y}}  
   - \langle A(\tU) ,\mathrm{D}^2\mathcal E(\tU)(\pmb U - \tU) \rangle_{\mathbb{Y}} 
   + \mathcal{K}(\tU)\mathcal{R}(\pmb U| \tU) 
   \,. 
\end{aligned}
\end{equation}
    Furthermore, we define the system operator $\mathcal{A}$ as the strong form of the system~\eqref{intro:Eq}, which is given as 
\begin{equation}\label{intro:so}
    \mathcal{A}(\tU):=\partial_{t}\tU+A(\tU).
\end{equation}
Finally, by~\cite[][Proposition 3.6]{ALREVS}, an energy-variational solution in the sense of Definition~\ref{defi: general evs} satisfies the following relative energy-dissipation estimate
\begin{equation}\label{intro:rede}
    \begin{aligned}
        \widetilde{\mathcal{R}}(\pmb{U},E|\tU)\Big|_{s}^{t}
        +\int_{s}^{t}\mathcal{W}^{(\mathcal{K})}(\pmb{U}|\tU)+\langle \mathcal{A}(\tU),\mathrm{D}^{2}\mathcal{E}(\tU)(\pmb{U}-\tU)\rangle_{{\mathbb{Y}}}\dd{\tau}
        \leq \int_{s}^{t}\mathcal{K}(\mathrm{D}\mathcal{E}(\tU))\widetilde{\mathcal{R}}(\pmb{U},E|\tU)\dd{\tau}.
    \end{aligned}
\end{equation}
Applying Gronwall's inequality and replacing the auxiliary energy $E$ by the energy $\mathcal{E}$ provides the relative energy inequality in this general setting as 
\begin{equation}\label{intro:rede with gronwall}
    \begin{aligned}
     \mathcal{R}(\pmb U(t) | \tU(t)) + \int_0^t \mathrm \exp{\left(\int_s^t\mathcal K(\tU) \dd{\tau} \right)}\left[ \mathcal{W}^{(\mathcal{K})}(\pmb U| \tU ) + \langle  \mathcal{A}(\tU), \mathrm{D}^2\mathcal E (\tU)(\pmb U-\tU) \rangle_{{\mathbb{Y}}} \right] \dd{s}  
 \\  \leq \mathcal{R}(\pmb{U}_0 | \tU(0))\mathrm \exp{\left(\int_0^t\mathcal{K}(\mathrm{D}\mathcal{E}(\tU))\dd{s} \right) }   \,
\end{aligned}
\end{equation}   
for a.e.\ $t\in (0,T)$ and all smooth test functions $\tU\in C_{0}^{\infty}([0,T);\mathbb{Y})$. 
\begin{defi}[Dissipative solutions for general system~\eqref{intro:Eq}]\label{defi: general ds}
    {We call a function $\pmb U:(0,T)\to\mathbb{V}$ a dissipative solution} for system~\eqref{intro:Eq}, if $\pmb U$ satisfies~\eqref{intro:rede with gronwall} for all sufficiently regular test functions $\tU\in C_{0}^{\infty}([0,T);\mathbb{Y})$ and for a.e. $t\in(0,T)$.
\end{defi}
\end{subequations}
\subsection{Energy-variational solution for system~\eqref{sys:two phase}}
\label{Sec:ENVarSol-Geo}
\begin{defi}[Energy-variational solution for the geodynamical two-phase system \eqref{sys:two phase}] \label{defi: EVS}
\begin{subequations}\phantom{x}\\
Let $\gamma\geq0$ in \eqref{sys:two phase} with a phase-field potential $W=W_\mathrm{dw}+W_\mathrm{sing}$ as in \eqref{def-W}. Let Assumptions~\ref{ASM:domain}-\ref{ASM:Double well potential} {hold}.  Let $\mathcal{K}$ be a regularity weight satisfying~\eqref{regularity weight}. A quintuplet $(v,S,\varphi,\mu,E)$ is called an energy-variational solution of type $\mathcal{K}$ of the geodynamical two-phase system~\eqref{sys:two phase} if the following properties are satisfied:
\\
1. The quintuplet $(v,S,\varphi,\mu,E)$ has the following regularity:
\begin{equation}
\label{EVS-reg}
\begin{aligned}    
	v&\in L_{\mathrm{loc}}^{\infty}((0,T),L_{\mathrm{div}}^{2}(\Omega))\cap L_{\mathrm{loc}}^{2}((0,T),H_{0,\mathrm{div}}^{1}(\Omega)),
    \\
	S&\in L_{\mathrm{loc}}^{\infty}((0,T),L_{\mathrm{sym,Tr}}^{2}(\Omega)),
    \\
	\varphi&\in L_{\mathrm{loc}}^{\infty}((0,T);H^{1}(\Omega))\cap L_{\mathrm{loc}}^{2}((0,T);H^{2}(\Omega))\text{ with }\abs{\varphi}\leq1\text{ a.e. in }\Omega\times[0,T),
    \\
	\mu&\in L_{\mathrm{loc}}^{2}((0,T);H^{1}(\Omega)),
    \\
    E&\in BV_{\mathrm{loc}}([0,T])\text{ satisfying }E(t)\geq\mathcal{E}(v(t),S(t),\varphi(t)) \text{ for a.e. } t\in[0,T).
\end{aligned}
\end{equation}
2. The quintuplet $(v,S,\varphi,\mu,E)$ satisfies the following estimate:
\begin{equation}\label{EVS inequality}
\begin{aligned}
    &\left(E(\tau)-\langle \rho v,\tilde{v}\rangle_{L^{2}}-\langle S,\tilde{S} \rangle_{L^{2}}-\langle \varphi,\tilde{\varphi} \rangle_{L^{2}}\right)\Big|_{s}^{t}
        \\
    +&\int_{s}^{t}\int_{\Omega}2\nu(\varphi)\abs{\sym{\nabla v}}^{2}+\gamma\abs{\nabla S}^{2}+m(\varphi)\abs{\nabla\mu}^{2}\dd{x}\dd{\tau}+\int_{s}^{t}\mathcal{P}(\varphi;S)-\mathcal{P}(\varphi;\tilde{S})\dd{\tau}
        \\
    +&\int_{s}^{t}\int_{\Omega}
    \rho v\cdot\partial_{t}\tilde{v}
    +\rho v\otimes v:\nabla\tilde{v}
    +v\otimes J:\nabla\tilde{v}
    -2\nu(\varphi)\sym{\nabla v}:\sym{\nabla\tilde{v}}
    -\eta(\varphi)S:\sym{\nabla\tilde{v}}\dd{x}\dd{\tau}
        \\
    +&\int_{s}^{t}\int_{\Omega}\varepsilon\nabla\varphi\otimes\nabla\varphi:\nabla\tilde{v}\dd{x}\dd{\tau}
        \\
    +&\int_{s}^{t}\int_{\Omega}S:\partial_{t}\tilde{S}
    +S\otimes v\threedotsbin\nabla\tilde{S}
    -\left( S\skw{\nabla v} - \skw{\nabla v}S\right):\tilde{S}-\gamma \nabla S\threedotsbin \nabla\tilde{S} 
    +\eta(\varphi)\sym{\nabla v}:\tilde{S}\dd{x}\dd{\tau}
        \\
    +&\int_{s}^{t}\int_{\Omega}\varphi\partial_{t}\tilde{\varphi}
    -(v\cdot\nabla\varphi)\tilde{\varphi}
    -m(\varphi)\nabla \mu:\nabla\tilde{\varphi}\dd{x}\dd{\tau}
        \\
    -&\int_{s}^{t}\int_{\Omega}\mu\tilde{\mu}
    -(-\varepsilon\Delta\varphi
    +\varepsilon^{-1} W^{\prime}_{\mathrm{dw}}(\varphi)
    +\varepsilon^{-1}\beta
    )\tilde{\mu}\dd{x}\dd{\tau}
        \\
    \leq&\int_{s}^{t}\mathcal{K}(\tilde{v},\tilde{S},\tilde{\varphi},\tilde{\mu})\left(E(\tau)-\mathcal{E}(v,S,\varphi)\right)\dd{\tau}+\int_{s}^{t}\langle f,v-\tilde{v}\rangle_{H^{1}}{\dd{\tau}} 
\end{aligned}
\end{equation}
for all $(\tilde{v},\tilde{S},\tilde{\varphi},\tilde{\mu})\in\mathfrak{T}$ and for a.e.\ $s<t\in(0,T)$ including {the value} $s=0$ with $(v(0),S(0),\varphi(0))=(v_{0},S_{0},\varphi_{0})$
and such that $\beta(x,t)\in \partial W_\mathrm{sing}(\varphi(x,t))$ a.e.\ in $\Omega\times(0,T)$ with $\partial W_\mathrm{sing}$ from \eqref{def-subdiff-Wsing}.
\end{subequations}
\end{defi}
Recall that the parameter $\varepsilon$, controlling the interfacial thickness, is fixed throughout this work. Since $\varepsilon$ is treated as a fixed model parameter, it can be normalized, without affecting the validity of the proofs. Hence, for simplicity of notation, we assume $\varepsilon=1$ in what follows.
\begin{remark}
    If $\gamma>0$, then the energy-variational solution enjoys additional regularity. In particular, we further have $S\in L_{\mathrm{loc}}^{2}((0,T);H_{\mathrm{sym,Tr}}^{1}(\Omega))$.
\end{remark}
\begin{remark}
    To treat mobilities that vanish in the pure phases, one introduces the flux
    \begin{equation*}
        F=m(\varphi)\nabla\mu
    \end{equation*}
    following~\cite{zbMATH06286084}. In this setting, the test function $\tilde{\mu}$ in~\eqref{EVS inequality} has to be chosen in the form 
    \begin{equation*}
        \tilde{\mu}=\divge{m(\varphi)\hat{\mu}}\,,
    \end{equation*}
    so \eqref{EVS inequality} is consistent with the definition of the flux $F$.
    However, notice that, when the mobility $m$ vanishes at the pure phases, the relation $F=m(\varphi)\nabla\mu$ can, in general, only be understood in a formal sense. Indeed, in this case, one does not have sufficient control of $\nabla\mu$, and, therefore, $F$ cannot be identified through the above expression as in~\cite{zbMATH06286084}. 
\end{remark}
\subsection{Properties of energy-variational solutions and connections to dissipative solutions}
\label{Sec:DissipSol-Geo}
In what follows, we collect useful properties of energy-variational solutions. 
\begin{remark}[{Semiflow} property]
    An energy-variational solution $(v,S,\varphi,\mu,E)$ in the sense of Definition~\ref{defi: EVS} for the geodynamical two-phase system~\eqref{sys:two phase} satisfies the {semiflow} property as introduced in~\cite[][Remark 3.3]{ALREVS} for general PDE systems. This means that for a solution $(v,S,\varphi,\mu,E)$ on $[0,T)$, every restriction to an interval $[s,t]$ for all $s<t\in[0,T)$ with  the initial value $(v(s),S(s),\varphi(s),\mu(s),E(s_{-}))$ is again a solution, with $s_-$ as in \eqref{lr-limit}. Moreover, let $0\leq r<s<t<T$. If $(v^{1},S^{1},\varphi^{1},\mu^{1},E^{1})$ is an energy-variational solution on $[r,s]$ to the initial data $(v^{1}(r),S^{1}(r),\varphi^{1}(r),\mu^{1}(r),E^{1}(r_{-}))$ and if $(v^{2},S^{2},\varphi^{2},\mu^{2},E^{2})$ is an energy-variational solution on $[s,t]$ with initial data $(v^{2}(s),S^{2}(s),\varphi^{2}(s),\mu^{2}(s),E^{2}(s_{-}))$  and such that 
    \begin{equation*}
    (v^{1}(s),S^{1}(s),\varphi^{1}(s),\mu^{1}(s),E^{1}(s_{+}))=(v^{2}(s),S^{2}(s),\varphi^{2}(s),\mu^{2}(s),E^{2}(s_{-}))\,, 
    \end{equation*}
    then the concatenation
    \begin{equation*}
        (v,S,\varphi,\mu,E):=\begin{cases}
        (v^{1},S^{1},\varphi^{1},\mu^{1},E^{1})&\text{ on }[r,s], \\
        (v^{2},S^{2},\varphi^{2},\mu^{2},E^{2})&\text{ on }{[s,t]} \\
    \end{cases}
    \end{equation*}
    is a solution on $[r,t]$ with initial value $(v^{1}(r),S^{1}(r),\varphi^{1}(r),\mu^{1}(r),E^{1}(r_{-}))$.
\end{remark}

\begin{prop}\label{recover weak momentum balance}
    Suppose that the regularity weight $\mathcal{K}$ satisfies $\mathcal{K}(\tilde{v},0,0,0)=0$ for all $\tilde{v}\in C_{0,\mathrm{div}}^{\infty}(\Omega\times[0,T))$. Then an energy-variational solution $(v,S,\varphi,\mu,E)$ of type $\mathcal{K}$ is a weak solution of the momentum balance~\eqref{sys: momentum}, i.e., 
    \begin{equation}\label{weak formualtion: momentum}
        \begin{aligned}
           &\langle \rho v,\Phi\rangle_{L^{2}}\Big|_{s}^{t}
           -\int_{s}^{t}\int_{\Omega}
        \rho v\cdot\partial_{t}\Phi
        +\rho v\otimes v:\nabla\Phi
        +v\otimes J:\nabla\Phi\dd{x}\dd{\tau}
        \\
        &+\int_{s}^{t}\int_{\Omega}
        2\nu(\varphi)\sym{\nabla v}:\sym{\nabla\Phi}
        +\eta(\varphi)S:\sym{\nabla \Phi}\dd{x}\dd{\tau}
        \\
        =&\int_{s}^{t}\langle f,\Phi\rangle_{H^{1}}\dd{\tau}
        +\int_{s}^{t}\int_{\Omega}\varepsilon\nabla\varphi\otimes\nabla\varphi:\nabla\Phi\dd{x}\dd{\tau} 
        \end{aligned}
    \end{equation}
is satisfied for all $\Phi\in C_{0,\mathrm{div}}^{\infty}(\Omega\times[0,T))$ and for a.e. $s<t\in(0,T)$ including $s=0$ with $v(0)=v_{0}$.
\end{prop}
\begin{proof}
    For this, we choose in~\eqref{EVS inequality} the test functions $\tilde{v}=\alpha\Phi$ with $\alpha>0$ and $\Phi\in C_{0,\mathrm{div}}^{\infty}(\Omega\times[0,T))$, as well as $\tilde{S}\equiv0$, $\tilde{\varphi}\equiv0$, $\tilde{\mu}\equiv0$. Then, multiplying both sides of the inequality \eqref{EVS inequality} by $1/\alpha$ and letting $\alpha\to\infty,$ yields
    \begin{equation}
        \begin{aligned}
            &\langle \rho v,\Phi\rangle_{L^{2}}\Big|_{s}^{t}
           -\int_{s}^{t}\int_{\Omega}
        \rho v\cdot\partial_{t}\Phi
        +\rho v\otimes v:\nabla\Phi
        +v\otimes J:\nabla\Phi\dd{x}\dd{\tau}
        \\
        &+\int_{s}^{t}\int_{\Omega}
        2\nu(\varphi)\sym{\nabla v}:\sym{\nabla\Phi}
        +\eta(\varphi)S:\sym{\nabla \Phi}\dd{x}\dd{\tau}
        \\
        \leq&\int_{s}^{t}\langle f,\Phi\rangle_{H^{1}}\dd{\tau}
        +\int_{s}^{t}\int_{\Omega}\varepsilon\nabla\varphi\otimes\nabla\varphi:\nabla\Phi\dd{x}\dd{\tau}.
        \end{aligned}
    \end{equation}
    Notice that the same inequality holds for $\tilde{\Phi}=-\Phi$. Therefore, we arrive at~\eqref{weak formualtion: momentum}.
\end{proof}
\begin{prop}\label{recover weak CH}
 Let the phase-field potential $W=W_\mathrm{dw}+W_\mathrm{sing}$ {be} as in \eqref{def-W}. 
    Suppose that the regularity weight $\mathcal{K}$ satisfies $\mathcal{K}(0,0,\tilde{\varphi},0)=0$ for all $\tilde{\varphi}\in C_{0}^{\infty}(\Omega\times[0,T))$. Then an energy-variational solution $(v,S,\varphi,\mu,E)$ of type $\mathcal{K}$ is a weak solution of the phase-field evolutionary law~\eqref{sys: first CHE}, i.e., 
    \begin{equation}\label{weak formualtion: CH}
        \begin{aligned}
            &\langle \varphi,\zeta\rangle_{L^{2}}\Big|_{s}^{t}
            -\int_{s}^{t}\int_{\Omega}
            \varphi\partial_{t}\zeta
            -v\cdot\nabla\varphi\zeta
            -m(\varphi)\nabla\mu\cdot\nabla\zeta\dd{x}\dd{\tau}=0
        \end{aligned}
    \end{equation}
    is satisfied for all $\zeta\in C_{0}^{\infty}(\Omega\times[0,T))$ and for a.e. $s<t\in(0,T)$ including $s=0$ with $\varphi(0)=\varphi_{0}$.
\end{prop}
\begin{proof}
    {This can be shown} by repeating the arguments of Proposition~\ref{recover weak momentum balance}.
\end{proof}
    
\begin{prop}\label{recover strong GT}
 Let the phase-field potential $W=W_\mathrm{dw}+W_\mathrm{sing}$ as in \eqref{def-W}.  
    Suppose that the regularity weight $\mathcal{K}$ satisfies $\mathcal{K}(0,0,0,\tilde{\mu})=0$ for all $\tilde{\mu}\in C_{0}^{\infty}(\Omega\times[0,T))$. Then an energy-variational solution $(v,S,\varphi,\mu,E)$ of type $\mathcal{K}$ is a strong solution of the {Gibbs--Thomson} law~\eqref{sys: second CHE}, i.e., 
    \begin{equation}\label{strong TG law}
        \begin{aligned}
            \mu&=-\Delta\varphi+W_{\mathrm{dw}}^{\prime}(\varphi)+\beta\quad \text{ and \; $\beta\in\partial W_\mathrm{sing}(\varphi)$ }
        \end{aligned}
    \end{equation}
     is satisfied for a.e. $(x,t)\in\Omega\times[0,T)$.
\end{prop}
\begin{proof}
    This proposition can be shown by following the same idea of Proposition~\ref{recover weak momentum balance} and with the help of the fundamental lemma of the calculus of variations.
\end{proof}
\begin{prop}
    Let $(v,S,\varphi,\mu,E)$ be an energy-variational solution of type $\mathcal{K}$. Let $\mathcal{K}$, $\mathcal{L}$ be regularity weights with the property $\mathcal{K}(\tilde{v},\tilde{S},\tilde{\varphi},\tilde{\mu})\leq\mathcal{L}(\tilde{v},\tilde{S},\tilde{\varphi},\tilde{\mu})$ for all $(\tilde{v},\tilde{S},\tilde{\varphi},\tilde{\mu})\in \mathfrak{T}$ and a.e. $t\in(0,T)$. Then $(v,S,\varphi,\mu,E)$ is also {an} energy-variational solution of type $\mathcal{L}$.
\end{prop}
\begin{proof}
    Notice that this proposition is a direct consequence of the fact $E(\tau)\geq\mathcal{E}(v(\tau),S(\tau),\varphi(\tau))$ for a.e. $\tau\in(0,T)$. 
\end{proof}

{To see} that an energy-variational solution satisfies a relative energy-dissipation estimate and to see the connection with dissipative solutions, we follow the general approach introduced in Section~\ref{intro:general pde setting} and recall from~\cite{arxiv:2509.25508} the definitions of the relative energy, relative dissipation, and the system operator for system \eqref{sys:two phase} as {follows:}

\begin{subequations}
We define the relative kinetic energy as 
\begin{equation}\label{defi: relative kinetic energy}
    \begin{aligned}        
    \mathcal{R}_{\mathrm{kin}}\left( \varphi;v|\tilde{v}\right):=\int_{\Omega}\rho(\varphi)\frac{\abs{v-\tilde{v}}^{2}}{2} \dd{x}=\mathcal{E}_{\mathrm{kin}}(\varphi,v-\tilde{v}).
    \end{aligned}
\end{equation}
Moreover, the relative elastic energy is given by
\begin{equation}\label{defi: relative elastic energy}
    \mathcal{R}_{\mathrm{el}}(S|\tilde{S}):=\int_{\Omega}\frac{|S-\tilde{S}|^{2}}{2}\dd{x}=\mathcal{E}_{\mathrm{el}}(S-\tilde{S}).
\end{equation}
In addition, we define the relative phase-field energy as 
\begin{equation}\label{defi: relative phase-field energy}
\begin{aligned}
    \mathcal{R}_{\mathrm{pf},\kappa}(\varphi|\tilde{\varphi})&
    :=\mathcal{E}_{\mathrm{pf},\kappa}({\varphi})
    -\mathcal{E}_{\mathrm{pf},\kappa}(\tilde{\varphi})
    -\mathrm{D}\mathcal{E}_{\mathrm{pf},\kappa}(\tilde{\varphi})(\varphi-\tilde{\varphi})
    \\
    &=\varepsilon\int_{\Omega}\frac{\abs{\nabla\varphi-\nabla\tilde{\varphi}}^{2}}{2}\dd{x}+\varepsilon^{-1}\int_{\Omega}W_{\kappa}(\varphi)-W_{\kappa}(\tilde{\varphi})-W_{\kappa}^{\prime}(\tilde{\varphi})(\varphi-\tilde{\varphi})\dd{x}{,}
\end{aligned}
\end{equation}
where
\begin{equation*}
    W_{\kappa}(\varphi):=W(\varphi)+\frac{\kappa}{2}\abs{\varphi}^{2}
\end{equation*}
with $\kappa\geq0$ from \eqref{assumption on singular potential}.  
Therefore, the relative total energy is given by
\begin{equation}\label{defi: relative total energy}
\begin{aligned}
\mathcal{R}\left(v,S,\varphi\big|\tilde{v},\tilde{S},\tilde{\varphi}\right):=\mathcal{R}_{\mathrm{kin}}(\varphi;v|\tilde{v})+\mathcal{R}_{\mathrm{el}}(S|\tilde{S})+\mathcal{R}_{\mathrm{pf},\kappa}(\varphi|\tilde{\varphi}).
\end{aligned}
\end{equation}
Next we introduce the system operator for system \eqref{sys:two phase} for the case that the singular potential is logarithmic, 
{i.e.,}\ $W=W_\mathrm{dw}+W_\mathrm{sg}$  in \eqref{def-W}. This requires, in particular, that the smooth test functions $\tilde\varphi$ take values in $(-1,1),$ only. Hence, in this setting,  the system operator for system \eqref{sys:two phase} is given by 
\begin{equation}\label{defi: system operator}
    \mathcal{A}_{\gamma}=\left( \mathcal{A}^{(1)},\mathcal{A}_{\gamma}^{(2)},\mathcal{A}^{(3)}\right)^{\top}{,}   
\end{equation}
where
\begin{align}
    \mathcal{A}^{(1)}(\tilde{v},\tilde{S},\tilde{\varphi})
    &:=
    \partial_{t}(\tilde{\rho}\tilde{v})
    +\divge{\tilde{v}\otimes(\tilde{\rho}\tilde{v}+\tilde{J})}
    -\divge{\eta(\tilde{\varphi})\tilde{S}+2\nu(\tilde{\varphi})\sym{\nabla\tilde{v}}}
    -\tilde{\mu}\nabla\tilde{\varphi}- f,
            \\
    \mathcal{A}_{\gamma}^{(2)}(\tilde{v},\tilde{S},\tilde{\varphi})
    &:=
    \partial_{t}\tilde{S}+\tilde{v}\cdot\nabla\tilde{S}+\left( \tilde{S}\skw{\nabla\tilde{v}}-\skw{\nabla\tilde{v}}\tilde{S}\right)-\gamma\Delta\tilde{S}-\eta(\tilde{\varphi})\sym{\nabla\tilde{v}},
            \\
    \mathcal{A}^{(3)}(\tilde{v},\tilde{\varphi})
    &:=
    \partial_{t}\tilde{\varphi}+\tilde{v}\cdot\nabla\tilde{\varphi}-\divge{m(\tilde{\varphi})\nabla\tilde{\mu}},
            \\
    \tilde{\mu}
    &:=
    -\Delta \tilde{\varphi}+W_{\mathrm{dw}}^{\prime}(\tilde{\varphi})+ W_\mathrm{sg}^{\prime}(\tilde\varphi),
\end{align}
for any test function 
$\tilde\varphi\in C_0^\infty(\Omega\times[0,T))$ with $|\tilde\varphi|<1$.   
In addition, the relative dissipation is defined as 
\begin{equation}\label{defi: relative dissipation}
    \begin{aligned}    \mathcal{W}_{\gamma}^{(\mathcal{K})}\left(v,S,\varphi|\tilde{v},\tilde{S},\tilde{\varphi}\right):=
        &\int_{\Omega}\gamma\abs{\nabla S-\nabla\tilde{S}}^{2}\dd{x}
    +\int_{\Omega}m(\varphi)\abs{\nabla\mu-\nabla\tilde{\mu}}^{2}\dd{x}
    \\
    &+\int_{\Omega}2\nu(\varphi)\abs{\sym{\nabla v}-\sym{\nabla\tilde{v}}}^{2}
    +2( \nu(\varphi)-\nu(\tilde{\varphi}))\sym{\nabla\tilde{v}}:\left({\nabla v}-{\nabla\tilde{v}}\right)\dd{x}
    \\
        &+\int_{\Omega}m(\varphi)\nabla\tilde{\mu}:(\nabla\mu-\nabla\tilde{\mu})+ \divge{m(\tilde{\varphi})\nabla\tilde{\mu}}\left(-\Delta(\varphi-\tilde{\varphi})+W^{\prime\prime}(\tilde{\varphi})(\varphi-\tilde{\varphi})\right)\dd{x}
    \\
        &+\int_{\Omega}(v-\tilde{v})\otimes(\rho v-\tilde{\rho}\tilde{v} + J-\tilde{J}):\nabla\tilde{v}
        +(\rho-\tilde{\rho})(v-\tilde{v})\cdot\partial_{t}\tilde{v}\dd{x}
    \\
        &-\int_{\Omega}(\eta(\varphi)-\eta(\tilde{\varphi}))(S-\tilde{S}):\nabla\tilde{v}
		+(\eta(\varphi)-\eta(\tilde{\varphi}))\tilde{S}:({\nabla v}-{\nabla\tilde{v}})\dd{x}
	\\
		&-\int_{\Omega}(S-\tilde{S})\otimes(v-\tilde{v})\threedotsbin\nabla\tilde{S} 
        +2 
        (S-\tilde{S})\skw{\nabla v-\nabla\tilde{v}}
        :{\tilde{S}} \dd{x}
	\\
		&-\int_{\Omega}\tilde{\mu}(\nabla\varphi-\nabla\tilde{\varphi})\cdot(v-\tilde{v})
		-\varepsilon(\nabla\varphi-\nabla\tilde{\varphi})\otimes(\nabla\varphi-\nabla\tilde{\varphi}):\nabla\tilde{v} \dd{x}
	\\
		&+\int_{\Omega}\frac{\kappa}{\varepsilon}(\nabla\mu-\nabla\tilde{\mu})\cdot(\nabla\varphi-\nabla\tilde{\varphi})
		+\frac{\kappa}{\varepsilon}(v-\tilde{v})\cdot\nabla\tilde{\varphi}(\varphi-\tilde{\varphi}) \dd{x}
	\\
		&+\mathcal{K}(\tilde{v},\tilde{S},\tilde{\varphi})\mathcal{R}(v,S,\varphi|\tilde{v},\tilde{S},\tilde{\varphi}).
\end{aligned}
\end{equation}
\end{subequations}
Furthermore, as~\eqref{intro:re up}, we define the relative energy involving the auxiliary energy variable $E$ {as} 
\begin{equation*}
    \widetilde{\mathcal{R}}(v,S,\varphi,E|\tilde{v},\tilde{S},\tilde{\varphi}):= E-\mathcal{E}(v,S,\varphi)+\mathcal{R}(v,S,\varphi|\tilde{v},\tilde{S},\tilde{\varphi}).
\end{equation*}

Notice that, in our setting, the spaces {are}
\begin{equation}
\label{def-VY}
    \begin{aligned}
        \mathbb{V}&:=L_{\mathrm{div}}^{2}(\Omega)\times L_{\mathrm{sym,Tr}}^{2}(\Omega)\times L^{2}(\Omega),\\
        \mathbb{Y}&:= H_{0,\mathrm{div}}^{1}(\Omega)\times H_{\mathrm{sym,Tr}}^{1}(\Omega)\times H^{1}(\Omega).
    \end{aligned}
\end{equation}

As in~\cite[][Definition 3.3]{arxiv:2509.25508}, we introduce the dissipative solutions for the system~\eqref{sys:two phase} with logarithmic potentials as follows.

\begin{defi}[Dissipative solution for system~\eqref{sys:two phase} with a logarithmic potential]\label{defi:ds singular}
    Let $\gamma\geq0$. Let the assumptions~\ref{ASM:domain}-\ref{ASM:Double well potential} {hold}. Suppose that the phase-field potential is given by $W=W_{\mathrm{dw}}+W_{\mathrm{sg}}$ as in \eqref{def-W}. Assume further that $W\in C^{3}(-1,1)$. Let $\mathcal{K}$ be a regularity weight satisfying~\eqref{regularity weight}. A quadruplet $(v,S,\varphi,\mu)$ is called a dissipative solution of type $\mathcal{K}$ of the two-phase system~\eqref{sys:two phase} with a logarithmic potential $W_\mathrm{sg}$ if the following properties are satisfied:\\
    \begin{subequations}
    1.\ The quadruplet $(v,S,\varphi,\mu)$ has the following regularity:
    \begin{equation}\label{defi: ds regularity}
    \begin{aligned}    
		v&\in L_{\mathrm{loc}}^{\infty}([0,T),L_{\mathrm{div}}^{2}(\Omega))\cap L_{\mathrm{loc}}^{2}([0,T),H_{0,\mathrm{div}}^{1}(\Omega)),\\
		S&\in L_{\mathrm{loc}}^{\infty}([0,T),L_{\mathrm{sym,Tr}}^{2}(\Omega)),\\
		\varphi&\in L_{\mathrm{loc}}^{\infty}([0,T);H^{1}(\Omega))\cap L_{\mathrm{loc}}^{2}([0,T);H^{2}(\Omega)), W^{\prime}(\varphi)\in L_{\mathrm{loc}}^{2}([0,T);L^{2}(\Omega)),\\
		\mu&\in L_{\mathrm{loc}}^{2}([0,T);H^{1}(\Omega)).
    \end{aligned}
    \end{equation}
    2.\ With the relative energy $\mathcal{R}$ from~\eqref{defi: relative total energy}, the system operator $\mathcal{A}_{\gamma}$ from~\eqref{defi: system operator}, the relative dissipation $\mathcal{W}_{\gamma}^{(\mathcal{K})}$ from~\eqref{defi: relative dissipation}, and the space $\mathbb{Y}$ from \eqref{def-VY}, the quadruplet $(v,S,\varphi,\mu)$ satisfies {the relative energy-dissipation estimate:}
    \begin{equation}\label{defi: ds estimate}
    \begin{aligned}
			&\mathcal{R}(v(t),S(t),\varphi(t)|\tilde{v}(t),\tilde{S}(t),\tilde{\varphi}(t))
	\\
			+\int_{0}^{t}\Big(& \left\langle \mathcal{A}_{\gamma}(\tilde{v},\tilde{S},\tilde{\varphi}),
			\begin{pmatrix}
				v-\tilde{v}\\
				S-\tilde{S}\\
				-\Delta(\varphi-\tilde{\varphi})+W^{\prime\prime}(\tilde{\varphi})(\varphi-\tilde{\varphi})+\frac{\kappa}{\varepsilon}(\varphi-\tilde{\varphi})-\frac{\rho_{2}-\rho_{1}}{2}(v-\tilde{v})\cdot\tilde{v}
			\end{pmatrix}
			\right\rangle_{\!\!\!\mathbb Y} 
		\\
			&+\mathcal{P}(\varphi;S)-\mathcal{P}(\varphi;\tilde{S})+\mathcal{W}_{\gamma}^{(\mathcal{K})}(v,S,\varphi|\tilde{v},\tilde{S},\tilde{\varphi})\Big)
			\exp\left(\int_{s}^{t} \mathcal{K}(\tilde{v},\tilde{S},\tilde{\varphi})\dd{\tau} \right) \dd{s}
		\\
			\leq &\mathcal{R}(v_{0},S_{0},\varphi_{0}|\tilde{v}(0),\tilde{S}(0),\tilde{\varphi}(0)) \exp\left(\int_{0}^{t} \mathcal{K}(\tilde{v},\tilde{S},\tilde{\varphi}) \dd{s} \right)
    \end{aligned}
    \end{equation}
     for all $\tilde{v}\in C_{0,\mathrm{div}}^{\infty}(\Omega\times[0,T))$, $\tilde{S}\in C_{0,\mathrm{sym,Tr}}^{\infty}(\Omega\times[0,T))$ and  $\tilde{\varphi}\in C_{0}^{\infty}(\Omega\times[0,T))$ with $\abs{\tilde{\varphi}}<1$ and for a.e. $t\in(0,T)$. 
    \end{subequations} 
\end{defi}
Building on the framework introduced in Section~\ref{intro:general pde setting} and following the same ideas as~\cite[Proposition 3.6]{arxiv:2509.25508} and~\cite[Proposition 3.6]{ALREVS}, we obtain the following relative energy-dissipation estimate for energy-variational solutions with a logarithmic potential.

\begin{prop} \label{relative energy-dissipation inequality for EVS}
Let $\gamma\geq0$. Suppose that the phase-field potential is given by $W=W_{\mathrm{dw}}+W_{\mathrm{sg}}$. Moreover, further assume that $W\in C^{3}(-1,1)$. For an energy-variational solution $(v,S,\varphi,\mu,E)$, it holds {that}   
    \begin{equation}\label{relative energy-dissipation estimate}
    \begin{aligned}
			&\widetilde{\mathcal{R}}(v,S,\varphi,E|\tilde{v},\tilde{S},\tilde{\varphi})\Big|_{s}^{t}
	\\
			&+\int_{s}^{t} \left\langle \mathcal{A}_{\gamma}(\tilde{v},\tilde{S},\tilde{\varphi}),
			\begin{pmatrix}
				v-\tilde{v}\\
				S-\tilde{S}\\
				-\Delta(\varphi-\tilde{\varphi})+W^{\prime\prime}(\tilde{\varphi})(\varphi-\tilde{\varphi})+\frac{\kappa}{\varepsilon}(\varphi-\tilde{\varphi})-\frac{\rho_{2}-\rho_{1}}{2}(v-\tilde{v})\cdot\tilde{v}
			\end{pmatrix}
			\right\rangle\dd{\tau}
		\\
			&+\int_{s}^{t}\mathcal{P}(\varphi;S)-\mathcal{P}(\varphi;\tilde{S})+\mathcal{W}_{\gamma}^{(\mathcal{K})}(v,S,\varphi|\tilde{v},\tilde{S},\tilde{\varphi})\dd{\tau}		
		\\
			\leq&\int_{s}^{t}\mathcal{K}(\tilde{v},\tilde{S},\tilde{\varphi})\widetilde{\mathcal{R}}(v,S,\varphi,E|\tilde{v},\tilde{S},\tilde{\varphi}){\dd{\tau}}
    \end{aligned}
    \end{equation}
    for all $s<t\in(0,T)$ including $s=0$ with $(v(0),S(0),\varphi(0))=(v_{0},S_{0},\varphi_{0})$ {for all} $\tilde{v}\in C_{0,\mathrm{div}}^{\infty}(\Omega\times[0,T))$, $\tilde{S}\in C_{0,\mathrm{sym,Tr}}^{\infty}(\Omega\times[0,T))$ and  $\tilde{\varphi}\in C_{0}^{\infty}(\Omega\times[0,T))$ with $\abs{\tilde{\varphi}}<1$.

     By applying a version of Gronwall's inequality and estimating $E\geq\mathcal{E}(v,S,\varphi)$ to~\eqref{relative energy-dissipation estimate}, one can see that every energy-variational solution is a dissipative solution in the sense of Definition~\ref{defi:ds singular} in the case that $W=W_{\mathrm{dw}}+W_{\mathrm{sg}}$.
\end{prop}

To define the dissipative solutions for the system~\eqref{sys:two phase} with a double-obstacle potential, we introduce the system operator for the case that the singular potential is indicator function, i.e., $W=W_{\mathrm{dw}}+I_{[-1,1]}$ in~\eqref{def-Wsing}. This allows that the smooth test functions $\tilde{\varphi}$ take values in $[-1,1]$. Hence, in this setting, the system operator for system~\eqref{sys:two phase} is given as in~\eqref{defi: system operator}, {except} for the term
\begin{equation*}
    \tilde{\mu}:=-\Delta\tilde{\varphi}+W_{\mathrm{dw}}^{\prime}(\tilde{\varphi})\,,
\end{equation*}
where we use that $0\in\partial I_{[-1,1]}(\tilde{\varphi})$ and  that we have the freedom to select the element of the subdifferential $\partial I_{[-1,1]}(\tilde\varphi)$ of a test function $\tilde\varphi$ to be $0$.

Now, in analogy with Definition~\ref{defi:ds singular}, we now define the dissipative solutions for the system~\eqref{sys:two phase} with double-obstacle potentials.

\begin{defi}[Dissipative solution for system~\eqref{sys:two phase} with a double-obstacle potential]\label{defi:ds double well}
\phantom{x}\\
    Let $\gamma\geq0$. Let the assumptions~\ref{ASM:domain}-\ref{ASM:Double well potential} {hold}. Suppose that the phase-field potential is given by $W=W_{\mathrm{dw}}+I_{[-1,1]}$. Assume further that $W_{\mathrm{dw}}\in C^{3}(-1,1)$. Let $\mathcal{K}$ be a regularity weight satisfying~\eqref{regularity weight}. A quadruplet $(v,S,\varphi,\mu)$ is called a dissipative solution of type $\mathcal{K}$ of the two-phase system~\eqref{sys:two phase} with {an} indicator potential $I_{[-1,1]}$ if the following properties are satisfied:
    \begin{enumerate}
    \item The quadruplet $(v,S,\varphi,\mu)$ has the regularities~\eqref{defi: ds regularity}.
    \item With the relative energy $\mathcal{R}$ from~\eqref{defi: relative total energy}, the system operator $\mathcal{A}_{\gamma}$ from~\eqref{defi: system operator}, the relative dissipation $\mathcal{W}_{\gamma}^{(\mathcal{K})}$ from~\eqref{defi: relative dissipation}, and the space $\mathbb{Y}$ from \eqref{def-VY}, the quadruplet $(v,S,\varphi,\mu)$ satisfies the relative energy-dissipation estimate~\eqref{defi: ds estimate}
     for all $\tilde{v}\in C_{0,\mathrm{div}}^{\infty}(\Omega\times[0,T))$, $\tilde{S}\in C_{0,\mathrm{sym,Tr}}^{\infty}(\Omega\times[0,T))$ and  $\tilde{\varphi}\in C_{0}^{\infty}(\Omega\times[0,T))$ with $\abs{\tilde{\varphi}}\leq1$ and for a.e.\ $t\in(0,T)$. 
    \end{enumerate} 
\end{defi}

We observe that the main difference between {Definitions~\ref{defi:ds singular} and~\ref{defi:ds double well}} is that in Definition~\ref{defi:ds double well}, the class of admissible test functions is less restricted. In particular, we allow $\tilde{\varphi}$ to take value $\pm1$. 

Similarly to Proposition~\ref{relative energy-dissipation inequality for EVS}, we also have a relative energy-dissipation estimate for energy-variational solutions in the case of the double-obstacle potential.
\begin{prop}\label{relative energy-dissipation inequality for double well}
    Let $\gamma\geq0$. Suppose that the phase-field potential is given by $W=W_{\mathrm{dw}}+I_{[-1,1]}$. Moreover, we further assume that $W_{\mathrm{dw}}\in C^{3}(\mathbb{R})$. For an energy-variational solution $(v,S,\varphi,\mu,E)$, estimate~\eqref{relative energy-dissipation estimate} {holds} for all $s<t\in(0,T)$ including $s=0$ with $(v(0),S(0),\varphi(0))=(v_{0},S_{0},\varphi_{0})$, for all $\tilde{v}\in C_{0,\mathrm{div}}^{\infty}(\Omega\times[0,T))$, $\tilde{S}\in C_{0,\mathrm{sym,Tr}}^{\infty}(\Omega\times[0,T))$ and  $\tilde{\varphi}\in C_{0}^{\infty}(\Omega\times[0,T))$ with $\abs{\tilde{\varphi}}\leq1$.

    By applying a version of Gronwall's inequality and estimating $E\geq\mathcal{E}(v,S,\varphi)$ to~\eqref{relative energy-dissipation estimate}, one can see that every energy-variational solution is a dissipative solution in the sense of {Definition~\ref{defi:ds double well}} in the case that $W=W_{\mathrm{dw}}+I_{[-1,1]}$.
\end{prop}

\section{Existence of energy-variational 
solutions for system \eqref{sys:two phase} with logarithmic  potential}\label{Sec:singular}
In~\cite{arxiv:2509.25508}, the Cahn–Hilliard type equation \eqref{sys: first CHE}  
in system \eqref{sys:two phase} was considered with constant mobility. Now, we generalize the analysis to the case that the mobility is a function satisfying Assumption~\ref{ASM:mobility}. Throughout Section~\ref{Sec:singular}, we assume the phase-field potential $W$ to be the sum of a double-well potential $W_{\mathrm{dw}}$ and a  logarithmic potential $W_{\mathrm{sg}}$, i.e.,
\begin{equation*}
    W(\varphi)=W_{\mathrm{dw}}(\varphi)+W_{\mathrm{sg}}(\varphi).
\end{equation*}
\subsection{Time discretization}\label{time discretization}

In this section, we will use an implicit time discretization to show the existence of weak solutions.

To start with, we first define another dissipation potential $\widetilde{\mathcal{P}}$ as
\begin{align}
    \begin{split}
    \widetilde{\mathcal{P}}:L^{2}(\Omega)\times H_{\mathrm{sym,Tr}}^{1}(\Omega)&\to[0,+\infty]{,}
    \\
    \widetilde{\mathcal{P}}(\varphi;S)&:=\begin{cases}
        \mathcal{P}(\varphi;S)&(\varphi,S)\in L^{2}(\Omega)\times H_{\mathrm{sym,Tr}}^{1}(\Omega)\cap\dom{\mathcal{P}},\\
        +\infty&\mbox{otherwise}.
    \end{cases}
    \end{split}
    \label{restriction of dissipation potential}
\end{align} 
$\widetilde{\mathcal{P}}$ can be viewed as the restriction of $\mathcal{P}$ from \eqref{integral form of dissipation potential} to the space $L^{2}(\Omega)\times H_{\mathrm{sym,Tr}}^{1}(\Omega)$. Notice that $\widetilde{\mathcal{P}}$ is proper with $\widetilde{\mathcal{P}}(\varphi;0)=0$ for all $\varphi \in L^2(\Omega)$. Moreover, for all $\varphi \in L^2(\Omega)$, the mapping $S \mapsto \widetilde{\mathcal{P}}(\varphi; S) $ is convex and lower semicontinuous in $H_{\mathrm{sym,Tr}}^{1}(\Omega)$. 
We write $\dom{\partial\widetilde{\mathcal{P}}(\varphi;\cdot)}$ to represent the domain of the convex partial subdifferential.

For the time discretization, let $h=\frac{T}{N}$ for $N \in \mathbb{N}$ and let $t_{0}=0$, $t_{k}=kh$, $t_{N}=T$. Let the initial data $v_{0}, S_{0}, \varphi_{0}$ satisfy Assumption \ref{ASM:initial data}. For all $k=0,1,\ldots,N-1$, let $v_{k}\in L_{\mathrm{div}}^{2}(\Omega)$, $S_{k}\in L_{\mathrm{sym,Tr}}^{2}(\Omega)$, $\varphi_{k}\in H^1(\Omega)$ with $W^{\prime}(\varphi_{k})\in L^2(\Omega)$ and $\rho_{k}=\frac{1}{2}(\rho_{1}+\rho_{2})+\frac{1}{2}(\rho_{2}-\rho_{1})\varphi_{k}$. Moreover, let
\begin{equation*}
    f_{k+1}=\frac{1}{h}\int_{t_{k}}^{t_{k+1}}f(\tau)\dd{\tau}
\end{equation*}
and set 
\begin{subequations}\label{time discrete problem}
\begin{align*}
	J_{k+1}:=-\frac{\rho_{2}-\rho_{1}}{2}{m(\varphi_{k})}\nabla\mu_{k+1}\,.
\end{align*}
For all $k\in\{0,1,\ldots,N-1\},$ with given $(v_k,S_k,\varphi_k,\mu_k),$ we determine {the tuple} $(v_{k+1},\allowbreak S_{k+1},\allowbreak\varphi_{k+1},\allowbreak\mu_{k+1})$  
such that
\begin{align*}
	\begin{split}
	v_{k+1}\in H_{0,\mathrm{div}}^{1}(\Omega),S_{k+1}\in H_{\mathrm{sym,Tr}}^{1}(\Omega)\cap\dom{\partial\widetilde{\mathcal{P}}(\varphi_{k};\cdot)},
	\varphi_{k+1}\in \dom{\mathrm{D} \mathcal{E}_{\mathrm{pf},\kappa}},\mu_{k+1}\in H_{\vec{n}}^{2}(\Omega),
	\end{split}
\end{align*}
where $H_{\vec{n}}^{2}(\Omega):=\{u\in H^{2}(\Omega):\vec{n}\cdot\nabla u|_{\partial\Omega}=0\}$, and such that $(v_{k+1},S_{k+1},\varphi_{k+1},\mu_{k+1})$  satisfies {the following}:\\
1. The weak formulation of the discrete momentum {balance,}
\begin{align}
	\langle \frac{\rho_{k+1} v_{k+1}-\rho_{k}v_{k}}{h}+ \divge{\rho_{k}v_{k+1}\otimes v_{k+1}},\Phi \rangle_{L^{2}}
    &	\nonumber\\
	+\langle 2\nu(\varphi_{k})\sym{\nabla v_{k+1}},\sym{\nabla\Phi} \rangle_{L^{2}} 
	-\langle \divge{\eta(\varphi_{k})S_{k+1}},\Phi \rangle_{L^{2}}&
	\nonumber\\
	+\langle (\divge{J_{k+1}}-\frac{\rho_{k+1}-\rho_{k}}{h}-v_{k+1}\cdot\nabla\rho_{k})\frac{v_{k+1}}{2},\Phi \rangle_{L^{2}}&
	\nonumber\\
    +\langle J_{k+1}\cdot\nabla v_{k+1}
    -\mu_{k+1}\nabla\varphi_{k},\Phi \rangle_{L^{2}}&={\langle f_{k+1},\Phi\rangle_{H^{1}}}
	\label{discrete velocity}
\end{align}
for all $\Phi\in C_{0,\mathrm{div}}^{\infty}(\Omega)$.\\
2. The weak formulation of the discrete evolution law for the stress {tensor,}
\begin{align}
	\langle \frac{S_{k+1}-S_{k}}{h} 
    + (v_{k+1}\cdot\nabla S_{k+1}),\Psi \rangle_{L^{2}}&
	\nonumber\\
	+\langle (S_{k+1}\skw{\nabla v_{k+1}}-\skw{\nabla v_{k+1}}S_{k+1}),\Psi \rangle_{L^{2}}& 
	\nonumber\\
	+\langle \xi_{k+1}^{k},\Psi\rangle_{H_{\mathrm{sym,Tr}}^{1}}
    +\langle \gamma\nabla S_{k+1},\nabla\Psi \rangle_{L^{2}}
    &=\langle \eta(\varphi_{k})\sym{\nabla v_{k+1}},\Psi\rangle_{L^{2}} 
	\label{discrete stress}
\end{align}
for all $\Psi\in C_{\mathrm{sym,Tr}}^{\infty}(\bar{\Omega})$, and here $\xi_{k+1}^{k}\in\partial\widetilde{\mathcal{P}}(\varphi_{k};S_{k+1})\subseteq (H_{\mathrm{sym,Tr}}^{1}(\Omega))^{\prime}$.\\
3. The discrete evolution law for the phase-field {variable,}
\begin{align}
	\frac{\varphi_{k+1}-\varphi_{k}}{h}+v_{k+1}\cdot\nabla\varphi_{k}&=\divge{m(\varphi_{k})\nabla\mu_{k+1}}, 
	\label{discrete first CH}
\end{align}
as well as
\begin{align}
	\mu_{k+1}+\kappa\frac{\varphi_{k+1}+\varphi_{k}}{2}&=-\Delta\varphi_{k+1}+W_{\kappa}^{\prime}(\varphi_{k+1}), 
	\label{discrete second CH}
\end{align}
almost everywhere in $\Omega$.
\end{subequations}

\begin{lem}[Existence of solutions to the time discrete problem]\label{Existence of solution to the time discrete problem}
	{Let} $v_{k}\in L_{\mathrm{div}}^{2}(\Omega)$, $S_{k}\in L_{\mathrm{sym,Tr}}^{2}(\Omega)$, $\varphi_{k}\in H^{2}(\Omega)$ with $\abs{\varphi_{k}}\leq1$ and $\rho_{k}=\frac{\rho_{2}-\rho_{1}}{2}\varphi_{k}+\frac{\rho_{2}+\rho_{1}}{2}$ be given {for $k\in\{0,1,\ldots,N-1\}$}. Let $\widetilde{\mathcal{P}}$ be as in \eqref{restriction of dissipation potential}  and set
    \begin{equation*}
        X:=H_{0,\mathrm{div}}^{1}(\Omega)\times H_{\mathrm{sym,Tr}}^{1}(\Omega)\cap\dom{\partial\widetilde{\mathcal{P}}(\varphi_{k};\cdot)}\times\dom{\mathrm{D}\mathcal{E}_{\mathrm{pf},\kappa}}\times H_{\vec{n}}^{2}(\Omega).
    \end{equation*} 
    Then there exists a quadruplet $(v_{k+1},S_{k+1},\varphi_{k+1},\mu_{k+1})\in X$ solving \eqref{discrete velocity}-\eqref{discrete second CH}. 
    Moreover, this solution satisfies the discrete energy-dissipation estimate
    \begin{equation}\label{discrete total energy estimate}
        \begin{aligned}
            &\mathcal{E}(v_{k+1},S_{k+1},\varphi_{k+1})
            \\
            &+h\int_{\Omega}2\nu(\varphi_{k})
            \abs{\sym{\nabla v_{k+1}}}^{2}
            +\gamma\abs{\nabla S_{k+1}}^{2}
            +m(\varphi_{k})\abs{\nabla\mu_{k+1}}^{2}
            \dd{x}
            +h\langle \xi_{k+1}^{k},S_{k+1} \rangle_{H_{\mathrm{sym,Tr}}^{1}}
            \\
            \leq&\mathcal{E}(v_{k},S_{k},\varphi_{k})
        +h\langle f_{k+1}, v_{k+1}\rangle_{H^{1}}.
        \end{aligned}
    \end{equation}
\end{lem}

\begin{proof}
    {We follow}~\cite[Lemma 4.3]{arxiv:2509.25508}: By formally testing \eqref{time discrete problem} with the solution $(v_{k+1},\allowbreak{S_{k+1}},\allowbreak\varphi_{k+1},\allowbreak\mu_{k+1})$ one obtains the discrete energy-dissipation estimate \eqref{discrete total energy estimate}. Indeed, the existence of the solution $(v_{k+1},{S_{k+1}},\varphi_{k+1},\mu_{k+1})$ can be obtained with the aid of Schaefer's fixed point theorem by reformulating the time-discrete problem \eqref{time discrete problem} in terms of two operators $\mathscr{L}_{k},\mathscr{F}_{k}:X\to Y,$ where
    the space $Y$ is given as
	\begin{equation*}
	    Y:=\left(H_{0,\mathrm{div}}^{1}(\Omega)\right)^{\prime}
        \times\left(H_{\mathrm{sym,Tr}}^{1}(\Omega)\right)^{\prime}
        \times L^{2}(\Omega)
        \times L^{2}(\Omega).
	\end{equation*}
    One can see that $w=(v_{k+1},{S_{k+1}},\varphi_{k+1},\mu_{k+1})$  is a weak solution of \eqref{time discrete problem}, if and only if 
    \begin{equation*}
    \mathscr{L}_{k}(w)-\mathscr{F}_{k}(w)=0\,. 
    \end{equation*}
    Hereby, the two operators can be defined as in the proof of \cite[Lemma 4.3]{arxiv:2509.25508} with the only difference that the third entry of operator $\mathscr{L}_{k}$ now is given as $-\divge{m(\varphi_{k})\nabla\mu}+\int_{\Omega}\mu\dd{x}$. 
    To apply Schaefer's fixed point theorem it has to be shown that $\mathscr{L}_{k}$ is invertible with a continuous inverse and that $\mathscr{F}_{k}$ is continuous and bounded. These properties can be verified by following the lines of the proof of \cite[Lemma 4.3]{arxiv:2509.25508} and the fixed point theorem can then be applied to the operator $\mathscr{F}_{k}\circ\mathscr{L}_{k}^{-1}$ as in the proof of \cite[Lemma 4.3]{arxiv:2509.25508}.  
\end{proof}

\subsection{Existence of energy-variational solutions for system \eqref{sys:two phase} with $\gamma>0$}

In this subsection, we show the existence of energy-variational solutions for the regularized system~\eqref{sys:two phase} with $\gamma>0$ and a logarithmic phase-field potential, i.e., such that $W=W_{\mathrm{dw}}+W_{\mathrm{sg}}$. Most of the steps can be carried out analogously to the proof of~\cite[Theorem 4.4]{arxiv:2509.25508}. Therefore, here it is sufficient to construct the auxiliary energy variable $E:[0,T]\to\mathbb{R}$ as an upper bound to the system energy along solutions and verify its convergence. 
\begin{theorem}\label{exitence evs regulairzed system}
    Let $\gamma>0$. Let Assumptions~\ref{ASM:domain}-\ref{ASM:singular potential} be satisfied. Then, there exists an energy-variational solution $(v,S,\varphi,\mu,E)$ of type $\mathcal{K}\equiv0$ to the system~\eqref{sys:two phase} with $\gamma>0$ in the sense of Definition~\ref{defi: EVS} with $E(0)=\mathcal{E}(v_{0},S_{0},\varphi_{0})$. Moreover, the phase-field variable $\varphi$ takes values in $(-1,1)$ a.e. in $\Omega\times(0,T)$.
\end{theorem}
\begin{proof}
In the following we verify the properties of an energy-variational solution stated in Items 1 and 2 of Definition \ref{defi: EVS}.
\\[1ex]
\textbf{Step 1: Selection of convergent subsequences and proof of regularity result \eqref{EVS-reg}.}
\begin{subequations}
For a function $u:[0,T]\to X$ with a Banach space $X$ we define the piecewise constant interpolant and piecewise affine interpolant {in time as}    
\begin{align}
    u^{N}(t)&:=u_{k+1}\text{ for }t\in(t_{k},t_{k+1}]\text{ and }{u^{N}(0)}:=u_{0},
    \\
    \overline{u}^{N}(t)&:=\frac{t-t_{k}}{h}u_{k+1}+\frac{t_{k+1}-t}{h}u_{k}\text{ for }t\in[t_{k},t_{k+1}].
\end{align}
Moreover, we also define
\begin{align}
    \partial_{t,h}^{+}u^{N}(t)&:=\frac{1}{h}(u^{N}(t+h)-u^{N}(t)),
	\partial_{t,h}^{-}u^{N}(t):=\frac{1}{h}(u^{N}(t)-u^{N}(t-h)),
        \\
	u_{h}^{N}(t)&:=u^{N}(t-h).
\end{align}
\end{subequations}
By following~\cite[Theorem 4.4]{arxiv:2509.25508}, we first obtain the following convergence results along a not relabelled subsequence as $N\to\infty$:  
\begin{subequations}
\label{conv-sols}
    \begin{align}
		v^{N}&\rightharpoonup v\mbox{ in }L^{2}(0,T^{\prime};H^{1}(\Omega)),
		\label{weak convergence of piecewise constant interpolation velocity space derivative}\\
		v^{N}&\overset{*}{\rightharpoonup}v\mbox{ in }L^{\infty}(0,T^{\prime};L_{\mathrm{div}}^{2}(\Omega)),
		\label{weak convergence of piecewise constant interpolation velocity}\\
		S^{N}&\rightharpoonup S\mbox{ in }L^{2}(0,T^{\prime};H^{1}(\Omega)),
		\label{weak convergence of piecewise constant interpolation stress space derivative}\\
		S^{N}&\overset{*}{\rightharpoonup}S\mbox{ in }L^{\infty}(0,T^{\prime};L_{\mathrm{sym,Tr}}^{2}(\Omega)),
		\label{weak convergence of piecewise constant interpolation stress}\\
		\varphi^{N}&\overset{*}{\rightharpoonup}\varphi\mbox{ in }L^{\infty}(0,T^{\prime};H^{1}(\Omega)),
		\label{weak convergence of piecewise constant interpolation phase-field}\\
        \varphi^{N}&\rightharpoonup\varphi\mbox{ in }L^{2}(0,T^{\prime};H^{2}(\Omega)),
        \label{weak convergence of piecewise constant interpolation phase-field H2}\\
		\mu^{N}&\rightharpoonup\mu\mbox{ in }L^{2}(0,T^{\prime};H^{1}(\Omega)),
		\label{weak convergence of piecewise constant interpolation chemical potential}\\
        \varphi^{N}&\to\varphi\mbox{ in }L^{2}(0,T^{\prime};H^{1}(\Omega)),
		\label{strong convergence of piecewise constant interpolation phase-field}\\
        v^{N}&\to v\mbox{ in }L^{2}(0,T^{\prime};L^{2}(\Omega)).
		\label{strong convergence of piecewise constant interpolation velocity}
	\end{align}
\end{subequations}
for all $0<T^{\prime}<T$. Moreover, for the right-hand side 
$f\in L_{\mathrm{loc}}^{2}(0,T^{\prime};H^{-1}(\Omega)^{3})$ in \eqref{sys: momentum}, by the construction of the piecewise constant interpolant $f^{N}$, we have 
\begin{align}
    \lVert f^{N}\rVert_{L^{2}(0,T^{\prime};H^{-1}(\Omega))^{3}}&\leq\lVert f\rVert_{L^{2}(0,T^{\prime};H^{-1}(\Omega))^{3}},
    \\
    f^{N}&\to f\text{ in }L^{2}(0,T^{\prime};H^{-1}(\Omega)^{3}).\label{strong convergence of piecewise constnat out loeading}
\end{align}
In addition, we define the following interpolants for the system energy evaluated along the time-discrete solutions obtained in 
{Lemma~\ref{Existence of solution to the time discrete problem}:}  
    \begin{align}
        E^{N}(t)&:=\mathcal{E}(v^{N}(t),S^{N}(t),\varphi^{N}(t)),
        \\
        \overline{E}^{N}(t)&:=\frac{t-t_{k}}{h}\mathcal{E}(v_{k+1},S_{k+1},\varphi_{k+1})
        +\frac{t_{k+1}-t}{h}\mathcal{E}(v_{k},S_{k},\varphi_{k})
        \text{ for }t\in[t_{k},t_{k+1}].
    \end{align}
    One can see $E^{N}:[0,T]\to\mathbb{R}$ is a piecewise constant function. Furthermore, from~\eqref{discrete total energy estimate} we have 
    \begin{equation*}
        \begin{aligned}
            &E^{N}(t_{k+1})
            +\int_{t_{k}}^{t_{k+1}}\frac{2\nu_{1}}{k_{\Omega}}\lVert v_{k+1}\rVert_{H^{1}}^{2}\dd{\tau}
            \\
            &\leq E^{N}(t_{k+1})
            +\int_{t_{k}}^{t_{k+1}}\int_{\Omega}2\nu(\varphi_{k})\abs{\sym{\nabla v_{k+1}}}^{2}\dd{x}\dd{\tau}
            \\
            &\leq E^{N}(t_{k})
            +\int_{t_{k}}^{t_{k+1}}\langle f^{N},v_{k+1}\rangle_{H^{1}}\dd{\tau}
            \\
            &\leq E^{N}(t_{k})+\frac{k_{\Omega}}{\nu_{1}}\int_{t_{k}}^{t_{k+1}}\lVert f^{N}\rVert_{H^{-1}}^{2}\dd{\tau}+\int_{t_{k}}^{t_{k+1}}\frac{\nu_{1}}{k_{\Omega}}\lVert v_{k+1}\rVert_{H^{1}}^{2}\dd{\tau}\,.
        \end{aligned}
    \end{equation*}
    Here we used Korn's inequality, {see} \cite[Theorem 6.3-4]{zbMATH07478418}, with Korn's constant $k_{\Omega}$ in view of the homogeneous Dirichlet boundary condition for $v_{k+1}$ together with Assumption~\ref{ASM:material paramaters} for $\nu$ to deduce the first estimate, \eqref{discrete total energy estimate} for the second estimate, and finally  Young's inequality to conclude the third estimate. Therefore, by rearranging terms on both sides, we arrive at
    \begin{equation}\label{bound on upper bound energy}
        E^{N}(t_{k+1})\leq E^{N}(t_{k})+C\int_{t_{k}}^{t_{k+1}}\lVert f^{N}\rVert_{H^{-1}}^{2}\dd{\tau}.
    \end{equation}
    Observe that, by rearranging terms and repeating this estimate for $k,k-1,{k-2,\ldots}$, \eqref{bound on upper bound energy} yields 
    \begin{equation*}
    \begin{split}
       E^{N}(t_{k+1})-C\int_{{0}}^{t_{k+1}}\lVert f^{N}\rVert_{H^{-1}}^{2}\dd{\tau}
       &\leq E^{N}(t_{k})-C\int_{{0}}^{t_{k}}\lVert f^{N}\rVert_{H^{-1}}^{2}\dd{\tau} \\
       &\leq E^{N}(t_{k-1})-C\int_{{0}}^{t_{k-1}}\lVert f^{N}\rVert_{H^{-1}}^{2}\dd{\tau}\leq\ldots \leq E^{N}(0)\,.
       \end{split}
    \end{equation*}
    This implies that the function $E^{N}(t)-C\int_{0}^{t}\lVert f^{N}\rVert_{H^{-1}}^{2}\dd{\tau}$ is  monotonically decreasing in $t$.  Clearly, $\int_{0}^{t}\lVert f^{N}\rVert_{H^{-1}}^{2}\dd{\tau}$ is a monotonically increasing function in $t$.  Thus, the function $E^{N}$ is given as the sum of the  monotonically decreasing function $E^{N}(t)-C\int_{0}^{t}\lVert f^{N}\rVert_{H^{-1}}^{2}\dd{\tau}$ and the monotonically increasing function $C\int_{0}^{t}\lVert f^{N}\rVert_{H^{-1}}^{2}\dd{\tau}$. In addition, by construction, $E^{N}(0)$ is uniformly bounded and $\int_{0}^{t}\lVert f^{N}\rVert_{H^{-1}}^{2}\dd{\tau}$ is uniformly bounded for a.e.\ $t\in(0,T^{\prime})$ for any $T'<T$ by Assumption \ref{ASM:initial data}. {Hence} $E^{N}(t)$ is uniformly bounded in $BV([0,T^{\prime}])$ and we obtain\begin{subequations}
    \begin{equation}
    \label{conv-E-BV}
        E^{N}\overset{*}{\rightharpoonup}E\text{ in }BV([0,T^{\prime}])
    \end{equation}
    for some $E\in BV_{\mathrm{loc}}([0,T])$. By the compact embedding of $BV$ spaces, we further have
    \begin{align}
        E^{N}&\to E\text{ in }L^{1}([0,T^{\prime})),\\
        E^{N}(t)&\to E(t)\text{ for a.e. }t\in[0,T^{\prime}]
    \end{align}
    along a not relabelled subsequence.  
    Since \eqref{bound on upper bound energy} provides an $L^\infty$-bound on $(E^N)_N$, we can improve the strong $L^1$-convergence to  
    \begin{equation}
        E^{N}\to E\text{ in }L^{p}([0,T^{\prime}))
        \;\text{ for any $p\in[1,\infty)$.}
    \end{equation}
     Thanks to the uniformly bounded total variation, we have
    \begin{equation}
        \int_{0}^{T^{\prime}}\abs{\overline{E}^{N}-E^{N}}\dd{\tau}\leq Ch\to0.
    \end{equation}
    Therefore, we derive
    \begin{align}
        \overline{E}^{N}&\to E\quad\,\text{ in }L^{p}([0,T^{\prime}))\;\;\text{ for any }p\in[1,\infty),
        \label{strong convergence of piecewise linear auxiliary energy}\\
        \overline{E}^{N}(t)&\to E(t)\text{ for a.e.\ }t\in[0,T^{\prime}].
        \label{ae convergence of piecewise linear auxiliary energy}
    \end{align}\end{subequations}
    For any $\zeta\in C_{0}^{\infty}([0,T))$ such that $\zeta(t)\geq0$ for all $t\in(0,T)$, we observe that 
    \begin{align*}
        \int_{0}^{T}\zeta(t)E(t)\dd{t}=&\lim_{N\to\infty}\int_{0}^{T}\zeta(t)E^{N}(t)\dd{t}
        \\
        =&\lim_{N\to\infty}\int_{0}^{T}\zeta(t)\mathcal{E}(v^{N}(t),S^{N}(t),\varphi^{N}(t))\dd{t}
        \geq\int_{0}^{T}\zeta(t)\mathcal{E}(v(t),S(t),\varphi(t))\dd{t},
    \end{align*}
    where the last inequality {holds thanks} to the convergence results \eqref{weak convergence of piecewise constant interpolation velocity space derivative}-\eqref{strong convergence of piecewise constant interpolation velocity}. This implies 
    \begin{equation*}
    E\geq\mathcal{E}(v,S,\varphi)\quad\text{a.e.\ in $(0,T)$.}
    \end{equation*}
This result together with the convergence results \eqref{conv-E-BV} and \eqref{conv-sols} provides the regularity statement \eqref{EVS-reg} in Definition \ref{defi: EVS}.  
\paragraph{Step 2: Proof of a discrete version of estimate \eqref{EVS inequality} using \eqref{time discrete problem} and \eqref{discrete total energy estimate}.} To this end, we take $\phi\in C_{0}^{\infty}([0,T^{\prime}))$ with $\phi\geq0$, $\tilde{v}\in C_{0,\mathrm{div}}^{\infty}(\Omega\times[0,T))$, $\tilde{S}\in C_{0,\mathrm{sym,Tr}}^{\infty}(\Omega\times[0,T))$, $\tilde{\varphi}\in C_{0}^{\infty}(\Omega\times[0,T))$, and $\tilde{\mu}\in C_{0}^{\infty}(\Omega\times[0,T))$. Inserting $\Phi_{k}:={\int_{t_{k}}^{t_{k+1}}-\phi\tilde{v}\dd{t}}$ into~\eqref{discrete velocity} and summing over $k=\{0,1,\ldots,N\}$ provides
\begin{subequations}
\begin{equation}\label{tested discrete velocity}
    \begin{aligned}
    -&\int_{0}^{T^{\prime}}\phi\int_{\Omega}\partial_{t}(\overline{\rho v}^{N})\cdot\tilde{v}\dd{x}\dd{\tau}
    +\int_{0}^{T^{\prime}}\phi\int_{\Omega}\rho_{h}^{N}v^{N}\otimes v^{N}:\nabla\tilde{v}\dd{x}\dd{\tau}
    \\
    -&\int_{0}^{T^{\prime}}\phi\int_{\Omega}2\nu(\varphi_{h}^{N})\sym{\nabla v^{N}}:\sym{\nabla\tilde{v}}+\eta(\varphi_{h}^{N})S^{N}:\nabla\tilde{v}\dd{x}\dd{\tau}
    \\
    +&\int_{0}^{T^{\prime}}\phi\int_{\Omega}v^{N}\otimes J^{N}:\nabla\tilde{v}
    +\mu^{N}\nabla\varphi_{h}^{N}\cdot\tilde{v}\dd{x}\dd{\tau}
    \\
    =-&{\int_{0}^{T^{\prime}}\phi\langle f^{N},\tilde{v} \rangle_{H^{1}}\dd{\tau}}.
    \end{aligned}
\end{equation}
Similarly, inserting $\Psi_{k}:=\int_{t_{k}}^{{t_{k+1}}}-\phi\tilde{S}\dd{t}$ into~\eqref{discrete stress} and summing over $k=\{0,1,\ldots,N\}$ yields
\begin{equation}\label{tested discrete stress}
    \begin{aligned}
        -&\int_{0}^{T^{\prime}}\phi\int_{\Omega}\partial_{t}\overline{S}^{N}:\tilde{S} \dd{x}\dd{\tau}
		-\int_{0}^{T^{\prime}}\phi\int_{\Omega}(v^{N}\cdot\nabla S^{N}):\tilde{S} \dd{x}\dd{\tau} 
		\\
		-&\int_{0}^{T^{\prime}}\phi\int_{\Omega}(S^{N}\skw{\nabla v^{N}}-\skw{\nabla v^{N}}S^{N}):\tilde{S} \dd{x}\dd{\tau}
		-\int_{0}^{T^{\prime}}\phi\langle \xi^{N},\tilde{S} \rangle_{H_{\mathrm{sym,Tr}}^{1}}\dd{\tau} 
		\\
		-&\int_{0}^{T^{\prime}}\phi\int_{\Omega}\gamma\nabla S^{N}\threedotsbin\nabla\tilde{S} \dd{x}\dd{\tau}
		 =-\int_{0}^{T^{\prime}}\phi\int_{\Omega}\eta(\varphi_{h}^{N})\sym{\nabla v^{N}}:\tilde{S} \dd{x}\dd{\tau}, 
    \end{aligned}
\end{equation}
where $\xi^{N}\in\partial\tilde{\mathcal{P}}(\varphi_{h}^{N};S^{N})$. Moreover, multiplying ~\eqref{discrete first CH} by $-\phi\tilde{\varphi}$ and integrating over space and time provides
\begin{equation}\label{tested discrete first CH}
    \begin{aligned}
        -\int_{0}^{{T}^{\prime}}\phi\int_{\Omega}\partial_{t}\overline{\varphi}^{N}\tilde{\varphi}\dd{x}\dd{\tau}
        -\int_{0}^{T^{\prime}}\phi\int_{\Omega}(v^{N}\mathbin{\cdot}\nabla\varphi_{h}^{N})\tilde{\varphi}\dd{x}\dd{\tau}
        =\int_{0}^{T^{\prime}}\phi\int_{\Omega}m(\varphi_{h}^{N})\nabla\mu^{N}\cdot\nabla\tilde{\varphi}\dd{x}\dd{\tau}.
    \end{aligned}
\end{equation}
Similarly, multiplying~\eqref{discrete second CH} by $-\phi\tilde{\mu}$, and integrating over space and time gives
\begin{equation}\label{tested discrete second CH}
    \begin{aligned}
        -\int_{0}^{T^{\prime}}\phi\int_{\Omega}\mu^{N}\tilde{\mu}+\kappa\frac{\varphi^{N}+\varphi_{h}^{N}}{2}\tilde{\mu}\dd{x}\dd{\tau}
        =-\int_{0}^{T^{\prime}}\phi\int_{\Omega}\left(-\Delta\varphi^{N}+W_{\kappa}^{\prime}(\varphi^{N})\right)\tilde{\mu}\dd{x}\dd{\tau}.
    \end{aligned}
\end{equation}
In addition, multiplying~\eqref{discrete total energy estimate} by $\phi$, and integrating over time gives
\begin{equation*}
        \begin{aligned}
&\int_{0}^{{T}^{\prime}}\phi\partial_{t}\overline{E}^{N}\dd{\tau}
        +\int_{0}^{T^{\prime}}\phi\int_{\Omega}
        2\nu(\varphi_{h}^{N})\abs{\sym{\nabla v^{N}}}^{2}
        +\gamma\abs{\nabla S^{N}}^{2}
        +m(\varphi_{h}^{N})\abs{\nabla\mu^{N}}^{2}
        \dd{x}\dd{\tau}\\
&\phantom{\int_{0}^{{T}^{\prime}}
\phi\partial_{t}\overline{E}^{N}\dd{\tau}}
        \,+\int_{0}^{T^{\prime}}\phi\langle \xi^{N},S^{N} \rangle_{H_{\mathrm{sym,Tr}}^{1}}\dd{\tau}
        \\
        &\leq \int_{0}^{T^{\prime}}\phi\langle f^{N}, v^{N}\rangle_{H^{1}}\dd{\tau}\,.
        \end{aligned}
    \end{equation*}
    Here $\xi^{N}\in\partial\tilde{\mathcal{P}}(\varphi_{h}^{N};S^{N}),$ which implies 
\begin{equation*}
    \begin{aligned}
        \langle \xi^{N}, S^{N}-\tilde{S} \rangle_{H_{\mathrm{sym,Tr}}^{1}}
        \geq
        \widetilde{\mathcal{P}}(\varphi_{h}^{N};S^{N})-\widetilde{\mathcal{P}}(\varphi_{h}^{N};\tilde{S})
        ={\mathcal{P}}(\varphi_{h}^{N};S^{N})-{\mathcal{P}}(\varphi_{h}^{N};\tilde{S}).
    \end{aligned}
\end{equation*}
    Using this information in the estimate above we arrive at 
\begin{equation}
\label{tested discrete total energy estimate}
        \begin{aligned}
        &\int_{0}^{{T}^{\prime}}\phi\partial_{t}\overline{E}^{N}\dd{\tau}+
        \int_{0}^{T^{\prime}}\phi\int_{\Omega}
        2\nu(\varphi_{h}^{N})\abs{\sym{\nabla v^{N}}}^{2}
        +\gamma\abs{\nabla S^{N}}^{2}
        +m(\varphi_{h}^{N})\abs{\nabla\mu^{N}}^{2}
        \dd{x}\dd{\tau}\\
        &\phantom{\int_{0}^{{T}^{\prime}}
\phi\partial_{t}\overline{E}^{N}\dd{\tau}}
        \,
        +\int_{0}^{T^{\prime}}\phi
        \left(\mathcal{P}(\varphi_{h}^{N};S^{N})
    -\mathcal{P}(\varphi_{h}^{N};\tilde{S})\right)
        \dd{\tau}
        \\
        &\leq \int_{0}^{T^{\prime}}\phi\langle f^{N}, v^{N}\rangle_{H^{1}}\dd{\tau}.
        \end{aligned}
    \end{equation} 
\end{subequations}
Hence, summing up \eqref{tested discrete velocity}-\eqref{tested discrete total energy estimate} and applying integration by parts in time, we obtain
\begin{equation}\label{time discrete EVS}
    \begin{aligned}
    -&\int_{0}^{T^{\prime}}\phi^{\prime}
    \left(\overline{E}^{N}(\tau)-\langle \overline{\rho v}^{N},\tilde{v}\rangle_{L^{2}}-\langle \overline{S}^{N},\tilde{S} \rangle_{L^{2}}-\langle \overline{\varphi}^{N},\tilde{\varphi} \rangle_{L^{2}}\right)
    \dd{\tau}
        \\
    +&\int_{0}^{T^{\prime}}\phi\int_{\Omega}
    2\nu(\varphi_{h}^{N})\abs{\sym{\nabla v^{N}}}^{2}
    +\gamma\abs{\nabla S^{N}}^{2}
    +m(\varphi_{h}^{N})\abs{\nabla\mu^{N}}^{2}\dd{x}\dd{\tau}
        \\
    +&\int_{0}^{T^{\prime}}\phi\left(\mathcal{P}(\varphi_{h}^{N};S^{N})
    -\mathcal{P}(\varphi_{h}^{N};\tilde{S})
    \right)\dd{\tau}
        \\
    +&\int_{0}^{T^{\prime}}\phi\int_{\Omega}
    \overline{\rho v}^{N}\cdot\partial_{t}\tilde{v}
    +\rho_{h}^{N} v^{N}\otimes v^{N}:\nabla\tilde{v}
    +v^{N}\otimes J^{N}:\nabla\tilde{v}
    \dd{x}\dd{\tau}
        \\
    -&\int_{0}^{T^{\prime}}\phi\int_{\Omega}
    2\nu(\varphi_{h}^{N})\sym{\nabla v^{N}}:\sym{\nabla\tilde{v}}
    +\eta(\varphi_{h}^{N})S^{N}:\sym{\nabla\tilde{v}}
    -\mu^{N}\nabla\varphi_{h}^{N}\cdot\tilde{v}\dd{x}\dd{\tau}
        \\
    +&\int_{0}^{T^{\prime}}\phi\int_{\Omega}
    \overline{S}^{N}:\partial_{t}\tilde{S}
    +S^{N}\otimes v^{N}\threedotsbin\nabla\tilde{S}
    -\left( S^{N}\skw{\nabla v^{N}} - \skw{\nabla v^{N}}S^{N}\right):\tilde{S}
    \dd{x}\dd{\tau}
        \\
    -&\int_{0}^{T^{\prime}}\phi\int_{\Omega}
    \gamma \nabla S^{N}\threedotsbin \nabla\tilde{S} 
    -\eta(\varphi_{h}^{N})\sym{\nabla v^{N}}:\tilde{S}\dd{x}\dd{\tau}
        \\
    +&\int_{0}^{T^{\prime}}\phi\int_{\Omega}
    \overline{\varphi}^{N}\partial_{t}\tilde{\varphi}
    -(v^{N}\cdot\nabla\varphi_{h}^{N})\tilde{\varphi}
    -m(\varphi_{h}^{N})\nabla \mu^{N}:\nabla\tilde{\varphi}\dd{x}\dd{\tau}
        \\
    -&\int_{0}^{T^{\prime}}\phi\int_{\Omega}
    \mu^{N}\tilde{\mu}+\kappa\frac{\varphi^{N}+\varphi_{h}^{N}}{2}-(-\Delta\varphi^{N}+ W_{\kappa}^{\prime}(\varphi^{N}))\tilde{\mu}\dd{x}\dd{\tau}
        \\
    \leq&
    \int_{\Omega}\rho_{0}v_{0}\cdot\tilde{v}(0)+S_{0}:\tilde{S}(0)+\varphi_{0}^{N}\tilde{\varphi}(0)\dd{x}
    +\int_{0}^{T^{\prime}}\phi\langle f^{N},v^{N}-\tilde{v}\rangle_{H^{1}}\dd{\tau}\,.
    \end{aligned}
\end{equation}
\paragraph{Step 3: Proof of estimate \eqref{EVS inequality} by passing to the limit $N\to\infty$ in \eqref{time discrete EVS}.} 
Now, for the limit passage $N\to\infty$ in \eqref{time discrete EVS}, observe that 
\begin{equation}
    \int_{\Omega}\left( S^{N}\skw{\nabla v^{N}} - \skw{\nabla v^{N}}S^{N}\right):\tilde{S}\dd{x}
\end{equation}
contains the terms of the form
\begin{equation}
    \int_{\Omega} S_{ik}^{N}(\partial_{x_j}v_{k}^{N})\tilde{S}_{ij}\dd{x}.
\end{equation}
Hence, with the help of integration by parts and using the Dirichlet boundary condition, we have
\begin{equation}
\label{trick-S-skw}
    \int_{\Omega} S_{ik}^{N}(\partial_{x_j}v_{k}^{N})\tilde{S}_{ij}\dd{x}
    =-\int_{\Omega} (\partial_{x_j}S_{ik}^{N}) v_{k}^{N}\tilde{S}_{ij}\dd{x}
    -\int_{\Omega} S_{ik}^{N} v_{k}^{N}(\partial_{x_{j}}\tilde{S}_{ij})\dd{x}.
\end{equation}
Thus, the strong $L^2$-convergence \eqref{strong convergence of piecewise constant interpolation velocity} of $(v^N)_N$ 
and the weak convergence result {\eqref{weak convergence of piecewise constant interpolation stress space derivative}} for $(S^N)_N$ are sufficient to pass to the limit on the right-hand side of \eqref{trick-S-skw}. Then, the integration by parts can be reversed in the limit.  
Therefore, also thanks to convergence results \eqref{weak convergence of piecewise constant interpolation velocity space derivative}-\eqref{strong convergence of piecewise constant interpolation velocity},~\eqref{strong convergence of piecewise constnat out loeading}, and~\eqref{strong convergence of piecewise linear auxiliary energy}-\eqref{ae convergence of piecewise linear auxiliary energy}, we can pass to the limit $N\to\infty$ in~\eqref{time discrete EVS} to arrive at
\begin{equation}
    \begin{aligned}
    -&\int_{0}^{T^{\prime}}\phi^{\prime}\left(E(\tau)-\langle \rho v,\tilde{v}\rangle_{L^{2}}-\langle S,\tilde{S} \rangle_{L^{2}}-\langle \varphi,\tilde{\varphi} \rangle_{L^{2}}\right)\dd{\tau}
        \\
    +&\int_{0}^{T^{\prime}}\phi\int_{\Omega}
    2\nu(\varphi)\abs{\sym{\nabla v}}^{2}
    +\gamma\abs{\nabla S}^{2}
    +m(\varphi)\abs{\nabla\mu}^{2}\dd{x}\dd{\tau}
    +\int_{0}^{T^{\prime}}\phi\left(\mathcal{P}(\varphi;S)-\mathcal{P}(\varphi;\tilde{S})\right)\dd{\tau}
        \\
    +&\int_{0}^{T^{\prime}}\phi\int_{\Omega}
    \rho v\cdot\partial_{t}\tilde{v}
    +\rho v\otimes v:\nabla\tilde{v}
    +v\otimes J:\nabla\tilde{v}
    -2\nu(\varphi)\sym{\nabla v}:\sym{\nabla\tilde{v}}
    -\eta(\varphi)S:\sym{\nabla\tilde{v}}\dd{x}\dd{\tau}
        \\
    +&\int_{0}^{T^{\prime}}\phi\int_{\Omega}
    \nabla\varphi\otimes\nabla\varphi:\nabla\tilde{v}\dd{x}\dd{\tau}
        \\
    +&\int_{0}^{T^{\prime}}\phi\int_{\Omega}
    S:\partial_{t}\tilde{S}
    +S\otimes v\threedotsbin\nabla\tilde{S}
    -\left( S\skw{\nabla v} - \skw{\nabla v}S\right):\tilde{S}-\gamma \nabla S\threedotsbin \nabla\tilde{S} 
    +\eta(\varphi)\sym{\nabla v}:\tilde{S}\dd{x}\dd{\tau}
        \\
    +&\int_{0}^{T^{\prime}}\phi\int_{\Omega}
    \varphi\partial_{t}\tilde{\varphi}
    -(v\cdot\nabla\varphi)\tilde{\varphi}
    -m(\varphi)\nabla \mu:\nabla\tilde{\varphi}\dd{x}\dd{\tau}
        \\
    -&\int_{0}^{T^{\prime}}\phi\int_{\Omega}
    \mu\tilde{\mu}-(-\Delta\varphi+ W^{\prime}(\varphi))\tilde{\mu}\dd{x}\dd{\tau}
        \\
    \leq&\int_{\Omega}\rho_{0}v_{0}\cdot\tilde{v}(0)+S_{0}:\tilde{S}(0)+\varphi_{0}\tilde{\varphi}(0)\dd{x}
    +\int_{0}^{T^{\prime}}\phi\langle f,v-\tilde{v}\rangle_{H^{1}}\dd{\tau}\,.
    \end{aligned}
\end{equation}
By applying Lemma~\ref{Lem: equivalence weak form}, we obtain~\eqref{EVS inequality} pointwise in time for all $s<t\in(0,T)$. 

Altogether, we have shown that the limit quintuplet $(v,S,\varphi,\mu,E)$ extracted by convergence results \eqref{conv-sols} and \eqref{conv-E-BV} satisfies inequality \eqref{EVS inequality} that characterizes an energy-variational solution.

Moreover, by following the lines of \cite[Theorem 4.4]{arxiv:2509.25508}, we further infer that $\varphi$ takes values in $(-1,1)$ a.e. in $\Omega\times(0,T)$.

\paragraph{Step 4: Attainment of the initial conditions.} It remains to show that the components $(v,S,\varphi,E)$ of the energy-variational solution obtained by convergence results \eqref{conv-sols} and \eqref{conv-E-BV} indeed attain the initial condition at $t=0$  for the given initial data $(v_{0},S_{0},\varphi_{0},\mathcal{E}(v_{0},S_{0},\varphi_{0}))$. 
Indeed, after possibly redefining the functions on a set 
of measure zero, we obtain directly from~\eqref{EVS inequality} that
\begin{equation}\label{right limit of initial datas}
\begin{aligned}
    &\lim_{t\to0}\left(E(t)-\langle \rho(t) v(t),\tilde{v}(t)\rangle_{L^{2}}-\langle S(t),\tilde{S}(t) \rangle_{L^{2}}-\langle \varphi(t),\tilde{\varphi}(t) \rangle_{L^{2}}\right)
    \\
    =&\mathcal{E}(v_{0},S_{0},\varphi_{0})-\langle \rho_{0} v_{0},{\tilde{v}(0)}\rangle_{L^{2}}-\langle {S_{0},\tilde{S}(0)} \rangle_{L^{2}}-\langle {\varphi_{0},\tilde{\varphi}(0)} \rangle_{L^{2}}.
\end{aligned}
\end{equation}
Therefore, by choosing $\tilde{v}\equiv0$, $\tilde{S}\equiv0$, $\tilde{\varphi}\equiv0$, we find
\begin{equation}
    E(0_{+})\leq\mathcal{E}(v_{0},S_{0},\varphi_{0})
\end{equation}
such that $E(0)=\mathcal{E}(v_{0},S_{0},\varphi_{0})$. Then, by multiplying~\eqref{right limit of initial datas} by $\alpha>0$ and choosing $\tilde{v}\equiv0$, $\tilde{\varphi\equiv0}$, and $\tilde{S}=\alpha^{-1}\Psi$ for $\Psi\in C_{0,\mathrm{sym,Tr}}^{\infty}(\Omega\times[0,T)$, we have
\begin{equation*}
    \lim_{t\to0_{+}}\langle S(t),\Psi\rangle_{L^{2}}\leq\langle S_{0},\Psi\rangle_{L^{2}}.
\end{equation*}
The same inequality holds for $\Psi$ replaced by $-\Psi$ which implies
\begin{equation*}
    \langle S(0_{+})-S_{0},\Psi\rangle_{L^{2}}=0.
\end{equation*}
Therefore, we infer $S(0_{+})=S_{0}$ in $L^{2}(\Omega)$ with the help of~\cite[Lemma 2.2]{ALREVS}, and, by continuous extension, we conclude $S(0)=S_{0}$. Analogously, one can show $\varphi(0)=\varphi_{0}$. Moreover, with the same argument we derive 
\begin{equation*}
    \lim_{t\to0_{+}}\langle \rho(t)v(t),\Phi\rangle_{L^{2}}\leq\langle \rho_{0}v_{0},\Phi\rangle_{L^{2}}.
\end{equation*}
Using that $\rho_{0}>0$, we also obtain $v(0_{+})=v_{0}$, and, therefore, we conclude $v(0)=v_{0}$. This finishes the proof of Theorem \ref{exitence evs regulairzed system}. 
\end{proof}
\begin{remark}
    From Proposition~\ref{recover strong GT}, one can see that an energy-variational solution obtained from Theorem~\ref{exitence evs regulairzed system} also satisfies
    \begin{equation}
        \mu=-\Delta\varphi+W^{\prime}(\varphi).
    \end{equation}
    Moreover, from Propositions~\ref{recover weak momentum balance} and~\ref{recover weak CH}, one can observe that the weak formulations~\eqref{weak formualtion: momentum} and~\eqref{weak formualtion: CH} of the momentum balance and the phase-field evolution law are satisfied as well.
\end{remark}

\subsection{Existence of energy-variational solutions for system \eqref{sys:two phase} with  $\gamma=0$}

Repeating the arguments for the proof of \cite[Theorem 5.1]{arxiv:2509.25508}, we can perform the limit passage $\gamma\to0$ and thus obtain the existence of energy-variational solutions for the {nonregularized} system~\eqref{sys:two phase} with $\gamma=0$ and a logarithmic phase-field potential, {i.e.,}\ with $W=W_{\mathrm{dw}}+W_{\mathrm{sg}}.$ 
\begin{theorem}\label{exitence evs non-regulairzed system}
Let $\gamma=0$ and the phase-field potential $W=W_{\mathrm{dw}}+W_{\mathrm{sg}}$ in \eqref{def-W}. Let the Assumptions~\ref{ASM:domain}-\ref{ASM:Double well potential} be satisfied. Then, there exists an energy-variational solution $(v,S,\varphi,\mu,E)$ of type $\mathcal{K}$ to system~\eqref{sys:two phase} with $\gamma=0$ in the sense of Definition~\ref{defi: EVS} with $E(0)=\mathcal{E}(v_{0},S_{0},\varphi_{0})$ and where the regularity weight is given by
\begin{equation}
    \mathcal{K}(\tilde{S}):=\frac{k_{\Omega}^{2}}{\nu_{1}}\lVert \tilde{S} \rVert_{L^{\infty}}^{2}
        \label{regularity weight for non-regularized}
\end{equation}
for all $(\tilde{v},\tilde{S},\tilde{\varphi},\tilde{\mu})\in\mathfrak{T}$. Here, the constant $k_{\Omega}>0$ is the constant from Korn's inequality and $\nu_{1}$ is the constant from Assumption~\ref{ASM:material paramaters}. Moreover, the phase-field variable $\varphi$ takes values in $(-1,1)$ a.e. in $\Omega\times(0,T)$.
\end{theorem}
\begin{proof}
    This theorem can be shown by repeating the arguments of the proof of~\cite[Theorem 5.1]{arxiv:2509.25508}.
    More precisely, the main issue in the limit passage $\gamma\to0$ concerns the term
    \begin{equation*}
        (S\skw{\nabla v}-\skw{\nabla v}S):\tilde{S},
    \end{equation*}
    which is analogous to the corresponding term in~\cite[Theorem 5.1]{arxiv:2509.25508} and can be treated in the same way. Indeed, thanks to {the} chosen regularity weight~\eqref{regularity weight for non-regularized}, the integral
    \begin{equation*}
        \int_{\Omega}2\nu_{1}\abs{\sym{\nabla v}}^{2}-(S\skw{\nabla v}-\skw{\nabla v}S):\tilde{S}\dd{x}+\mathcal{K}(\tilde{S})\int_{\Omega}\frac{|S|^{2}}{2}\dd{x}
    \end{equation*}
    defines a convex and lower semicontinuous functional. Therefore, the functional is weakly lower semicontinuous. This allows one to pass to the limit.
\end{proof}
\begin{remark}\label{EVS satisfies strong TG law}
    From Proposition~\ref{recover strong GT}, one can see that an energy-variational solution obtained from Theorem~\ref{exitence evs non-regulairzed system} also satisfies
    \begin{equation}
        \mu=-\Delta\varphi+W^{\prime}(\varphi).
    \end{equation}
    Moreover, from Propositions \ref{recover weak momentum balance} and~\ref{recover weak CH}, one can deduce that also the weak formulations~\eqref{weak formualtion: momentum} and~\eqref{weak formualtion: CH} of the momentum balance and the phase-field evolution law are satisfied.
\end{remark}

\section{Limit passage from a logarithmic to a double-obstacle 
potential for  system \eqref{sys:two phase} with $\gamma=0$} \label{sec:5}
In Section~\ref{Sec:singular}, we established the existence of energy-variational solutions for the geodynamical two-phase system \eqref{sys:two phase} with a logarithmic potential satisfying Assumption~\ref{ASM:singular potential} both in the regularized setting with stress diffusion $\gamma>0$ and in the {nonregularized} setting without stress diffusion $\gamma=0$. However, both for analytical and modeling reasons, the {nonsmooth} double-obstacle potential is often preferred: it confines the phase-field variable $\varphi$ to the interval $[-1,1]$ by taking the value $+\infty$ in $\mathbb{R}\backslash[-1,1],$ but allows $\varphi$ to attain the pure phases located in $-1$ and $1$. Therefore, in this {section}, we investigate the {nonregularized} system with double-obstacle potentials satisfying Assumption~\ref{ASM:Double well potential} and prove the existence of energy-variational solutions and dissipative solutions by means of variational convergence using an approximation of the obstacle potential by logarithmic potentials. To do so, we first discuss the variational convergence of the logarithmic to the double-obstacle phase-field energy in the sense of Mosco-convergence as well as important compactness results for {the limit passage} in Section \ref{Sec:Mosco}. Subsequently we perform the limit passage $\alpha\searrow0$ in the setting of energy-variational solutions in Section \ref{sec: alpha limit evs} and carry out the corresponding limit passage in the setting of dissipative solutions in Section \ref{Sec:dissipsol-alpha0}.
\subsection{Mosco convergence and compactness}
\label{Sec:Mosco}
For $\alpha>0$, we define a logarithmic potential as 
\begin{equation}
    W_{\mathrm{sg},\alpha}(\varphi):=\alpha\left((1+\alpha+\varphi)\ln(1+\alpha+\varphi)+(1+\alpha-\varphi)\ln(1+\alpha-\varphi)\right).
\end{equation}
In this way, the singularities of $W_{\mathrm{sg},\alpha}$ are located in $1+\alpha$ and $-1-\alpha$. 
Moreover, we introduce the modified phase-field potential
\begin{equation}\label{modified potential}
    {W}_{\alpha}(\varphi):=W_{\mathrm{dw}}(\varphi)+ W_{\mathrm{sg},\alpha}(\varphi).
\end{equation}
Notice that $W_{\alpha}$ satisfies property \eqref{assumption on singular potential} of Assumption \ref{ASM:singular potential} when replacing the interval $(-1,1)$ by the interval $(-1-\alpha,1+\alpha)$. 
Thus, in the following, we first consider system~\eqref{sys:two phase} with the phase-field potential $W_{\alpha}$ as defined in~\eqref{modified potential}.
Then, for all $\alpha>0$,  Theorem~\ref{exitence evs non-regulairzed system} allows us to conclude that there exists an energy-variational solution $(v_{\alpha},S_{\alpha},\varphi_{\alpha},\mu_{\alpha},E_{\alpha})$ of system~\eqref{sys:two phase} for the {nonregularized} system $\gamma=0$ with potential $W_\alpha$ of type $\mathcal{K}$, where the regularity weight $\mathcal{K}$ is given by~\eqref{regularity weight for non-regularized}, i.e., $(v_{\alpha},S_{\alpha},\varphi_{\alpha},\mu_{\alpha},E_{\alpha})$ satisfies
\begin{subequations}
\label{EVS-alpha}
\begin{equation}\label{EVS inequality alpha}
\begin{aligned}
    &\left(E_{\alpha}(\tau)-\langle \rho_{\alpha} v_{\alpha},\tilde{v}\rangle_{L^{2}}-\langle S_{\alpha},\tilde{S} \rangle_{L^{2}}-\langle \varphi_{\alpha},\tilde{\varphi} \rangle_{L^{2}}\right)\Big|_{s}^{t}
        \\
    +&\int_{s}^{t}\int_{\Omega}2\nu(\varphi_{\alpha})\abs{\sym{\nabla v_{\alpha}}}^{2}+\gamma\abs{\nabla S_{\alpha}}^{2}+m(\varphi_{\alpha})\abs{\nabla\mu_{\alpha}}^{2}\dd{x}\dd{\tau}+\int_{s}^{t}\mathcal{P}(\varphi_{\alpha};S_{\alpha})-\mathcal{P}(\varphi_{\alpha};\tilde{S})\dd{\tau}
        \\
    +&\int_{s}^{t}\int_{\Omega}
    \rho_{\alpha} v_{\alpha}\cdot\partial_{t}\tilde{v}
    +\rho_{\alpha} v_{\alpha}\otimes v_{\alpha}:\nabla\tilde{v}
    +v_{\alpha}\otimes J_{\alpha}:\nabla\tilde{v}\dd{x}\dd{\tau}
    \\
    -&\int_{s}^{t}\int_{\Omega}2\nu(\varphi_{\alpha})\sym{\nabla v_{\alpha}}:\sym{\nabla\tilde{v}}
    +\eta(\varphi_{\alpha})S_{\alpha}:\sym{\nabla\tilde{v}}
    -\nabla\varphi_{\alpha}\otimes\nabla\varphi_{\alpha}:\nabla\tilde{v}\dd{x}\dd{\tau}
        \\
    +&\int_{s}^{t}\int_{\Omega}S_{\alpha}:\partial_{t}\tilde{S}
    +S_{\alpha}\otimes v_{\alpha}\threedotsbin\nabla\tilde{S}
    -\left( S_{\alpha}\skw{\nabla v_{\alpha}} - \skw{\nabla v_{\alpha}}S_{\alpha}\right):\tilde{S}
    +\eta(\varphi_{\alpha})\sym{\nabla v_{\alpha}}:\tilde{S}
    \dd{x}\dd{\tau}
    \\
    +&\int_{s}^{t}\int_{\Omega}\varphi_{\alpha}\partial_{t}\tilde{\varphi}
    -(v_{\alpha}\cdot\nabla\varphi_{\alpha})\tilde{\varphi}
    -m(\varphi_{\alpha})\nabla \mu_{\alpha}:\nabla\tilde{\varphi}\dd{x}\dd{\tau}
        \\
    -&\int_{s}^{t}\int_{\Omega}\mu_{\alpha}\tilde{\mu}
    -(-\Delta\varphi_{\alpha}
    +W_{\mathrm{dw}}^{\prime}(\varphi_{\alpha})
    +\beta_{\alpha})\tilde{\mu}\dd{x}\dd{\tau}
        \\
    \leq&\int_{s}^{t}\mathcal{K}(\tilde{v},\tilde{S},\tilde{\varphi},\tilde{\mu})\left(E_{\alpha}(\tau)-\mathcal{E}_{\alpha}(v_{\alpha},S_{\alpha},\varphi_{\alpha})\right)\dd{\tau}+\int_{s}^{t}\langle f,v_{\alpha}-\tilde{v}\rangle_{H^{1}}\dd{\tau}
\end{aligned}
\end{equation}
where
\begin{equation}
    \mathcal{E}_{\alpha}(v,S,\varphi):=\int_{\Omega}
    \frac{\rho(\varphi)}{2}\abs{v}^{2}
    +\frac{1}{2}\abs{S}^{2}
    +\frac{1}{2}\abs{\nabla\varphi}^{2}
    +W_{\mathrm{dw}}(\varphi)
    +W_{\mathrm{sg},\alpha}(\varphi)\dd{x}\,,
\end{equation}
as well as
\begin{equation}
\label{EVS-alpha-beta} 
    \beta_{\alpha}= W_{\mathrm{sg},\alpha}^{\prime}(\varphi_{\alpha}).
\end{equation}  
\end{subequations}
To pass to the limit $\alpha\searrow0$ in \eqref{EVS-alpha} we can mainly repeat the arguments of Theorem~\ref{exitence evs non-regulairzed system} and~\cite[Theorem 5.1]{arxiv:2509.25508}. Only for the term $\beta_{\alpha}$ we have to pay special attention and need argue differently, due to the fact  that the derivative of the logarithmic potential blows up at $\pm(1+\alpha)$. To overcome this difficulty, in view of \eqref{EVS-alpha-beta}, we would like to employ weak-strong closedness of subdifferentials under Mosco-convergence.  For this, we introduce for any $T^{\prime}\in(0,T)$ and $\alpha>0$ the following time-integrated functionals
\begin{equation}\label{space time energy: singular}
\begin{aligned}
\mathcal{E}_{\mathrm{sg},\alpha}(\varphi)&:=
    \begin{cases}
        \int_{0}^{T^{\prime}}\int_{\Omega}W_{\mathrm{sg},\alpha}(\varphi)\dd{x}\dd{\tau}&\text{ for }\varphi\in\dom{\mathcal{E}_{\mathrm{sg},\alpha}},
        \\
        +\infty&\text{ otherwise},
    \end{cases}\\
    \text{where }\;\dom{\mathcal{E}_{\mathrm{sg},\alpha}}&:=\left\{ \varphi\in L^{2}(0,T^{\prime};L^{2}(\Omega)):\abs{\varphi}\leq1+\alpha\text{ a.e. in }\Omega\times(0,T^{\prime}) \right\},
\end{aligned}
\end{equation}
as well as
\begin{equation}\label{space time energy: indicator}
\begin{aligned}
        \mathcal{I}_{K}(\varphi)&:=\begin{cases}
        0&\text{ for }\varphi\in K,\\
        \infty&\text{ for }\varphi\notin K,
    \end{cases}\\
 \text{where }\;   K&:=\left\{\varphi\in L^{2}(0,T^{\prime};L^{2}(\Omega)):\abs{\varphi}\leq1\text{ a.e. in }\Omega\times(0,T^{\prime})\right\}.
\end{aligned}    
\end{equation}
With the next lemma we verify that Mosco-convergence of the functionals $(\mathcal{E}_{\mathrm{sg},\alpha})_\alpha$ to $\mathcal{I}_K$. 
\begin{lem}\label{M-conv:singular potential}
Let $\mathcal{E}_{\mathrm{sg},\alpha}$ and $\mathcal{I}_{K}$ be given as~\eqref{space time energy: singular} and~\eqref{space time energy: indicator}.
The functional $\mathcal{E}_{\mathrm{sg},\alpha}$ converges to $\mathcal{I}_{K}$ in the sense of Mosco in $L^{2}(0,T^{\prime};L^{2}(\Omega))$, i.e., $\mathcal{E}_{\mathrm{sg},\alpha} \xrightarrow{\mathrm{M}} \mathcal{I}_{K}$ in $L^{2}(0,T^{\prime};L^{2}(\Omega))$. More precisely, the following two properties hold:
\begin{subequations}
\begin{enumerate}
\item[\textup{(M1)}]{liminf-inequality:} For every sequence $(\varphi_\alpha)_\alpha$ with  $\varphi_{\alpha} \rightharpoonup \varphi$ in $L^{2}(0,T^{\prime};L^{2}(\Omega))$, {there holds}
\begin{equation}\label{Mosco: liminf}
    \liminf_{\alpha \searrow 0} \mathcal{E}_{\mathrm{sg},\alpha}(\varphi_{\alpha}) \geq \mathcal{I}_{K}(\varphi).    
\end{equation}
\item[\textup{(M2)}]{Strong convergence of recovery sequences:} For every $\varphi\in L^{2}(0,T^{\prime};L^{2}(\Omega))$ there exists a sequence $(\varphi_{\alpha})_{\alpha}\subseteq L^{2}(0,T^{\prime};L^{2}(\Omega))$ {such that} $\varphi_{\alpha} \to \varphi$ in $L^{2}(0,T^{\prime};L^{2}(\Omega))$ and
\begin{equation}\label{Mosco: limsup}
    \limsup_{\alpha \searrow 0}\mathcal{E}_{\mathrm{sg},\alpha}(\varphi_{\alpha}) \leq \mathcal{I}_{K}(\varphi).
\end{equation}
\end{enumerate}
\end{subequations}
\end{lem}

\begin{proof}
    To show (M1), consider a sequence 
    $\varphi_{\alpha} \rightharpoonup \varphi$ in $L^{2}(0,T^{\prime};L^{2}(\Omega))$. If $\liminf_{\alpha \searrow 0} \mathcal{E}_{\mathrm{sg},\alpha}(\varphi_{\alpha})=+\infty$, there is nothing to show. Hence, we assume without loss of generality that $\sup_{\alpha}\mathcal{E}_{\mathrm{sg},\alpha}(\varphi_{\alpha})<\infty$. By definition of $\mathcal{E}_{\mathrm{sg},\alpha}$, we have $\abs{\varphi_{\alpha}}\leq1+\alpha$ a.e.\ in $\Omega\times(0,T^{\prime})$ for each $\alpha$. Now, we define
    \begin{equation*}
        \tilde{\varphi}_{\alpha}:=\max\{-1,\min\{1,\varphi_{\alpha}\}\}.
    \end{equation*}
    Notice that $\tilde{\varphi}_{\alpha}\in K$ and 
    \begin{equation*}
        \abs{\tilde{\varphi}_{\alpha}-\varphi_{\alpha}}\leq\alpha\text{ a.e. in }\Omega\times(0,T^{\prime}).
    \end{equation*}
    {Hence}
    \begin{equation*}
        \lVert \tilde{\varphi}_{\alpha}-\varphi_{\alpha}\rVert_{L^{2}(0,T^{\prime};L^{2}(\Omega))}\leq C\alpha\to0\quad\text{as }\alpha\to0.
    \end{equation*}
    Since $\varphi_{\alpha}\rightharpoonup\varphi$ in $L^{2}(0,T^{\prime};L^{2}(\Omega))$, we also have that  $\tilde{\varphi}_{\alpha}\rightharpoonup\varphi$ in $L^{2}(0,T^{\prime};L^{2}(\Omega))$. Further observe the set $K$ is closed and convex in $L^{2}(0,T^{\prime};L^{2}(\Omega))$, and, therefore, weakly closed. Hence, $\tilde{\varphi}_{\alpha}\in K$ for all $\alpha$ implies $\varphi\in K$. Consequently, we obtain
    \begin{equation*}
        \liminf_{\alpha \searrow 0} \mathcal{E}_{\mathrm{sg},\alpha}(\varphi_{\alpha}) \geq 0=\mathcal{I}_{K}(\varphi).
    \end{equation*}
    To show (M2), let $\varphi\in L^{2}(0,T^{\prime};L^{2}(\Omega))$.  In fact, it is possible to take the constant sequence $\varphi_{\alpha}\equiv\varphi$ as a recovery sequence: If  $\varphi\in K$ there holds $\abs{\varphi}\leq1$ a.e.\ in $\Omega\times(0,T^{\prime})$, and then we have
    \begin{equation*}
    \begin{aligned}
        \mathcal{I}_{K}(\varphi)=0\leq \mathcal{E}_{\mathrm{sg},\alpha}(\varphi)&=\alpha\int_{0}^{T^{\prime}}\int_{\Omega}(1+\alpha+\varphi)\ln{(1+\alpha+\varphi)}+(1+\alpha-\varphi)\ln{(1+\alpha-\varphi})\dd{x}\dd{\tau}
        \\
        &\leq {C}\alpha((\alpha+2)\ln(\alpha+2)+\alpha\ln(\alpha))\to0=\mathcal{I}_{K}(\varphi){.}
    \end{aligned}
    \end{equation*}
    If $\varphi\notin K$, there exists a constant $c>0$ such that  $\mathcal{L}^4(\{\abs{\varphi}\geq1+c\})>0$ with $\mathcal{L}^4$ the {Lebesgue measure} in $\Omega\times[0,T']$. Hence, we have 
    \begin{equation*}
        \mathcal{E}_{\mathrm{sg},\alpha}(\varphi)=+\infty=\mathcal{I}_{K}(\varphi)
    \end{equation*}
    for all $\alpha<c,$ so that \eqref{Mosco: limsup} {is satisfied.}
\end{proof}

\begin{corollary}\label{M-conv:subdifferential}
    Suppose $\beta_{\alpha} \rightharpoonup \beta$ in $L^{2}(0,T^{\prime};L^{2}(\Omega))${, and suppose} $\varphi_{\alpha} \to \varphi$ in $L^{2}(0,T^{\prime};L^{2}(\Omega))$ with
\begin{equation*}
    \beta_{\alpha}\in \partial \mathcal{E}_{\mathrm{sg},\alpha}(\varphi_{\alpha})\text{ for all }\alpha{.}    
\end{equation*}
    {Then}
\begin{equation*}
    \beta\in\partial\mathcal{I}_{K}(\varphi).
\end{equation*}
\end{corollary}

\begin{proof}
    To show the corollary, we consider the {Legendre--Fenchel} conjugates
    \begin{subequations}
    \begin{align}
        \mathcal{E}_{\mathrm{sg},\alpha}^{*}(\beta_{\alpha})
        :=&
        \sup\left\{
        \langle \beta_{\alpha},\varphi_{\alpha} \rangle_{L^{2}(0,T^{\prime};L^{2}(\Omega))}-\mathcal{E}_{\mathrm{sg},\alpha}(\varphi_{\alpha})
        :\varphi_{\alpha}\in L^{2}(0,T^{\prime};L^{2}(\Omega)) 
        \right\}\,,
        \\
        \mathcal{I}_{K}^{*}(\beta)
        :=&
        \sup\left\{
        \langle \beta,\varphi \rangle_{L^{2}(0,T^{\prime};L^{2}(\Omega))}
        -\mathcal{I}_{K}(\varphi)
        :\varphi\in L^{2}(0,T^{\prime};L^{2}(\Omega))
        \right\}
        \,.
    \end{align}
    \end{subequations}
    From~\cite[Theorem 1]{zbMATH03399971}, we have
    \begin{equation}
        \mathcal{E}_{\mathrm{sg},\alpha}^{*}\xrightarrow{\mathrm{M}} \mathcal{I}_{K}^{*}.
    \end{equation}
    By Fenchel equivalence, {see}~\cite{zbMATH03058046}, we have
    \begin{equation}\label{Fenchel equivalence}
        \langle \beta_{\alpha},\varphi_{\alpha}\rangle_{L^{2}(0,T^{\prime};L^{2}(\Omega))}
        \geq
        \mathcal{E}_{\mathrm{sg},\alpha}(\varphi_{\alpha})
        +\mathcal{E}_{\mathrm{sg},\alpha}^{*}(\beta_{\alpha}).
    \end{equation}
    \begin{subequations}
    From (M1), we have
    \begin{align}
        \liminf_{\alpha \searrow 0} \mathcal{E}_{\mathrm{sg},\alpha}(\varphi_{\alpha}) 
        &\geq \mathcal{I}_{K}(\varphi),
        \label{Mosco liminf: origin}\\
        \liminf_{\alpha \searrow 0} \mathcal{E}_{\mathrm{sg},\alpha}^{*}(\beta_{\alpha}) &\geq \mathcal{I}_{K}^{*}(\beta).
        \label{Mosco liminf: dual}
    \end{align}
    Moreover, since we assume that $\beta_{\alpha} \rightharpoonup \beta$ and $\varphi_{\alpha} \to \varphi$ in $L^{2}(0,T^{\prime};L^{2}(\Omega))$, {we have}
    \begin{align}
        \lim_{\alpha \searrow 0}\langle \beta_{\alpha},\varphi_{\alpha}\rangle_{L^{2}(0,T^{\prime};L^{2}(\Omega))}
        =\langle \beta,\varphi\rangle_{L^{2}(0,T^{\prime};L^{2}(\Omega))}.
        \label{Mosco lim: weak-strong}
    \end{align}
    \end{subequations}
    Inserting~\eqref{Mosco liminf: origin}-\eqref{Mosco lim: weak-strong} into~\eqref{Fenchel equivalence} yields
    \begin{equation}
        \langle \beta,\varphi\rangle_{L^{2}(0,T^{\prime};L^{2}(\Omega))}
        \geq
        \mathcal{I}_{K}(\varphi)
        +\mathcal{I}_{K}^{*}(\beta)\,,
    \end{equation}
    which is equivalent to
    \begin{equation}
        \beta\in\partial\mathcal{I}_{K}(\varphi)
    \end{equation}
    by {Fenchel equivalence.}
\end{proof}

\begin{remark}\label{subdifferential integrand}
    With the help of~\cite[Proposition 2.53]{zbMATH05948485}, one can see that 
    \begin{equation*}
        \partial\mathcal{I}_{K}(\varphi)=\left\{ \xi\in L^{2}(0,T^{\prime};L^{2}(\Omega)):\xi(x,t)\in\partial {I}_{[-1,1]}(\varphi(x,t))\text{ for a.e. }(x,t)\in \Omega\times(0,T^{\prime})\right\}.
    \end{equation*}
\end{remark}
\begin{lem}[$\alpha$-uniform estimates]\label{est:alpha uniform}
    There is a constant $C>0$ such that the energy-variational solutions $(v_{\alpha},S_{\alpha},\varphi_{\alpha},\mu_{\alpha}, E_\alpha)_\alpha$ obtained from Theorem~\ref{exitence evs non-regulairzed system} for system \eqref{sys:two phase} with $\gamma=0$ and the phase-field potentials $(W_\alpha)_\alpha$ satisfy the following $\alpha$-uniform estimates for all $0<T^{\prime}<T$: 
    \begin{subequations}
       \begin{align}
           \|E_\alpha\|_{L^\infty(0,T')}+ \lVert v_{\alpha}\rVert_{L^{\infty}(0,T^{\prime};L^{2}(\Omega))}
                +\lVert S_{\alpha}\rVert_{L^{\infty}(0,T^{\prime};L^{2}(\Omega))}
                +\lVert \nabla\varphi_{\alpha}\rVert_{L^{\infty}(0,T^{\prime};L^{2}(\Omega))}&
                \nonumber\\
                +\lVert \sym{\nabla v_{\alpha}}\rVert_{L^{2}(0,T^{\prime};L^{2}(\Omega))}
                +\lVert \nabla\mu_{\alpha}\rVert_{L^{2}(0,T^{\prime};L^{2}(\Omega))}
                &\leq C,
                \label{est:modified total energy}\\
                \sup_{[0,T^{\prime}]}\int_{\Omega}F(\varphi_{\alpha}(t))\dd{x}
                +
                \int_{0}^{T^{\prime}}\int_{\Omega}
                \abs{\Delta\varphi_{\alpha}}^{2}
                \dd{x}\dd{t}
                &\leq C,
                \label{est:laplacian}\\
                \int_{0}^{T^{\prime}}\int_{\Omega}\abs{\beta_{\alpha}}^{2}\dd{x}\dd{t}
                &\leq C.
                \label{est:derivative of singular potential}
            \end{align}
        Above in \eqref{est:laplacian},  $F\in C^2(\mathbb{R})$  is a {nonnegative} function with the properties 
         $F(0)=F^{\prime}(0)=0$ and $F^{\prime\prime}(s)=1/m(s)$ for all $s\in\mathbb{R}$.  
        \end{subequations}
\end{lem}

\begin{proof}
        \eqref{est:modified total energy} follows directly by choosing $\tilde{v}\equiv0$, $\tilde{S}\equiv0$, $\tilde{\varphi}\equiv0$ and $\tilde{\mu}\equiv0$ in \eqref{EVS inequality} and by exploiting the assumptions on the material parameters and the given data. To derive \eqref{est:laplacian}, first notice that Proposition~\ref{recover weak CH} implies that $(v_{\alpha},S_{\alpha},\varphi_{\alpha},\mu_{\alpha}, E_\alpha)$ satisfies~\eqref{weak formualtion: CH} for each $\alpha>0$. Hence, with the help of~\eqref{est:modified total energy}, we have that $\partial_{t}\varphi_{\alpha}$ is bounded in $L^{2}(0,T^{\prime};H^{-1}(\Omega))$. Moreover, from Remark~\ref{EVS satisfies strong TG law}, we obtain that $(v_{\alpha},S_{\alpha},\varphi_{\alpha},\mu_{\alpha}, E_\alpha)$  satisfies~\eqref{strong TG law} with $W_\alpha$. 
        
        Now, let $F\in C^2(\mathbb{R})$ be a {nonnegative} function satisfying  $F(0)=F^{\prime}(0)=0$ and $F^{\prime\prime}(s)=1/m(s)$ for all $s\in\mathbb{R}$. Observe that $F^{\prime}(\varphi_{\alpha})$ is {nonnegative} and $F^{\prime}(\varphi_{\alpha})$ is an admissible test function for~\eqref{weak formualtion: CH}. Then, testing~\eqref{weak formualtion: CH} with $F^{\prime}(\varphi_{\alpha})$ and integration by parts in time gives 
        \begin{equation}\label{tested by F}
            \begin{aligned}
            &\int_{0}^{T^{\prime}}
            \langle \partial_{t}\varphi_{\alpha},F^{\prime}(\varphi_{\alpha})\rangle_{H^{1}}
            \dd{\tau}
            +\int_{0}^{T^{\prime}}\int_{\Omega}
            v_{\alpha}\cdot\nabla\varphi_{\alpha}F^{\prime}(\varphi_{\alpha})
            \dd{x}\dd{\tau}
            \\=&-\int_{0}^{T^{\prime}}\int_{\Omega}m(\varphi_{\alpha})\nabla\mu_{\alpha}\cdot\nabla F^{\prime}(\varphi_{\alpha})\dd{x}\dd{\tau}
            \\
            =&-\int_{0}^{T^{\prime}}\int_{\Omega}
            \nabla \mu_{\alpha}\cdot\nabla\varphi_{\alpha}
            \dd{x}\dd{\tau}
            \\
            =&\int_{0}^{T^{\prime}}\int_{\Omega}
            W_{\alpha}^{\prime}(\varphi_{\alpha})\Delta\varphi_{\alpha}
            -\abs{\Delta\varphi_{\alpha}}^{2}
            \dd{x}\dd{\tau}\,,
        \end{aligned}
        \end{equation}
     where we also used \eqref{strong TG law} to arrive at the last line.
        Observe that the left-hand side of~\eqref{tested by F} gives
        \begin{align}
            \int_{0}^{T^{\prime}}
            \langle \partial_{t}\varphi_{\alpha},F^{\prime}(\varphi_{\alpha})\rangle_{H^{1}}
            \dd{\tau}
            &=\int_{\Omega}F(\varphi_{\alpha}(T^{\prime}))\dd{x}
            -\int_{\Omega}F(\varphi_{0})\dd{x},
            \\
            \int_{0}^{T^{\prime}}\int_{\Omega}
            v_{\alpha}\cdot\nabla\varphi_{\alpha}F^{\prime}(\varphi_{\alpha})
            \dd{x}
            &=\int_{0}^{T^{\prime}}\int_{\Omega}
            v_{\alpha}\cdot\nabla F(\varphi_{\alpha})
            \dd{x}
            =0.
        \end{align}
        Moreover, we estimate the first term on the right-hand side of~\eqref{tested by F} as  
        \begin{equation}
        \begin{aligned}
            \int_{0}^{T^{\prime}}\int_{\Omega}
            W_{\alpha}^{\prime}(\varphi_{\alpha})\Delta\varphi_{\alpha}
            \dd{x}\dd{\tau}
            &=\int_{0}^{T^{\prime}}\int_{\Omega}
                W_{\mathrm{dw}}^{\prime}(\varphi_{\alpha})\Delta\varphi_{\alpha}
            \dd{x}\dd{\tau}
            +\int_{0}^{T^{\prime}}\int_{\Omega}
                W_{\mathrm{sg},\alpha}^{\prime}(\varphi_{\alpha})\Delta\varphi_{\alpha}
            \dd{x}\dd{\tau}
            \\
            &=-\int_{0}^{T^{\prime}}\int_{\Omega}
                W_{\mathrm{dw}}^{\prime\prime}(\varphi_{\alpha})\abs{\nabla\varphi_{\alpha}}^{2}
            \dd{x}\dd{\tau}
            -\int_{0}^{T^{\prime}}\int_{\Omega}
                W_{\mathrm{sg},\alpha}^{\prime\prime}(\varphi_{\alpha})\abs{\nabla\varphi_{\alpha}}^{2}
            \dd{x}\dd{\tau}
            \\
            &\leq-\int_{0}^{T^{\prime}}\int_{\Omega}
                W_{\mathrm{dw}}^{\prime\prime}(\varphi_{\alpha})\abs{\nabla\varphi_{\alpha}}^{2}
            \dd{x}\dd{\tau}
            \leq \int_{0}^{T^{\prime}}\int_{\Omega}
                |W_{\mathrm{dw}}^{\prime\prime}(\varphi_{\alpha})|\abs{\nabla\varphi_{\alpha}}^{2}
            \dd{x}\dd{\tau}.
        \end{aligned}
        \end{equation}
        Therefore, we obtain
        \begin{equation}\label{est with F}
            \begin{aligned}
            \int_{\Omega}F(\varphi_{\alpha}(T^{\prime}))\dd{x}
            +\int_{0}^{T^{\prime}}\int_{\Omega}\abs{\Delta\varphi_{\alpha}}^{2}\dd{x}\dd{\tau}
            \leq C\left( 
            \int_{\Omega}F(\varphi_{0})\dd{x}
            +\int_{0}^{T^{\prime}}\int_{\Omega}\abs{W_{\mathrm{dw}}^{\prime\prime}(\varphi_{\alpha})}\abs{\nabla\varphi_{\alpha}}^{2}\dd{x}\dd{\tau}
            \right){,}
        \end{aligned}
        \end{equation}
        which implies \eqref{est:laplacian}.  
        
        Now, to derive \eqref{est:derivative of singular potential}, we test~\eqref{strong TG law}, where $W=W_\alpha=W_\mathrm{dw}+W_{\mathrm{sg},\alpha},$  with the test function  $\varphi_{\alpha}-\varphi_{\alpha,\Omega},$ where $\varphi_{\alpha,\Omega}$ {is} as in \eqref{mean value in Omega}, and integrate over $\Omega$ to get
        \begin{equation}
            \begin{aligned}
            \int_{\Omega}\mu_{\alpha}(\varphi_{\alpha}-\varphi_{\alpha,\Omega})\dd{x}
            =&
            \int_{\Omega}W_{\mathrm{dw}}^{\prime}
            (\varphi_{\alpha})(\varphi_{\alpha}-\varphi_{\alpha,\Omega})
            \dd{x}
            \nonumber\\
            &+\int_{\Omega}
            W_{\mathrm{sg},\alpha}^{\prime}(\varphi_{\alpha})(\varphi_{\alpha}-\varphi_{\alpha,\Omega})
            \dd{x}
            -\int_{\Omega}\Delta\varphi_{\alpha}(\varphi_{\alpha}-\varphi_{\alpha,\Omega})\dd{x}\,.
        \end{aligned}
        \end{equation}
        Above relation {holds} for a.e.\ $t\in(0,T^{\prime})$. Since $\lim_{s\to\pm(1+\alpha)}W_{\mathrm{sg},\alpha}^{\prime}(s)=\pm\infty$, we show that 
        \begin{equation}
            W_{\mathrm{sg},\alpha}^{\prime}(s)(s-\varphi_{\alpha,\Omega})\geq C_{1}|W_{\mathrm{sg},\alpha}^{\prime}(s)|-C_{0}   
        \end{equation}
        for all $s\in(-1-\alpha,1+\alpha)$, using the fact that $\varphi_{\alpha,\Omega}\in(-1-\alpha+c,1+\alpha-c)$ for some small $c>0${;} see also~\cite[Lemma 4.2]{zbMATH06210388}. {Hence}
        \begin{equation*}
            \int_{\Omega}\abs{W_{\mathrm{sg},\alpha}^{\prime}(\varphi_{\alpha})}\dd{x}
            \leq C\left(\int_{\Omega}W_{\mathrm{sg},\alpha}^{\prime}(\varphi_{\alpha})(\varphi_{\alpha}-\varphi_{\alpha,{\Omega}})\dd{x}+1\right)
        \end{equation*}
        for a.e.\ $t\in(0,T^{\prime})$. Moreover, we have
        \begin{align*}
            \abs{\int_{\Omega}\mu_{\alpha}(\varphi_{\alpha}-\varphi_{\alpha,\Omega})\dd{x}}
            &=\abs{\int_{\Omega}(\mu_{\alpha}-\mu_{\alpha,\Omega})\varphi_{\alpha}\dd{x}}
            \leq\int_{\Omega}\abs{\nabla\mu_{\alpha}}^{2}\dd{x},
            \\
            \abs{\int_{\Omega}\Delta\varphi_{\alpha}(\varphi_{\alpha}-\varphi_{\alpha,\Omega})\dd{x}}
            &\leq\int_{\Omega}\abs{\nabla\varphi_{\alpha}}^{2}\dd{x}
        \end{align*}
        for a.e.\ $t\in(0,T^{\prime})$, with the help of Poincar\'e inequality{;} see~\cite[Chapter 5.8.1]{zbMATH05681750} for details. Therefore, we obtain
        \begin{equation}
            \begin{aligned}
            \int_{\Omega}\abs{W_{\mathrm{sg},\alpha}^{\prime}(\varphi_{\alpha})}\dd{x}
            &\leq
            C\left(
            \int_{\Omega}\abs{\nabla\mu_{\alpha}}^{2}\dd{x}
            +\int_{\Omega}\abs{\nabla\varphi_{\alpha}}^{2}\dd{x}
            +\int_{\Omega}\abs{W_{\mathrm{dw}}^{\prime}(\varphi_{\alpha})(\varphi_{\alpha}-\varphi_{\alpha,\Omega})}\dd{x}
            +1\right)
            \nonumber\\
            &\leq C\left(\int_{\Omega}\abs{\nabla\mu_{\alpha}}^{2}\dd{x}
            +1\right)
        \end{aligned}
        \end{equation}
        for a.e.\ $t\in(0,T^{\prime})$, using that $\abs{\varphi_{\alpha}}<(1+\alpha)$. Additionally, integrating \eqref{strong TG law} over $\Omega$ implies
        \begin{equation*}
        \begin{aligned}
            \abs{\int_{\Omega}\mu_{\alpha}\dd{x}}
            &\leq
            \int_{\Omega}\abs{W_{\mathrm{dw}}^{\prime}(\varphi_{\alpha})}\dd{x}
            +\int_{\Omega}\abs{W_{\mathrm{sg},\alpha}^{\prime}(\varphi_{\alpha})}\dd{x}
            +\abs{\int_{\Omega}\Delta\varphi_{\alpha}\dd{x}}
            \\
            &\leq C\left(\int_{\Omega}\abs{\nabla\mu_{\alpha}}^{2}\dd{x}+1\right)
        \end{aligned}
        \end{equation*}
        for a.e. $t\in(0,T^{\prime})$. Hence, we derive
        \begin{align}
            \int_{\Omega}\abs{\mu_{\alpha}(t)}^{2}\dd{x}
            \leq
            C\left(\int_{\Omega}\abs{\nabla\mu_{\alpha}(t)}^{2}\dd{x}+1\right)
        \end{align}
        for a.e.\ $t\in(0,T^{\prime})$, with the help of Poincar\'e inequality. Now, from~\eqref{strong TG law}, we have
        \begin{align*}
            \abs{W_{\mathrm{sg},\alpha}^{\prime}(\varphi_{\alpha})}^{2}
            \leq 
            \abs{\Delta\varphi_{\alpha}}^{2}
            +\abs{W_{\mathrm{dw}}^{\prime}(\varphi_{\alpha})}^{2}
            +\abs{\mu_{\alpha}}^{2}.
        \end{align*}
        Integrating over space and time yields
        \begin{align}
            \int_{0}^{T^{\prime}}\int_{\Omega}\abs{W_{\mathrm{sg},\alpha}^{\prime}(\varphi_{\alpha})}^{2}\dd{x}\dd{\tau}
            \leq C\left( \int_{0}^{T^{\prime}}\int_{\Omega}\abs{\nabla\mu_{\alpha}}^{2}\dd{x}\dd{\tau}+1\right).
        \end{align}
       Therefore, we arrive at~\eqref{est:derivative of singular potential} with the help of~\eqref{est:modified total energy}.
\end{proof}
\subsection{Limit passage $\alpha\searrow 0$ in the setting of energy-variational solutions}\label{sec: alpha limit evs} 
In this {section}, we perform the limit passage $\alpha \searrow 0$ to obtain the existence of energy-variational solutions for system~\eqref{sys:two phase} with a double-obstacle potential $W_\mathrm{dw}+I_{[-1,1]}$.
\begin{theorem}[Limit passage in the setting of energy-variational solutions]\label{exitence evs double well}
         Let $\gamma=0$ and let the phase-field potential be given by  $W=W_{\mathrm{dw}}+I_{[-1,1]}$ as in \eqref{def-W}. Let the  Assumptions~\ref{ASM:domain}-\ref{ASM:Double well potential} be satisfied. Then there exists an energy-variational solution $(v,S,\varphi,\mu,E)$ of type $\mathcal{K}$, in the sense of Definition~\ref{defi: EVS} with $E(0)=\mathcal{E}(v_{0},S_{0},\varphi_{0})$, for the {nonregularized} system~\eqref{sys:two phase} with $\gamma=0$ where $\mathcal{K}$ is given as~\eqref{regularity weight for non-regularized}. Moreover, $\abs{\varphi}\leq1$ a.e.\ in $\Omega\times[0,T)$.
\end{theorem} 
\begin{proof}
        \begin{subequations}From Lemma~\ref{est:alpha uniform}, we obtain
        \begin{align}
            \lVert v_{\alpha}\rVert_{L^{2}(0,T^{\prime};H_{0,\mathrm{div}}^{1}(\Omega))}
            +\lVert v_{\alpha}\rVert_{L^{\infty}(0,T^{\prime};L_{\mathrm{div}}^{2}(\Omega))}&\leq C,
            \label{BND:v}\\
            \lVert S_{\alpha}\rVert_{L^{\infty}(0,T^{\prime};L_{\mathrm{sym,Tr}}^{2}(\Omega))}
            &\leq C,
            \label{BND:S}\\
            \lVert \varphi_{\alpha}\rVert_{L^{\infty}(0,T^{\prime};H^{1}(\Omega))}&\leq C,
            \label{BND:phase}\\
            \lVert \mu_{\alpha}\rVert_{L^{2}(0,T^{\prime};H^{1}(\Omega))}&\leq C,
            \label{BND:flux}
        \end{align}
        for all $0<T^{\prime}<T$. Moreover, by~\eqref{est:laplacian} and the homogeneous Neumann boundary condition $\vec{n}\cdot\nabla\varphi_{\alpha}|_{\partial\Omega}=0$, we can further derive
        \begin{align}
            \lVert \varphi_{\alpha}\rVert_{L^{2}(0,T^{\prime};H^{2}(\Omega))}\leq C 
        \end{align}
        \end{subequations}
        also thanks to \cite[Theorem 3.1.3.3]{zbMATH05960425}.
        \begin{subequations}
        Therefore, by a classical diagonalization argument, we can extract a not relabeled subsequence and a limit quadruplet $(v,S,\varphi,\mu)$ such that
        \begin{align}
            v_{\alpha}&\rightharpoonup v\text{ in }L^{2}(0,T^{\prime};H_{0,\mathrm{div}}^{1}(\Omega)),
            \label{WK CONV:v}\\
            v_{\alpha}&\overset{*}{\rightharpoonup} v\text{ in }L^{\infty}(0,T^{\prime};L_{\mathrm{div}}^{2}(\Omega)),
            \\
            S_{\alpha}&\overset{*}{\rightharpoonup} S\text{ in } L^{\infty}(0,T^{\prime};L_{\mathrm{sym,Tr}}^{2}(\Omega)),
            \label{WK CONV:S}\\
            \varphi_{\alpha}&\rightharpoonup \varphi\text{ in }L^{2}(0,T^{\prime};H^{2}(\Omega)),
            \label{WK CONV:phase}\\
            \varphi_{\alpha}&\overset{*}{\rightharpoonup}\varphi\text{ in }L^{\infty}(0,T^{\prime};H^{1}(\Omega)),
            \\
            \mu_{\alpha}&\rightharpoonup\mu\text{ in }L^{2}(0,T^{\prime};H^{1}(\Omega)),
            \label{WK CONV:flux}
        \end{align}
        for all $0<T^{\prime}<T$. Furthermore, since~\eqref{weak formualtion: momentum} and~\eqref{weak formualtion: CH} are satisfied, we conclude the {strong convergence results}
        \begin{align}
            v_{\alpha}&\to v\text{ in }L^{2}(0,T^{\prime};L^{2}(\Omega)),
            \label{STR CONV:v}\\
            \varphi_{\alpha}&\to\varphi\text{ in }L^{2}(0,T^{\prime};H^{1}(\Omega)),
            \label{STR CONV:S}\\
            \varphi_{\alpha}&\to\varphi\text{ a.e.\ in }\Omega\times[0,T)\label{AE CONV:phase},
        \end{align}
        with the help of {Aubin--Lions lemma}. Since $\abs{\varphi_{\alpha}}<1+\alpha$ a.e.\ in $\Omega\times[0,T)$, we deduce that $\abs{\varphi}\leq1$ a.e.\ in $\Omega\times[0,T)$.

        Moreover, for the energy functions $(E_\alpha)_\alpha$ we have for all $\alpha>0${,}
        \begin{equation}
        \label{E-alpha-Linfty}
            E_{\alpha}\Big|_{0}^{t}\leq C\int_{0}^{t}\lVert f\rVert_{H^{-1}}^{2}\dd{\tau}
        \end{equation}
        by a similar argument as in the proof of Theorem~\ref{exitence evs regulairzed system}. Thus, for each $\alpha>0$ the function $E_{\alpha}$ can be viewed as the sum of a  monotonically  decreasing function $E_{\alpha}(t)-C\int_{0}^{t}\lVert f\rVert_{H^{-1}}^{2}\dd{\tau}$ and a bounded  monotonically increasing function $C\int_{0}^{t}\lVert f\rVert_{H^{-1}}^{2}\dd{\tau}$. In addition, the sequence of initial energies $(E_{\alpha}(0))_\alpha$ is uniformly bounded. Therefore, $(E_{\alpha})_\alpha$ is uniformly bounded in $BV([0,T^{\prime}])$ and hence we conclude the existence of a not {relabeled} subsequence and of a limit energy function $E\in BV([0,T^{\prime}])$ such that 
        \begin{equation}
            E_{\alpha}\overset{*}{\rightharpoonup} E\text{ in }BV([0,T^{\prime}]). 
        \end{equation} 
        Moreover, by the compact embedding of $BV$ spaces, we further have
        \begin{align*}
            E_{\alpha}&\to E\text{ in }L^{1}([0,T^{\prime}]),
            \\
            E_{\alpha}(t)&\to E(t)\text{ for a.e.\ }t\in[0,T^{\prime}]
        \end{align*}
        along a further, not {relabeled} subsequence. 
        In addition, thanks to the uniform $L^\infty$-bound \eqref{E-alpha-Linfty}, we can further improve the strong $L^1$-convergence to 
        \begin{equation}\label{STR CONV:Upper bound energy}
            E_{\alpha}\to E\text{ in }L^{p}([0,T^{\prime}]) \quad\text{ for any }p\in[1,\infty)\,.
        \end{equation} 
        By following the arguments in the proof of Theorem~\ref{exitence evs regulairzed system}, we conclude that
        \begin{equation*}
            E\geq\mathcal{E}(v,S,\varphi)
        \end{equation*}
        a.e. in $(0,T)$.

        Furthermore, from the uniform bound \eqref{est:derivative of singular potential} on $(\beta_\alpha)_\alpha$, we derive that 
        \begin{equation}\label{WK CONV:subdifferential}
            \beta_\alpha\rightharpoonup\beta\text{ in }L^{2}(0,T^{\prime};L^{2}(\Omega)).
        \end{equation}
        Recall that we have $\varphi_{\alpha}\to\varphi$ in $L^{2}(0,T^{\prime};L^{2}(\Omega))$. Hence, we obtain $\beta\in\partial\mathcal{I}_{K}(\varphi)$ thanks to Corollary~\ref{M-conv:subdifferential}. Notice that $\beta(x,t)\in\partial{I}_{[-1,1]}(\varphi(x,t))$ for a.e.\ $(x,t)\in\Omega\times(0,T)$ by  Remark~\ref{subdifferential integrand}. 
        \end{subequations}
        Therefore, the limit passage is a direct consequence of the convergence results \eqref{WK CONV:v}-\eqref{WK CONV:subdifferential}. One can also follow the idea of the proof of~\cite[][Theorem 5.1]{arxiv:2509.25508}.
\end{proof}
\begin{remark}
    From Proposition~\ref{recover strong GT}, one can see that an energy-variational solution obtained from Theorem~\ref{exitence evs double well} also satisfies
    \begin{equation}
        \mu(x,t)=-\Delta\varphi(x,t)+W_{\mathrm{dw}}^{\prime}(\varphi(x,t))+\beta(x,t)\quad\text{a.e.\ in }\Omega\times[0,T']\,,
    \end{equation}
     where $\beta\in\partial\mathcal{I}_K$. 
    Moreover, from {Propositions}~\ref{recover weak momentum balance} and~\ref{recover weak CH}, one can see the weak formulations~\eqref{weak formualtion: momentum} and~\eqref{weak formualtion: CH} are satisfied.
\end{remark}
\begin{remark}\label{rm: on evs with regularized system}
    In the case of the regularized system~\eqref{sys:two phase} with $\gamma>0$, one can verify the existence of energy-variational solutions with a double-obstacle potential by following the same idea of Lemma~\ref{est:alpha uniform} and Theorem~\ref{exitence evs double well}.
\end{remark}
\subsection{Limit passage $\alpha\searrow 0$ in the setting of dissipative solutions}
\label{Sec:dissipsol-alpha0}
In Section~\ref{sec: alpha limit evs}, we established the existence of energy-variational solutions by performing limit passage $\alpha \searrow 0$. Although Proposition~\ref{relative energy-dissipation inequality for double well} ensures that every energy-variational solution is also a dissipative solution, in this section, we additionally carry out the limit passage $\alpha \searrow 0$ directly within the framework of dissipative solutions. The convergence results thus obtained in this work are summarized in Remark \ref{Rem-final-conv}. 
\begin{theorem}[Limit passage in the setting of dissipative solutions]\label{exitence ds double well} 
Let $\gamma=0$ and the phase-field potential be given by $W=W_{\mathrm{dw}}+I_{[-1,1]}$ in \eqref{def-W}. Let the  Assumptions~\ref{ASM:domain}-\ref{ASM:Double well potential} be satisfied. Assume further that $W_{\mathrm{dw}}\in C^{3}(\mathbb{R})$. Then there exists a dissipative solution $(v,S,\varphi,\mu)$ of type $\mathcal{K}$, in the sense of Definition~\ref{defi:ds double well}, for the non-regularized system~\eqref{sys:two phase} with $\gamma=0,$ where $\mathcal{K}$ is given as~\eqref{regularity weight for non-regularized}. Moreover, there holds $\abs{\varphi}\leq1$ a.e.\ in $\Omega\times[0,T)$.
\end{theorem}
\begin{proof}
Without loss of generality, assume that $0<\alpha<1$. By~\cite[][Theorem 5.1]{arxiv:2509.25508}, there exists a dissipative solution $(v_{\alpha},S_{\alpha},\varphi_{\alpha},\mu_{\alpha})$ of type $\mathcal{K}$, in the sense of Definition~\ref{defi:ds singular}, where $\mathcal{K}$ is given as in \eqref{regularity weight for non-regularized}. Moreover, for each $\alpha>0,$ the dissipative solution $(v_{\alpha},S_{\alpha},\varphi_{\alpha},\mu_{\alpha})$ satisfies
\begin{equation}
    \mu_{\alpha}=-\Delta\varphi_{\alpha}+W_{\mathrm{dw}}^{\prime}(\varphi_{\alpha})+W_{\mathrm{sg},\alpha}^{\prime}(\varphi_{\alpha}).
\end{equation}
Following the {arguments} of the proofs of  Lemma~\ref{est:alpha uniform} and Theorem~\ref{exitence evs double well}, we conclude the existence of a not {relabeled} subsequence and of  a limit quadruplet $(v,S,\varphi,\mu)$ such that
    \begin{align}
        v_{\alpha}&\rightharpoonup v\text{ in }L^{2}(0,T^{\prime};H_{0,\mathrm{div}}^{1}(\Omega)),
            \label{WK CONV DS:v}\\
        v_{\alpha}&\overset{*}{\rightharpoonup} v\text{ in }L^{\infty}(0,T^{\prime};L_{\mathrm{div}}^{2}(\Omega)),
            \\
        S_{\alpha}&\overset{*}{\rightharpoonup} S\text{ in } L^{\infty}(0,T^{\prime};L_{\mathrm{sym,Tr}}^{2}(\Omega)),
            \label{WK CONV DS:S}\\
        \varphi_{\alpha}&\rightharpoonup \varphi\text{ in }L^{2}(0,T^{\prime};H^{2}(\Omega)),
            \label{WK CONV DS:phase}\\
        \varphi_{\alpha}&\overset{*}{\rightharpoonup}\varphi\text{ in }L^{\infty}(0,T^{\prime};H^{1}(\Omega)),
            \\
        \mu_{\alpha}&\rightharpoonup\mu\text{ in }L^{2}(0,T^{\prime};H^{1}(\Omega)),
            \label{WK CONV DS:flux}\\
        v_{\alpha}&\to v\text{ in }L^{2}(0,T^{\prime};L^{2}(\Omega)),
            \label{STR CONV DS:v}\\
        \varphi_{\alpha}&\to\varphi\text{ in }L^{2}(0,T^{\prime};H^{1}(\Omega)),
            \label{STR CONV DS:S}\\
        \varphi_{\alpha}&\to\varphi\text{ a.e.\ in }\Omega\times[0,T)
            \label{AE CONV DS:phase}.
        \end{align}
Since $\abs{\varphi_{\alpha}}<1+\alpha$ a.e.\ in $\Omega\times[0,T)$ for each $\alpha>0$, we deduce that also $\abs{\varphi}\leq1$ a.e.\ in $\Omega\times[0,T)$.

In addition, we have
 \begin{equation}\label{WK CONV DS:subdifferential}
    W_{\mathrm{sg},\alpha}^{\prime}(\varphi_{\alpha})\rightharpoonup\beta\text{ in }L^{2}(0,T^{\prime};L^{2}(\Omega))\,,
\end{equation}
where $\beta\in\partial\mathcal{I}_{K}(\varphi)$ and $\beta(x,t)\in\partial{I}_{[-1,1]}(\varphi(x,t))$ for a.e.\  $(x,t)\in\Omega\times(0,T)$ by Remark~\ref{subdifferential integrand}. 

Let $\tilde{v}\in C_{0,\mathrm{div}}^{\infty}(\Omega\times[0,T))$, $\tilde{S}\in C_{0,\mathrm{sym,Tr}}^{\infty}(\Omega\times[0,T))$ and $\tilde{\varphi}\in C_{0}^{\infty}(\Omega\times[0,T))$ with $\abs{\tilde{\varphi}}\leq1$ be suitable test functions. We define $\tilde{\varphi}_{\alpha}$ as
\begin{equation}\label{STR CONV tilde varphi}
    \tilde{\varphi}_{\alpha}:=(1-\alpha^{\theta})\tilde{\varphi}
\end{equation}
for some fixed $0<\theta<\frac{1}{2}$. Notice that $\tilde{\varphi}_{\alpha}\in C_{0}^{\infty}(\Omega\times[0,T))$ and, since $\alpha\in(0,1),$ there also holds $\tilde{\varphi}_{\alpha}\in[-1+\alpha^{\theta},1-\alpha^{\theta}]\subseteq(-1-\alpha,1+\alpha)$ a.e. in $\Omega\times[0,T)$.  Hence, $(\tilde{v},\tilde{S},\tilde{\varphi}_{\alpha})$ are admissible test functions for Definition~\ref{defi:ds singular}.

{Additionally}  
\begin{equation*}
    \abs{\tilde{\varphi}-\tilde{\varphi}_{\alpha}}\leq \alpha^{\theta}\abs{\tilde{\varphi}}.
\end{equation*}
Therefore, with the choice $0<\theta<\frac{1}{2}$, we conclude that 
\begin{equation}\label{strong convergence of test function tilde varphi}
    \tilde{\varphi}_{\alpha}\to\tilde\varphi\text{ in }W^{k,p}(\Omega\times(0,T))
\end{equation}
for all $1\leq p\leq\infty$ and all $k\in\mathbb{N}$. 

Now, we would like to deduce a convergence result for the chemical potentials $(\tilde\mu_\alpha)_\alpha,$ where 
\begin{equation}
    \tilde{\mu}_{\alpha}=-\Delta\tilde{\varphi}_{\alpha}+W_{\mathrm{dw}}^{\prime}(\tilde\varphi_{\alpha})+W_{\mathrm{sg},\alpha}^{\prime}(\tilde\varphi_{\alpha}).
\end{equation}
For this, we in particular check the convergence of the derivatives of the logarithmic  potentials $(W_{\mathrm{sg},\alpha}(\tilde\varphi_\alpha))_\alpha$ up to order three.  
For the first derivatives, we have
\begin{equation}\label{uniform convergence of first derivative}
    \begin{aligned}
        \abs{ W_{\mathrm{sg},\alpha}^{\prime}(\tilde{\varphi}_{\alpha})}
        &=\abs{\alpha\ln(1+\alpha+\tilde{\varphi}_{\alpha})-\alpha\ln(1+\alpha-\tilde{\varphi}_{\alpha})}
        \\
        &\leq \alpha(C+\abs{\ln(\alpha+\alpha^{\theta})}) 
        \to0
    \end{aligned}
\end{equation}
uniformly on $\Omega\times(0,T)$.

Next, for the second derivatives, we derive
\begin{equation}\label{uniform convergence of second derivative}
    \begin{aligned}
        \abs{ W_{\mathrm{sg},\alpha}^{\prime\prime}(\tilde{\varphi_{\alpha}})}
        =&\abs{\frac{\alpha}{1+\alpha+\tilde{\varphi}_{\alpha}}+\frac{\alpha}{1+\alpha-\tilde{\varphi}_{\alpha}}}
        \\
        \leq& C\frac{\alpha}{\alpha+\alpha^{\theta}}
        \to0.
    \end{aligned}
\end{equation}
uniformly on $\Omega\times(0,T)$.

In addition, for the third derivative, we have
\begin{equation}\label{uniform convergence of third derivative}
    \begin{aligned}
        \abs{ W_{\mathrm{sg},\alpha}^{\prime\prime\prime}(\tilde{\varphi}_{\alpha})}
        =&\abs{\frac{\alpha}{(1+\alpha-\tilde{\varphi}_{\alpha})^{2}}-\frac{\alpha}{(1+\alpha+\tilde{\varphi}_{\alpha})^{2}}}
        \\
        \leq& C\frac{\alpha}{(\alpha+\alpha^{\theta})^{2}}\to0
    \end{aligned}
\end{equation}
uniformly on $\Omega\times(0,T)$, since $0<\theta<\frac{1}{2}$. Therefore, also in view of the convergence results~\eqref{STR CONV tilde varphi}, we now conclude  
\begin{equation}\label{strong convergence of test function tilde mu}
\begin{aligned}
    \tilde{\mu}_{\alpha}=-\Delta\tilde{\varphi}_{\alpha}+W_{\mathrm{dw}}^{\prime}(\tilde{\varphi}_{\alpha})+ W_{\mathrm{sg},\alpha}^{\prime}(\tilde{\varphi}_{\alpha})
    \to-\Delta\tilde{\varphi}+W_{\mathrm{dw}}^{\prime}(\tilde{\varphi})=\tilde{\mu}\text{ in }W^{2,p}(\Omega\times(0,T))
\end{aligned}
\end{equation}
for all $1\leq p<\infty$.

Following the arguments of the proof of~\cite[Theorem 5.1]{arxiv:2509.25508}, thanks to the convergence results \eqref{WK CONV DS:v}-\eqref{AE CONV DS:phase}, \eqref{strong convergence of test function tilde varphi}, and \eqref{strong convergence of test function tilde mu}, the limit passage $\alpha \searrow 0$ can be performed in all terms of \eqref{EVS inequality alpha}, except for {the term}  
\begin{equation}\label{term with tested chemi}
    \divge{m(\tilde{\varphi}_{\alpha})\nabla\tilde{\mu}_{\alpha}}\left(-\Delta(\varphi_{\alpha}-\tilde{\varphi}_{\alpha})
    +W_{\mathrm{dw}}^{\prime\prime}(\tilde{\varphi}_{\alpha})(\varphi_{\alpha}-\tilde{\varphi}_{\alpha})
    +W_{\mathrm{sg},\alpha}^{\prime\prime}(\tilde{\varphi}_{\alpha})(\varphi_{\alpha}-\tilde{\varphi}_{\alpha})\right)\,.
\end{equation}
Instead, to pass to the limit in this term in \eqref{term with tested chemi}, notice that the convergence results \eqref{WK CONV DS:phase} and \eqref{strong convergence of test function tilde varphi}-\eqref{strong convergence of test function tilde mu} give 
\begin{equation*}
\begin{aligned}
    &\int_{0}^{T^{\prime}}\phi\int_{\Omega}
    \divge{m(\tilde{\varphi}_{\alpha})\nabla\tilde{\mu}_{\alpha}}\left(-\Delta(\varphi_{\alpha}-\tilde{\varphi}_{\alpha})
    +W_{\mathrm{dw}}^{\prime\prime}(\tilde{\varphi}_{\alpha})(\varphi_{\alpha}-\tilde{\varphi}_{\alpha})
    +W_{\mathrm{sg},\alpha}^{\prime\prime}(\tilde{\varphi}_{\alpha})(\varphi_{\alpha}-\tilde{\varphi}_{\alpha})\right)
    \dd{x}\dd{t}\\
    \to&
    \int_{0}^{T^{\prime}}\phi\int_{\Omega}
    \divge{m(\tilde{\varphi})\nabla\tilde{\mu}}\left(-\Delta(\varphi-\tilde{\varphi})
    +W_{\mathrm{dw}}^{\prime\prime}(\tilde{\varphi})(\varphi-\tilde{\varphi})
    \right)
    \dd{x}\dd{t}
\end{aligned}
\end{equation*}
for all test {functions} $\tilde{\varphi}\in C_{0}^{\infty}(\Omega\times[0,T))$ with $\abs{\tilde{\varphi}}\leq1$ also in view of {Assumptions}~\ref{ASM:mobility}-\ref{ASM:Double well potential}{.}
\end{proof}
\begin{remark}\label{rm: on ds with regularized system}
    In the case of the regularized two-phase system~\eqref{sys:two phase} with $\gamma>0$, one can show the existence of dissipative solutions with a double-obstacle potential by following the same arguments as for the proofs of Lemma~\ref{est:alpha uniform} and Theorem~\ref{exitence ds double well}.
\end{remark}
\begin{remark}
\label{Rem-final-conv}
Combining {Theorem}~\ref{exitence evs double well}, Remark~\ref{rm: on evs with regularized system}, Theorem~\ref{exitence ds double well}, and Remark~\ref{rm: on ds with regularized system} together with {Propositions}~\ref{relative energy-dissipation inequality for EVS} {and~\ref{relative energy-dissipation inequality for double well}}, we arrive at the following commutative diagram: Both for the regularized two-phase system~\eqref{sys:two phase} with $\gamma>0$ and for the {nonregularized} two-phase systems~\eqref{sys:two phase} with $\gamma=0$, we have the following convergence results for and between the setting of energy-variational solutions (E.V.S) and the setting of dissipative solutions (D.S.):  
\begin{center}
\begin{tikzcd}[column sep=6em, row sep=6em]
\text{(E.V.S.)$_\alpha$} \arrow[r, "\text{Proposition}~\ref{relative energy-dissipation inequality for EVS}"] \arrow[d, "{\alpha \searrow 0}"] & \text{(D.S.)$_\alpha$} \arrow[d,"{\alpha \searrow 0}"] \\
\text{(E.V.S.)} \arrow[r, "\text{Proposition}~\ref{relative energy-dissipation inequality for double well}"] & \text{(D.S.)} 
\end{tikzcd}
\end{center}
We note that the energy-variational solution is  simpler than the dissipative one as it  incorporates less terms and only requires weaker regularity assumptions of the phase-field potential. This can for instance be seen in the limit passage $\alpha\searrow 0$, which requires significantly more care in the dissipative solution framework. 
\end{remark}

\paragraph{Acknowledgements: }{Cheng and Thomas} are grateful for the financial support by Deutsche Forschungsgemeinschaft (DFG) within CRC 1114 \emph{Scaling Cascades in Complex Systems}, Project-Number 235221301, Project B09 \emph{Materials with Discontinuities on Many Scales} as well as Project C09 \emph{Dynamics of Rock Dehydration on Multiple Scales}. {Lasarzik} is grateful for the hospitality of the \textit{Erwin Schrödinger International Institute for Mathematics and Physics (ESI)} during the thematic program on \textit{Free Boundary Problems}.

\bibliographystyle{alpha}  
\addcontentsline{toc}{section}{References}
\bibliography{Refs}

\end{document}